\documentclass[10pt]{article}

\title{Well-posedness, magnetic helicity conservation, inviscid limit and asymptotic stability for the generalized Navier-Stokes-Maxwell equations with the Hall effect}
\author{Kyungkeun Kang\thanks{
		Department of Mathematics,
		Yonsei University, 03722 Seoul, Republic of Korea. E-mail address: \url{kkang@yonsei.ac.kr}}
	\and
	Jihoon Lee\thanks{
		Department of Mathematics, 
		Chung-Ang University, 
		06974 Seoul, Republic of Korea. E-mail address: \url{jhleepde@cau.ac.kr}}
	\and
	Dinh Duong Nguyen\thanks{
		Department of Mathematics,
		Yonsei University, 03722 Seoul, Republic of Korea
		and Department of Mathematics, 
		Chung-Ang University, 
		06974 Seoul, Republic of Korea. E-mail address:  \url{dinhduongnguyen.math@gmail.com}
	}
}

\usepackage{mathtools}
\usepackage{amssymb,bbm,amsmath,amsfonts,amsthm}
\usepackage{vmargin} 
\usepackage[linktocpage,colorlinks=true,raiselinks]{hyperref}
\hypersetup{urlcolor=black, citecolor=red, linkcolor=blue}
\usepackage{cleveref}
\usepackage{tkz-euclide}
\usepackage{multirow}
\usepackage{pgfplots}
\pgfplotsset{compat=newest}
\usepgfplotslibrary{fillbetween}
\usepackage{capt-of}

\usepackage{geometry}
\numberwithin{equation}{section}

\DeclareMathOperator*{\esssup}{ess\,sup}

\newtheorem{theorem}{Theorem}[section]

\newtheorem{lemma}{Lemma}[section]

\theoremstyle{definition}
\newtheorem{remark}{Remark}[section]
\newenvironment{AMS}{}{}
\newenvironment{keywords}{}{}

\begin{document}
	\maketitle
	
	\begin{abstract}
		This paper is devoted to studying the well-posedness, (conditional) conservation of magnetic helicity,
		inviscid limit and asymptotic stability of the generalized Navier-Stokes-Maxwell equations (NSM) under the Hall effect in two and three dimensions. More precisely, in the viscous case we prove the global well-posedness of NSM for small initial data, which allows us to establish a connection with either the Hall-magnetohydrodynamics (H-MHD) system as the speed of light tends to infinity or NSM without the Hall coefficient as this constant goes to zero. In addition, in the inviscid case the local well-posedness of NSM is also obtained for possibly large initial data. Moreover, under suitable conditions on the initial data and additional assumptions of solutions to NSM in three dimensions, the magnetic helicity is conserved as the electric conductivity goes to infinity. It is different to the case of the fractional H-MHD with critical fractional Laplacian exponents for both the velocity and magnetic fields, where the conservation of magnetic helicity can be provided for smooth initial data without any further conditions on the solution. 
		Furthermore, the asymptotic stability of NSM around a constant magnetic field is established in the case of having a velocity damping term.
	\end{abstract}
	
	\begin{keywords}
		\textbf{Keywords:} 
		Navier-Stokes-Maxwell, Hall-magnetohydrodynamics, well-posedness, magnetic helicity conservation, 
		inviscid limit, 
		asymptotic stability. 
	\end{keywords}
	
	\begin{AMS}
		\textbf{Mathematics Subject Classification:} 35Q35, 35Q60, 76D03, 76W05, 78A25.
	\end{AMS}
	
	\allowdisplaybreaks 
	
	%
	\section{Introduction}
	%
	
	%
	\subsection{The systems}
	%
	
	Let us consider the following generalized incompressible Navier-Stokes-Maxwell equations\footnote{Here, the usual fractional Laplacian operator is defined in terms of Fourier transform, i.e., for $\alpha \in \mathbb{R}$
		\begin{equation*}
			\mathcal{F}((-\Delta)^\alpha(f))(\xi) := |\xi|^{2\alpha}\mathcal{F}(f)(\xi) \qquad \text{where}\qquad \mathcal{F}(f)(\xi) := \int_{\mathbb{R}^d} \exp\{-i\xi\cdot x\} f(x) \,dx \quad \text{for } \xi \in \mathbb{R}^d.
		\end{equation*}
		In the case of $\alpha = 0$, we use the standard convention that $(-\Delta)^0$ is the identity operator.}
	\begin{equation} \label{NSM} \tag{NSM}
		\left\{
		\begin{aligned}
			\partial_t v + v \cdot \nabla v + \nabla \pi &= -\nu(-\Delta)^\alpha v + j \times B, 
			\\
			\frac{1}{c}\partial_t E - \nabla \times B &= -j,
			\\
			\frac{1}{c}\partial_t B + \nabla \times E &= 0,
			\\
			\sigma(cE + v \times B) &= j,
			\\
			\textnormal{div}\, v = \textnormal{div}\, B &= 0,
		\end{aligned}
		\right.
	\end{equation}
	where $\alpha \geq 0$, for $d \in \{2,3\}$, $(v,E,B,j) : \mathbb{R}^d \times (0,\infty) \to \mathbb{R}^3$ and $\pi : \mathbb{R}^d \times (0,\infty) \to \mathbb{R}$ are the velocity, electric, magnetic and electric current fields,  and scalar pressure of the fluid, respectively. The positive constants $\nu,\sigma$ and $c$ denote in order the viscosity, electric conductivity and  speed of light. We will denote the initial data by $(v,E,B)_{|_{t=0}} = (v_0,E_0,B_0)$. We note that the case $d = 2$ is also known as the $2\frac{1}{2}$-dimensional version. Let us also quickly recall the standard meaning of the above system. In \eqref{NSM}, through the Lorentz force $j \times B$ (under quasi-neutrality assumptions) and the electric current field $j$ the Navier-Stokes equations (the first line) are coupled to the Maxwell system, where the latter one consists of the Ampere's equations + Maxwell's correction (the second line) and the Faraday's law (the third line). In addition, the fourth line is the usual Ohm's law and the last one stands for the incompressiblity of the velocity and magnetic fields. It can be seen that if the term $\frac{1}{c}\partial_t E$ is neglected formally for either large $c$ or time-independent $E$, then \eqref{NSM} reduces to the usual $2\frac{1}{2}$-dimensional fractional magnetohydrodynamics (MHD) equations, i.e., \eqref{HMHD} with $\kappa = 0$ and $\beta = 1$ (for more physical introduction to the magnetohydrodynamics, see \cite{Biskamp_1993,Davidson_2017}). Therefore,  \eqref{NSM} with $\alpha = 1$ is also known as the full MHD system.
	
	In fact, by ignoring thermal effects, \eqref{NSM} with $\alpha = 1$ can be derived from kinetic equations (see \cite{Jang-Masmoudi_2012}). By considering a solenoidal Ohm’s law\footnote{In this case, $j =  \sigma(-\nabla \bar{\pi} + cE + v \times B)$ with $\textnormal{div}\, j = \textnormal{div}\, E = 0$ and for some additional electromagnetic pressure $\bar{\pi}$, see \eqref{NSM-SO}.} instead, it also can be formally obtained as a limiting system of a two-fluid incompressible Navier–Stokes–Maxwell system by taking the momentum transfer coefficient $\epsilon > 0$ tends to zero (see \cite{Arsenio-Ibrahim-Masmoudi_2015}). More precisely, if $v^+$ and $v^-$ denote the cations and anions velocities, respectively, with the same viscosity $\mu > 0$ and the corresponding thermal pressures $\pi^+$ and $\pi^-$, then the scaled two-fluid incompressible Navier–Stokes–Maxwell equations were proposed in \cite{Giga-Yoshida_1984} and will be written in the following form\footnote{In fact, the authors in \cite{Giga-Yoshida_1984} suggested a more general model with different coefficients appearing in the equations.} (we use the same notation for the electric and magnetic fields as previously)
	\begin{equation*} \tag{2-NSM} \label{2NSM}
		\left\{
		\begin{aligned}
			\partial_t v^+ + v^+ \cdot \nabla v^+ +  \frac{1}{2\sigma\epsilon^2}(v^+-v^-) &= \mu \Delta v^+ - \nabla \pi^+ + \frac{1}{\epsilon}(cE + v^+ \times B),
			\\
			\partial_t v^- + v^- \cdot \nabla v^- - \frac{1}{2\sigma\epsilon^2}(v^+-v^-) &= \mu \Delta v^- - \nabla \pi^- - \frac{1}{\epsilon}(cE + v^- \times B),
			\\
			\frac{1}{c}\partial_t E - \nabla \times B &= -\frac{1}{2\epsilon}(v^+-v^-),
			\\
			\frac{1}{c}\partial_t B + \nabla \times E &= 0,
			\\
			\textnormal{div}\, v^+ = \textnormal{div}\, v^- = \textnormal{div}\, E &= \textnormal{div}\, B = 0,
		\end{aligned}
		\right.
	\end{equation*}
	which models the motion of a plasma of positively (cations) and negatively (anions) charged
	particles under the assumption of equal masses. In the above system, the condition $\textnormal{div}\, E = 0$, which is known as a degenerate Gauss’s law (see \cite{Arsenio-Ibrahim-Masmoudi_2015}) and follows from the charge neutrality and the incompressibility of the plasma (see \cite{Giga-Yoshida_1984}), and the third term on the right-hand side of the second equation presents the momentum transfer between the two fluids. The existence and uniqueness of global energy solutions to \eqref{2NSM} (for more general coefficients) have recently obtained in \cite{Giga-Ibrahim-Shen-Yoneda_2018} in two dimensions. In the three-dimensional case, they also showed the existence of global energy solutions and local well-posedness (LWP) of \eqref{2NSM} for initial data $(v^{\pm}_0,E_0,B_0) \in H^\frac{1}{2} \times L^2 \times L^2$, and this local solution can be globally extended for small $v^{\pm}_0$ in the $\dot{H}^\frac{1}{2}$ norm. It can be seen that the energy equality to \eqref{2NSM} formally reads\footnote{Here, $\|(f_1,...,f_n)\|^m_X := \sum^n_{i=1}\|f_i\|^m_X$ for $n,m \in \mathbb{N}$ and some functional space $X$.}
	\begin{equation*}
		\frac{1}{2} \frac{d}{dt}\left\|\left(\frac{v^{\pm}}{\sqrt{2}},E,B\right)\right\|^2_{L^2} + \mu\left\|\frac{1}{\sqrt{2}} \nabla v^{\pm}\right\|^2_{L^2} + \frac{1}{\sigma}\left\|\frac{1}{2\epsilon}(v^+-v^-)\right\|^2_{L^2} = 0,
	\end{equation*}
	which suggests us to consider a system, which is satisfied by the following quantities
	\begin{equation*}
		k := \frac{1}{2\epsilon}(v^+-v^-),\quad u := \frac{1}{2}(v^++v^-),\quad p_1 := \frac{1}{2}(\pi^++\pi^-) \quad \text{and} \quad  p_2 := \frac{\epsilon}{2}(\pi^+-\pi^-)
	\end{equation*}
	and is given by rewriting \eqref{2NSM} as follows
	\begin{equation*} 
		\left\{
		\begin{aligned}
			\partial_t u + u \cdot \nabla u + \epsilon^2 k \cdot \nabla k  &= \mu\Delta u - \nabla p_1 + k \times B, 
			\\
			\epsilon^2(\partial_t k + u \cdot \nabla k + k \cdot \nabla u) + \frac{1}{\sigma} k &= \epsilon^2 \mu \Delta k - \nabla p_2 + cE + u \times B, 
			\\	
			\frac{1}{c}\partial_t E - \nabla \times B &= -k,
			\\
			\frac{1}{c}\partial_t B + \nabla \times E &= 0,
			\\
			\textnormal{div}\, u = \textnormal{div}\, k = \textnormal{div}\, E &= \textnormal{div}\, B = 0.
		\end{aligned}
		\right.
	\end{equation*}
	As $\epsilon \to 0$, the above system formally converges to \eqref{NSM} with $\alpha = 1$ and a solenoidal Ohm’s law instead of the usual one (see \cite{Arsenio-Ibrahim-Masmoudi_2015,Arsenio-SaintRaymond_2016}), i.e., the following system (for simplicity, we will replace $u,k,p_1,p_2$ and $\mu$ by $v,j,\pi_1,\pi_2$ and $\nu$, respectively), which will be written in a more general form as follows
	\begin{equation*} \label{NSM-SO} \tag{NSM-SO}
		\left\{
		\begin{aligned}
			\partial_t v + v \cdot \nabla v + \nabla \pi_1 &= -\nu(-\Delta)^\alpha v + j \times B, 
			\\	
			\frac{1}{c}\partial_t E - \nabla \times B &= -j,
			\\
			\frac{1}{c}\partial_t B + \nabla \times E &= 0,
			\\
			\sigma(-\nabla \pi_2 + cE + v \times B) &= j, 
			\\
			\textnormal{div}\, v = \textnormal{div}\, j = \textnormal{div}\, E &= \textnormal{div}\, B = 0,
		\end{aligned}
		\right.
	\end{equation*}
	that shares a similar structure and mathematical difficulties to those of \eqref{NSM}. In fact, we will list known results to \eqref{NSM} and it is possible to obtain similar ones to \eqref{NSM-SO}. The rigorous proof of the convergence from \eqref{2NSM} to \eqref{NSM-SO} as $\epsilon \to 0$ does not seem to be known for $L^2$ initial data. In \cite{Arsenio-Ibrahim-Masmoudi_2015}, the authors established the limit as first $c \to \infty$ and then $\epsilon \to 0$, where \eqref{2NSM} converges weakly to the standard $2\frac{1}{2}$-dimensional MHD system, i.e., \eqref{HMHD} with $\alpha = 1$ and $\kappa = 0$. They also pointed out that the same result also holds in the case $c \to \infty$ and $\epsilon \to 0$ at the same time, but with additional conditions on the relation between $\epsilon$ and $c$ in which $\epsilon$ is considered as a function of  $c$ satisfying further assumptions. However, the other order of taking limit has not been confirmed yet, i.e., the limit as $\epsilon \to 0$ first and then $c \to \infty$, where the limiting system is the same as the previous case. In addition, it is also very interested and much more complicated to consider  either \eqref{NSM} or \eqref{NSM-SO} with a generalized Ohm's law, which can be derived from either the two-fluid Navier-Stokes-Maxwell equations or kinetic models with different masses (see \cite{Acheritogaray-Degond-Frouvelle-Liu_2011,Jang-Masmoudi_2012,Peng-Wang-Xu_2022}) for $\alpha = 1$, in particular, the new system takes a more general form as follows 
	\begin{equation*} \label{NSM-GO} \tag{NSM-GO}
		\left\{
		\begin{aligned}
			\partial_t v + v \cdot \nabla v + \nabla \pi_1 &= -\nu(-\Delta)^\alpha v + j \times B, 
			\\	
			\frac{1}{c}\partial_t E - \nabla \times B &= -j,
			\\
			\frac{1}{c}\partial_t B + \nabla \times E &= 0,
			\\
			\sigma(-\nabla \pi_2 + cE + v \times B)  &= j + \kappa j \times B, 
			\\
			\textnormal{div}\, v = \textnormal{div}\, j = \textnormal{div}\, E &= \textnormal{div}\, B = 0,
		\end{aligned}
		\right.
	\end{equation*}
	which takes into account of the Hall effect for some nonnegative constant $\kappa$. This new constant is corresponding to the magnitude of the Hall effect
	compared to the typical fluid length scale. Furthermore, by taking the limit as $c \to \infty$ formally or ignoring the term $\frac{1}{c}\partial_t E$ (for example, $E$ is time-independent), \eqref{NSM-GO} reduces to the standard fractional Hall-magnetohydrodynamics equations with $\beta = 1$, i.e., in general for $\beta \geq 0$
	\begin{equation} \label{HMHD} \tag{H-MHD}
		\left\{
		\begin{aligned}
			\partial_t v + v \cdot \nabla v + \nabla \pi &= -\nu(-\Delta)^\alpha v + (\nabla \times B) \times B, 
			\\
			\partial_t B - \nabla \times (v \times B) &= -\frac{1}{\sigma}(-\Delta)^\beta B - \frac{\kappa}{\sigma} \nabla \times ((\nabla \times B) \times B),
			\\
			\textnormal{div}\, v = \textnormal{div}\, B &= 0.
		\end{aligned}
		\right.
	\end{equation}
	Indeed, the Hall term in \eqref{HMHD} plays an important role in magnetic reconnection, which can not be explained by using \eqref{MHD} only, and is derived from either two-fluid models or kinetic equations in \cite{Acheritogaray-Degond-Frouvelle-Liu_2011}. The systematical study of the above equations was initiated in \cite{Lighthill_1960} long time ago. However, even in the case that $d = 2$ (also known as the $2\frac{1}{2}$-dimensional version) and $\alpha = \beta = 1$, the global regularity issue of \eqref{HMHD} has not been fully established for general initial data. In this case, the existence of global energy solutions has been provided in \cite{Chae-Degond-Liu_2014} in both two and three dimensions, but it is not the case for \eqref{NSM}, \eqref{NSM-SO} and \eqref{NSM-GO} as mentioned before. In addition, by using the convex integration framework, the author in \cite{Dai_2021} proved the nonuniqueness of weak solutions in the Leray-Hopf class for $d = 3$. In the case of without the resistivity, singularity formation and ill-posedness results were also obtained for instance in \cite{Chae-Weng_2016} and \cite{Jeong-Oh_2022, Jeong-Oh_2024}, respectively. Furthermore, global small initial data solutions in both cases $d = 2$ and $d = 3$ have been provided in \cite{Bae-Kang_2023,Chae-Degond-Liu_2014,Chae-Lee_2014,Chang-Lkhagvasuren-Yang_2023,Danchin-Tan_2021,Danchin-Tan_2022,Liu-Tan_2021,Tan_2022,Wan_2015,Wan-Zhou_2015,Wan-Zhou_2019,Wan-Zhou_20191,Ye_2022}. Especially, global smooth solutions for a specific class of large initial data have been found in \cite{Zhang_2019}. In the critical fractional exponent case, where either $(\alpha,\beta) = (1,\frac{3}{2})$ in two dimensions (see \cite[Proposition 1.1]{KLN_2024_2}) or $(\alpha,\beta) = (\frac{5}{4},\frac{7}{4})$ in three dimensions (see \cite{Wan_2015}), the global well-posedness (GWP) of \eqref{HMHD} can be obtained. See also the LWP of \eqref{HMHD} for possibly large initial data in \cite{Dai_2020,Dai_2021_2}. Further properties of weak solutions are also investigated in \cite{Dumas-Sueur_2014}. In the stationary case, regularity and partial regularity of  weak solutions can be found in \cite{Chae-Wolf_2017} and in \cite{Chae-Wolf_2015} in the two and three-dimensional cases, respectively. As mentioned previously, \eqref{HMHD} can be obtained formally from \eqref{NSM-GO}.  Therefore, it is reasonable to consider conditional GWP results for \eqref{NSM-GO}, for instance, under smallness assumptions of initial data. 
	
	Finally, it is also convenient to write down the standard fractional MHD system as follows
	\begin{equation} \label{MHD} \tag{MHD}
		\left\{
		\begin{aligned}
			\partial_t v + v \cdot \nabla v + \nabla \pi &= -\nu (-\Delta)^\alpha v + B \cdot \nabla B, 
			\\
			\partial_t B + v \cdot \nabla B &= -\frac{1}{\sigma} (-\Delta)^\beta B + B \cdot \nabla v,
			\\
			\textnormal{div}\, v = \textnormal{div}\, B &= 0,
		\end{aligned}
		\right.
	\end{equation}
	where $(v,B) : \mathbb{R}^d \times (0,\infty) \to \mathbb{R}^d$ and $\pi : \mathbb{R}^d \times (0,\infty) \to \mathbb{R}$ for $d \in \{2,3\}$. In three dimensions, it is well-known that by using some vector identities,  \eqref{HMHD} with $\kappa = 0$ and \eqref{MHD} are formally equivalent to each other up to a modified pressure.
	
	%
	\subsection{The state of the art}
	%
	
	\textbf{A. The case of \eqref{NSM} with $d = 2$.} Let us give a quick review on the study of \eqref{NSM} in two dimensions with $\alpha = 1$. Formally, its energy balance is given by (the same for \eqref{NSM-SO} and \eqref{NSM-GO} in both cases $d = 2$ and $d = 3$)
	\begin{equation*}
		\frac{1}{2} \frac{d}{dt}\|(v,E,B)\|^2_{L^2} + \nu\|\nabla v\|^2_{L^2} + \frac{1}{\sigma}\|j\|^2_{L^2} = 0.
	\end{equation*}
	Thus, similar to in the case of the usual Navier-Stokes equations, we could expect the existence of global energy solutions (see \cite{Leray_1933,Leray_1934}). However, it seems that this energy equality is not enough to obtain the existence of $L^2$ weak solutions, which is different to that of \eqref{2NSM} in the two and three-dimentional cases as mentioned previously. The main difficulty is the lack of compactness, due to the hyperbolicity 
	of the Maxwell equations, which is needed to pass to the limit as $n \to \infty$ of the term $j^n \times B^n$, especially for the one $E^n \times B^n$, where $n$ is the regularization parameter of a standard approximate system to \eqref{NSM}. Therefore, higher regular data should be considered on the GWP issue. The first GWP result to \eqref{NSM} was obtained in \cite{Masmoudi_2010} in the case where
	\begin{equation*}
		(v_0,E_0,B_0) \in L^2 \times H^s \times H^s \quad \text{ for } s \in (0,1).
	\end{equation*}
	In addition, higher regular estimates are also provided in \cite{Masmoudi_2010} in the case where $(v_0,E_0,B_0) \in H^\delta \times H^s \times H^s$ for $\delta \geq 0$, $s \geq 1$, $s - 2 < \delta \leq s$, see also \cite{Kang-Lee_2013} for another proof and \cite{Fan-Ozawa_2020} for the case of bounded domains. The GWP is also obtained in \cite{Ibrahim-Keraani_2012} for small initial  data satisfying\footnote{For the definitions of the spaces $\dot{B}^0_{2,1}$ and $L^2_{\textnormal{log}}$, see \cite{Bahouri_2019,Bahouri-Chemin-Danchin_2011,Ibrahim-Keraani_2012}.} 
	\begin{equation*}
		(v_0,E_0,B_0) \in \dot{B}^0_{2,1} \times L^2_{\textnormal{log}} \times L^2_{\textnormal{log}},
	\end{equation*}
	where we have the following relations $\dot{B}^0_{2,1} \subset L^2$ and $\cup_{s>0} H^s \subset L^2_{\textnormal{log}} \subset L^2$. However, the LWP has not been constructed for the above arbitrary large initial data. After that, the authors in \cite{Germain-Ibrahim-Masmoudi_2014} have been considered mild solutions to \eqref{NSM} and they obtained the LWP 
	for possibly large initial data and the GWP for small initial data under the assumption
	\begin{equation*}
		(v_0,E_0,B_0) \in L^2 \times L^2_{\textnormal{log}} \times L^2_{\textnormal{log}}.
	\end{equation*}
	Here, in the two previous results,  in order to estimate the term $E \times B$ coming from $j \times B$ in some homogeneous Besov spaces, the authors used the paraproduct estimate \cite[Corollary 2.55]{Bahouri-Chemin-Danchin_2011} 
	and it is critical in two dimentions, thus the extra logarithmic regularity of $(E_0,B_0)$ is needed. Recently, the authors in \cite{Arsenio-Gallagher_2020} revisited \eqref{NSM} in the case where $(v_0,E_0,B_0) \in L^2 \times H^s \times H^s$ for $s \in (0,1)$, as considered previously in \cite{Masmoudi_2010}, with providing further improvements, which include some $c$-independent estimates of $(v,E,B)$. That allowed them to investigate the asymptotic behavior as $c \to \infty$ by proving the convergence of solutions to \eqref{NSM} to that of the standard $2\frac{1}{2}$-dimensional MHD equations, i.e., \eqref{HMHD} with $\alpha = \beta = 1$ and $\kappa = 0$, in the sense of distributions. Recently, the authors in \cite{KLN_2024_2} provided the GWP of \eqref{NSM} for a larger class of initial data for $\alpha = 1$ and the stability around a constant magnetic field of \eqref{NSM} for $\alpha = 0$.

	\textbf{B. The case of \eqref{NSM} with $d = 3$.} As mentioned previously, the existence of energy solutions is unknown so far. We will shortly recall some results to \eqref{NSM} in the three-dimensional case with $\alpha = 1$. One of the first results was given in \cite{Ibrahim-Keraani_2012}, where the authors constructed global small solutions with initial data
	\begin{equation*}
		(v_0,E_0,B_0) \in \dot{B}^\frac{1}{2}_{2,1} \times \dot{H}^\frac{1}{2} \times \dot{H}^\frac{1}{2}.
	\end{equation*}
	For large initial data in some $\ell^1$ weighted space in Fourier side, the authors in \cite{Ibrahim-Yoneda_2012} have been provided the local in time existence of mild solutions. Moreover, by using the fact that the damped-wave operator does not have any smoothing
	effect, they also showed these local solutions lost regularity in some finite time. Later on, the above result in \cite{Ibrahim-Keraani_2012} was extended in \cite{Germain-Ibrahim-Masmoudi_2014} in which either local large initial data solutions or global small intial data ones was provided for initial data in the following space 
	\begin{align*}
		(v_0,E_0,B_0) \in \dot{H}^\frac{1}{2} \times \dot{H}^\frac{1}{2} \times \dot{H}^\frac{1}{2}.
	\end{align*}
	Recently,  the existence of weak solutions was built in \cite{Arsenio-Gallagher_2020} for the initial data in the form of
	\begin{equation*}
		(v_0,E_0,B_0) \in L^2 \times H^s \times H^s \qquad \text{for} \quad s \in \left[\frac{1}{2},\frac{3}{2}\right),
	\end{equation*} 
	under the smallness assumption of $(E_0,B_0)$ in the $\dot{H}^s$ norm (the smallness assumption is related to only the $L^2$ norm of $(v_0,E_0,B_0)$ and $\dot{H}^s$ norm of $(E_0,B_0)$). 
	
	We also note that time-periodic small solutions and their asymptotic stability were investiagted in \cite{Ibrahim-Lemarie-Rieusset-Masmoudi_2018}. For further results to \eqref{NSM} (and also to \eqref{NSM-SO}) such as GWP for small data and LWP for possibly large data, loss of regularity, asymptotic behaviors, existence of global weak solutions with small data, global regularity criteria, time periodic solutions and so on, we refer the reader to \cite{Arsenio-Gallagher_2020,Arsenio-Ibrahim-Masmoudi_2015,Germain-Ibrahim-Masmoudi_2014,Ibrahim-Keraani_2012,Ibrahim-Lemarie-Rieusset-Masmoudi_2018,Ibrahim-Yoneda_2012,Jiang-Luo-Tang_2020,Kang-Lee_2013,Peng-Wang-Xu_2022,Wen-Ye_2020,Yue-Zhong_2016}. Finally, the authors in \cite{KLN_2024_2} have recently provided the GWP of \eqref{NSM} in the case of $\alpha = \frac{3}{2}$ and the stability of \eqref{NSM} near a constant magnetic field in the case of $\alpha = 0$.

	\textbf{C. The case of \eqref{NSM-GO} with new issues.} 
	In this part, we quickly focus on some new difficult points of \eqref{NSM-GO} compared to either \eqref{NSM} or \eqref{NSM-SO} (see more details in Section \ref{sec:stra} below). Mathematically, \eqref{NSM} is similar to \eqref{NSM-SO} as mentioned previously, in the sense that all mentioned results of \eqref{NSM} in Subsections 1.2-A and B can also be obtained for \eqref{NSM-SO} with using similar strategies of proof. In addition, \eqref{NSM-GO} reduces formally to \eqref{NSM-SO} in the case of $\kappa = 0$. Therefore, it is more suitable to compare \eqref{NSM-SO} and \eqref{NSM-GO}.
	
	\begin{enumerate}
		\item (\textbf{Compared to \eqref{HMHD}}) As mentioned before, if either $\partial_t E = 0$ or taking $c \to \infty$ and $\alpha = 1$ then formally \eqref{NSM-GO} reduces to \eqref{HMHD} with $\alpha = \beta = 1$. Thus, the system is more complex that \eqref{HMHD} in which the GWP has not been known for arbitrary initial data, even in two dimensions. However, as in the cases of the usual Navier-Stokes and \eqref{MHD} equations in three dimensions, we can hope that this problem can be overcomed by assuming the smallness of initial data. That is the main reason why we will focus on obtaining the GWP of \eqref{NSM-GO} for small initial data. Note furthermore that we do not have the Laplacian of $B$ in \eqref{NSM-GO} and $j$ is not $\nabla \times B$ compared to those of in \eqref{HMHD}. These observations indicate that there might exist non standard difficulties in the mathematical analysis of the system and the proofs in the mentioned papers on \eqref{HMHD} seem can not be used here.  
		
		\item (\textbf{Compared to \eqref{NSM} and \eqref{NSM-SO}}) In addition, the electric current field $j$ in \eqref{NSM-GO}, which is given by a generalized Ohm's law and depends on $(v,E,B,\pi_2)$ in an implicit way. That point is different to those of in the cases of \eqref{NSM} and \eqref{NSM-SO}. It is not clear to us how $j$ can be expressed explicitly by $(v,E,B)$ only.  Indeed, it can be seen from \eqref{NSM-GO} that
		\begin{equation*}
			j = M^{-1} \sigma(-\nabla \pi_2 + cE + v \times B),
		\end{equation*}
		where 
		\begin{equation*}
			M = M(\kappa,B) := 
			\left(
			\begin{matrix}
				1 & \kappa B_3 & -\kappa B_2
				\\
				-\kappa B_3 & 1 & \kappa B_1
				\\
				\kappa B_2 & -\kappa B_1 & 1
			\end{matrix}
			\right)
			\quad \text{with} \quad \textnormal{det}(M) = \kappa^2 |B|^2 + 1.			
		\end{equation*}
		On the right-hand side of the previous form, we still have the appearance of $\pi_2$, which in fact implicitly depends on $j$. On the other hand, by using the Leray projection, it follows that 
		\begin{equation*}
			\mathbb{P}(Mj) = \sigma(cE + \mathbb{P}(v \times B)).
		\end{equation*}
		However, this formula seems not to be useful enough. Therefore, it seems to us that the techniques in \cite{Arsenio-Gallagher_2020,Germain-Ibrahim-Masmoudi_2014,Ibrahim-Keraani_2012,KLN_2024_2,Masmoudi_2010} can not be applied directly here, mainly due to the fact that these papers employed of the explicit form of $j$ in terms of $(v,E,B)$ during the proofs. Furthermore, we also need to find another way to approximate the system before adapting the ideas in \cite{Germain-Ibrahim-Masmoudi_2014,Ibrahim-Keraani_2012} in controlling $v$ in the $L^2_tL^\infty_x$ norm and a new idea in the bound of $(E,B)$ in the $L^2_t\dot{H}^1_x$ norm.
	\end{enumerate}
	
	It seems to us that there are only a few results on \eqref{NSM-GO} in two and three dimensions. First of all, as mentioned previously, the Hall effect has been derived in  \cite{Acheritogaray-Degond-Frouvelle-Liu_2011,Jang-Masmoudi_2012}. In addition, the authors in \cite{Peng-Wang-Xu_2022} investigated the convergence from \eqref{NSM-GO} to \eqref{HMHD} in the case $\alpha = \beta = 1$ under suitable assumptions on the existence of solutions and also on the property of these solutions. Thus, in the next subsection, we will provide the mathematical analysis to \eqref{NSM-GO}. More precisely, we will prove the GWP for small initial data (see Theorem \ref{theo_global_small}) and LWP for large initial data (see Theorem \ref{theo_local}). We also study the (conditional) conservation of magnetic helicity of \eqref{NSM-GO} and \eqref{HMHD} as the electric conductivity $\sigma$ goes to infinity (see Theorems \ref{theo-MH} and \ref{theo-C-HMHD}, respectively), and the stability and large-time behavior to \eqref{NSM-GO} as well (see Theorem \ref{theo_stability}).
	
	%
	\subsection{Main results}
	%
	
	For the reader's convenience, before going to the detailed statements, let us first summarize the main results in the present paper as follows:
	\begin{enumerate}
		\item[1.] The GWP of \eqref{NSM-GO} for $\nu > 0$, $\alpha = 1$ and $d \in \{2,3\}$ in Theorem \ref{theo_global_small};
		
		\item[2.] The LWP for $\nu = 0$ and the inviscid limit of \eqref{NSM-GO} for $d \in \{2,3\}$ in Theorem \ref{theo_local};
				
		\item[3.] The (conditional) conservation of magnetic helicity of \eqref{NSM-GO} with $\alpha = 1$ and \eqref{HMHD} with $\alpha = \frac{5}{4}$, $\beta = \frac{7}{4}$ as $\sigma \to \infty$ for $\nu > 0$ and $d = 3$ in Theorems \ref{theo-MH} and \ref{theo-C-HMHD};
		
		\item[4.] The asymptotic stability near a magnetohydrostatic equilibrium with a constant (or equivalently bounded) magnetic field of \eqref{NSM-GO} for $\alpha = 0$, $\nu > 0, \kappa > 0$ and $d \in \{2,3\}$ in Theorem \ref{theo_stability};
	\end{enumerate}
	which will be precisely presented in the following subsubsections, respectively.
	
	%
	\subsubsection{Global and local well-posedness}
	%
	
	In this part, we aim to provide the GWP to \eqref{NSM-GO} in the case $\alpha = 1$ for small initial data and the LWP to \eqref{NSM-GO} in the inviscid case for possibly large initial data, respectively.
	
	\begin{theorem}[\textnormal{Global well-posedness for small initial data}] \label{theo_global_small}
		Let $\alpha = 1$, $d \in \{2,3\}$, $c \geq 1$, $\kappa,\nu, \sigma > 0$ and  $\Gamma_0 := (v_0,E_0,B_0) \in X^s := (H^{s-1} \times H^s \times H^s)(\mathbb{R}^d)$ with $s \in \mathbb{R}, s > \frac{d}{2}$. The following statements hold.
		\begin{enumerate}
			\item[\textnormal{(}i\textnormal{)}] \textnormal{(Global well-posedness).} There exists a positive constant $C_1 = C_1(d,\kappa,\nu,\sigma,s)$ such that if
			\begin{equation} \label{small_data} \tag{A}
				C_1 \|\Gamma_0\|_{X^s} \leq 1
			\end{equation}
			then \eqref{NSM-GO} admits a unique global solution $\Gamma := (v,E,B)$ with
			\begin{align*}
				v &\in L^\infty(0,\infty;H^{s-1}) \cap L^2(0,\infty;L^\infty), \quad \nabla v \in L^2(0,\infty;H^{s-1}), 
				\\
				E &\in L^\infty(0,\infty;H^s) \cap L^2(0,\infty;\dot{H}^{s''}), \quad s'' \in [0,s],
				\\
				B &\in L^\infty(0,\infty;H^s) \cap L^2(0,\infty;\dot{H}^{s'}), \quad s' \in [1,s],
				\\
				j &\in L^2(0,\infty;H^s),
			\end{align*}
			satisfying for $t > 0$ and for some positive constant $C_2 = C_2(d,\kappa,\nu,\sigma,s)$
			\begin{align*}
				\|\Gamma(t)\|^2_{X^s} + \int^t_0 \|\nabla v\|^2_{H^{s-1}} + \|v\|^2_{L^\infty} 
				+ \|E\|^2_{\dot{H}^{s''}} + \|B\|^2_{\dot{H}^{s'}} 
				+ \|j\|^2_{H^s}  \,d\tau \leq C_2\|\Gamma_0\|^2_{X^s}.
			\end{align*}
			
			\item[\textnormal{(}ii\textnormal{)}] \textnormal{(The limit as $c \to \infty$: A connection with \eqref{HMHD}).} Let $c \geq 1$ and $(v^c_0,E^c_0,B^c_0) \in H^{s-1} \times H^s \times H^s$ with $\textnormal{div}\, v^c_0 = \textnormal{div}\, B^c_0 = \textnormal{div}\, E^c_0 = 0$. Assume that $(v^c_0,E^c_0,B^c_0)$ satisfies \eqref{small_data} uniformly in terms of $c$ and as $c \to \infty$
			\begin{equation*}
				(v^c_0,E^c_0,B^c_0) \rightharpoonup (\bar{v}_0,\bar{E}_0,\bar{B}_0) \qquad \text{in} \quad H^{s-1} \times H^s \times H^s
			\end{equation*}
			for some $(\bar{v}_0,\bar{E}_0,\bar{B}_0)$ with $\textnormal{div}\, \bar{v}_0 = \textnormal{div}\, \bar{B}_0 = \textnormal{div}\, \bar{E}_0 = 0$. Then, there exists a sequence of global solutions $(v^c,E^c,B^c)$ to \eqref{NSM-GO} with $\alpha = 1$ and $(v^c,E^c,B^c)_{|_{t=0}} = (v^c_0,E^c_0,B^c_0)$, which are given as in Part $(i)$. In addition,  up to an extraction of a subsequence, $(v^c,B^c)$ converges to $(v,B)$ in the sense of distributions as $c \to \infty$, where
			$(v,B)$ satisfies \eqref{HMHD} with $\alpha = \beta = 1$ and $(v,B)_{|_{t=0}} = (\bar{v}_0,\bar{B}_0)$.
			
			\item[\textnormal{(}iii\textnormal{)}] \textnormal{(The limit as $\kappa \to 0$: A connection with \eqref{NSM-SO}).} Let $\kappa \in (0,1)$, $\Gamma_0 := (v_0,E_0,B_0)$ satisfying $C'_1 \|\Gamma_0\|_{X^s} \leq 1$ for some positive constant $C'_1 = C'_1(d,\nu,\sigma,s) > C_1$ and $\Gamma^\kappa
			 := (v^\kappa,E^\kappa,B^\kappa)$ be the unique global solution of \eqref{NSM-GO} given as in Part $(i)$ with $\Gamma^\kappa_{|_{t=0}} = \Gamma_0$. Let $\Gamma := (v,E,B)$ be the unique global solution of \eqref{NSM-SO} with $\Gamma^\kappa_{|_{t=0}} = \Gamma_0$ can be obtained as in Part $(i)$ with $\kappa = 0$. Then, for any $T \in (0,\infty)$ and for $t \in (0,T)$, $s' \in [0,s-1)$ and $s'' \in [0,s)$
			\begin{align*}
				\|(v^\kappa-v)(t)\|_{H^{s'}} &\leq \kappa^\frac{s-1-s'}{s-1}C(T,d,\nu,\sigma,s,\Gamma_0),
				\\
				\|(E^\kappa-E,B^\kappa-B)(t)\|_{H^{s''}} &\leq \kappa^\frac{s-s''}{s}C(T,d,\nu,\sigma,s,\Gamma_0).
			\end{align*}
		\end{enumerate}
	\end{theorem}
	
	\begin{remark} \label{re1}
		We add some comments to Theorem  \ref{theo_global_small} as follows:
		\begin{enumerate}
			\item The assumption on the initial data of the velocity is close to the best results up to now to \eqref{NSM} (without the Hall term), which are $v_0 \in L^2$ (see \cite{Arsenio-Gallagher_2020,Germain-Ibrahim-Masmoudi_2014,KLN_2024_2,Masmoudi_2010}) in two dimensions and in homogeneous version, $v_0 \in \dot{H}^\frac{1}{2}$ (see \cite{Germain-Ibrahim-Masmoudi_2014}) in three dimensions. In addition, as it will be seen in the proof of Theorem \ref{theo_global_small} in which we need to assume that $(E_0,B_0) \in H^s$ with $s > \frac{d}{2}$. It is mainly due to the controlling of the Hall term $(\Lambda^s j^n, \Lambda^s(j^{n-1} \times B^n))$, where somehow we have to bound the $L^\infty$ norms of $j^{n-1}$ and $B^n$ by their $H^s$ ones. In addition, to bound the term $(\Lambda^s j^n, \Lambda^s(v^n \times B^{n-1}))$ for instance, which seems to us that we should have at least $v_0 \in \dot{H}^{s-1}$. To investigate the GWP of \eqref{NSM-GO} for small initial data in either the critical case when $s = \frac{d}{2}$ or $s \in (0,\frac{d}{2})$, new ideas should be provided. Here, we also note that GWP to \eqref{HMHD} for small data in critical spaces (in the sense of Remark \ref{re-MHC} below), i.e., either $(v_0,B_0) \in (L^2 \times H^1)(\mathbb{R}^2)$ or $(v_0,B_0) \in (\dot{H}^\frac{1}{2} \times \dot{H}^\frac{1}{2})(\mathbb{R}^3)$ with $\nabla \times B_0 \in \dot{H}^\frac{1}{2}(\mathbb{R}^3)$, has recently obtained in \cite{Danchin-Tan_2021,Danchin-Tan_2022,Tan_2022,Zhang_2023} mainly due to some new observations and cancellations, see also \cite{Wan-Zhou_2015,Wan-Zhou_2019,Wan-Zhou_20191}.  At the moment, these results can be considered as the best ones. However, at the moment it is not clear to us how to apply the ideas in these papers to \eqref{NSM-GO}, mainly due to the lack of dissipation of $(E,B)$, even we have a connection between the two systems as given in Part $(ii)$ above. Therefore, our conditions on the initial data to \eqref{NSM-GO} is very close to the best mentioned ones to \eqref{HMHD}.
			
			\item As mentioned previously in the introduction that the weak convergence from \eqref{NSM-GO} to \eqref{HMHD} has been established in \cite{Peng-Wang-Xu_2022} in the case $\alpha = \beta = 1$, under suitable assumptions on the existence of solutions and also on their properties. Here, we provide the weak convergence in Therem \ref{theo_global_small}-$(ii)$ under the smallness assumption on the initial data only. Theorem  \ref{theo_global_small}-$(ii)$ also says that the hyperbolicity of \eqref{NSM} (due to the Maxwell equations) is weakly transferred into the parabolicity of \eqref{HMHD} as $c \to \infty$. We note that the weak convergence as $c \to \infty$ from \eqref{NSM} to \eqref{MHD} has been provided either in the case $\alpha = \beta = 1$ with $d = 2$  in \cite{Arsenio-Gallagher_2020} or in the case $\alpha = \frac{3}{2}$ and $\beta = 1$ with $d = 3$ in \cite{KLN_2024_2}. In addition, we note that the assumption $c \geq 1$ will be used to obtain  $c$-independent estimates of $(E,B)$ in $L^2_t(H^{s''} \times H^{s'})_x$, which is needed in proving the main $c$-independent estimate at each level $n \in \mathbb{N}$. 
			
			\item In Part $(iii)$, we note that the existence and uniqueness of solutions $(v,E,B)$ to \eqref{NSM-SO} in $(0,T)$, for any $T \in (0,\infty)$, can be obtained as in \cite{Arsenio-Gallagher_2020,KLN_2024_2,Masmoudi_2010} for large initial data with time-dependent estimates in two dimensions. In addition, in three dimensions the same thing can also be provided as in Part $(i)$. Moreover, as mentioned previously in the introduction, the GWP of \eqref{NSM-SO} for small initial data in $\dot{H}^\frac{1}{2}(\mathbb{R}^3)$ can be established as the case of \eqref{NSM} given in \cite{Germain-Ibrahim-Masmoudi_2014}. Finally, the limit from \eqref{HMHD} to \eqref{MHD} as the Hall coefficient goes to zero was also established in \cite{Dai_2020,Wan-Zhou_2019}.
		\end{enumerate}
	\end{remark}
	
	\begin{theorem}[\textnormal{Local well-posedness for possibly large initial data}] \label{theo_local}
		Let $\alpha = 1$, $c, \kappa,\sigma > 0$ and $\Gamma_0 := (v_0,E_0,B_0) \in H^s(\mathbb{R}^d)$ with $s \in \mathbb{R},s > \frac{d}{2} + 1$ and $d \in \{2,3\}$. 
		\begin{enumerate}
			\item[\textnormal{(}i\textnormal{)}]\textnormal{(Local well-posedness).}  \eqref{NSM-GO} with $\nu = 0$ has a unique solution $(v,E,B)$ in $(0,T_0)$ for some  $T_0 = T_0(d,\kappa,\sigma,s,\Gamma_0) \in (0,1]$ satisfying 
			\begin{align*}
				&(v,E,B) \in L^\infty(0,T_0;H^s), \quad j \in L^2(0,T_0;H^s),
			\end{align*}
			and for $t \in (0,T_0)$ 
			\begin{align*}
				\|(v,E,B)(t)\|^2_{H^s} + \int^t_0 
				\|j\|^2_{H^s}  \,d\tau \leq C(d,\kappa,\sigma,s)\|\Gamma_0\|^2_{H^s}.
			\end{align*}
			
			\item[\textnormal{(}ii\textnormal{)}] \textnormal{(Inviscid limit).} Let $\nu > 0$. There exists a sequence of solutions $(v^\nu,E^\nu,B^\nu)$ to \eqref{NSM-GO} in $(0,T_{01})$  with  $(v^\nu,E^\nu,B^\nu)_{|_{t=0}} = \Gamma_0$ and $T_{01} = T_{01}(d,\kappa,\sigma,s,\Gamma_0) \in (0,1]$.
			Moreover, for $t \in (0,T_0 := \min\{T_{01},T_{02}\})$ and for $s' \in [0,s)$
			\begin{equation*}
				\|(v^\nu-v,E^\nu-E,B^\nu-B)(t)\|_{H^{s'}} \leq \nu^\frac{s-s'}{s}C(d,\sigma,s,\Gamma_0),
			\end{equation*}
			where $(v,E,B)$ is the unique solution to \eqref{NSM-GO} with $\nu = 0$ and $(v,E,B)_{|_{t=0}} = \Gamma_0$ given as in Part $(i)$ in $(0,T_{02})$.
			
			\item[\textnormal{(}iii\textnormal{)}] \textnormal{(The limit as $c \to \infty$).} Let $c > 0$ and $(v^c_0,E^c_0,B^c_0) \in H^s(\mathbb{R}^d)$ satisfying $\textnormal{div}\, v^c_0 = \textnormal{div}\, B^c_0 = 0$ and as $c \to \infty$
			\begin{equation*}
				(v^c_0,E^c_0,B^c_0) \rightharpoonup (\bar{v}_0,\bar{E}_0,\bar{B}_0) \qquad \text{in} \quad H^s
			\end{equation*}
			for some $(\bar{v}_0,\bar{E}_0,\bar{B}_0)$ with $\textnormal{div}\, \bar{v}_0 = \textnormal{div}\, \bar{B}_0 = 0$. Then there exists a sequence of solutions $(v^c,E^c,B^c)$ to \eqref{NSM-GO} with $\nu = 0$ and $(v^c,E^c,B^c)_{|_{t=0}} = (v^c_0,E^c_0,B^c_0)$ given as in Part $(i)$ in $(0,T_0)$ for some $T_0 > 0$. In addition,  up to an extraction of a subsequence, $(v^c,B^c)$ converges to $(v,B)$ in the sense of distributions as $c \to \infty$, where
			$(v,B)$ satisfies \eqref{HMHD} with $\nu = 0$, $\beta = 1$ and $(v,B)_{|_{t=0}} = (\bar{v}_0,\bar{B}_0)$.
			
			\item[\textnormal{(}iv\textnormal{)}] \textnormal{(The limit as $\kappa \to 0$).} Let $\kappa \in (0,1)$ and $(v^\kappa,E^\kappa,B^\kappa)$ be the unique  solution of \eqref{NSM-GO} given as in Part $(i)$ with  the initial data $\Gamma_0 := (v_0,E_0,B_0)$. Let $\Gamma := (v,E,B)$ be the unique solution of \eqref{NSM-SO} with the same initial data, which is obtained as in Part $(i)$ with $\kappa = 0$. Then, for $t \in (0,T_0)$ and $s' \in [0,s)$
			\begin{equation*}
				\|(v^\kappa-v,E^\kappa-E,B^\kappa-B)(t)\|_{H^{s'}} \leq \kappa^\frac{s-s'}{s}C(d,\nu,\sigma,s,\Gamma_0).
			\end{equation*}
		\end{enumerate}
		
	\end{theorem}
	
	\begin{remark}
		We add some comments to Theorem \ref{theo_local} as follows:
		\begin{enumerate}
			\item As in the case of the Euler equations (see \cite{Majda_Bertozzi_2002}), in Part $(i)$ the condition $s > \frac{d}{2} + 1$ is due to the controlling of the convection term in the $\dot{H}^s$ estimate in the inviscid case $\nu = 0$. However, it is possible to control other terms with $s > \frac{d}{2}$ only. Thus, if $\nu > 0$ and $\alpha = 1$ in \eqref{NSM-GO} then the idea in \cite[Theorem 1.2]{Fefferman-McCormick-Robinson-Rodrigo_2014} can be applied to control the $\dot{H}^s$ estimate of the convection term, i.e., 
			as in the proof of Part $(i)$-Step 2,  for $s > \frac{d}{2}$
			\begin{equation*}
				\int_{\mathbb{R}^d} [\Lambda^s (v^1 \cdot \nabla v^2) - v^1 \cdot \nabla \Lambda^s v^2] \cdot \Lambda^s v^2 \,dx \leq C(d,s) \|\nabla v^1\|_{H^s} \|v^2\|^2_{H^s},
			\end{equation*}
			which allows us to provide the LWP of \eqref{NSM-GO} for $H^s(\mathbb{R}^d)$ initial data with $s > \frac{d}{2}$. In addition, in this viscous case the assumption on the initial velocity can be reduced more to $v_0 \in H^{s-1+\epsilon_0}(\mathbb{R}^d)$ for $s > \frac{d}{2}$ and any $\epsilon_0 \in (0,1)$ by adapting the technique in \cite{Fefferman-McCormick-Robinson-Rodrigo_2017} for \eqref{MHD} with $\alpha = 1$, $\nu > 0$ and $\sigma = \infty$, where their result is almost optimal. Furthermore, as mentioned previously, the LWP of \eqref{NSM-GO} in critical space $H^\frac{d}{2}(\mathbb{R}^d)$ could require a new technique, since \cite[Theorem 1.2]{Fefferman-McCormick-Robinson-Rodrigo_2014} does not hold at least for $d = 2$. In the case of \eqref{HMHD} with $\alpha = \beta = 1$, $\nu,\sigma,\kappa > 0$, see also \cite{Dai_2021_2} for initial data $(v_0,B_0) \in (H^s \times H^{s+1-\epsilon_0})(\mathbb{R}^3)$ for $s > \frac{1}{2}$ and any $\epsilon_0 \in (0,s-\frac{1}{2})$. The LWP of \eqref{HMHD} with $\alpha = \beta  = 1$ for possibly large initial data in the critical spaces (in the sense of Remark \ref{re-MHC}) $(v_0,B_0) \in (L^2 \times H^1)(\mathbb{R}^2)$ and $(v_0,B_0) \in (\dot{H}^\frac{1}{2} \times \dot{H}^\frac{1}{2})(\mathbb{R}^3)$ with $\nabla \times B_0 \in \dot{H}^\frac{1}{2}(\mathbb{R}^3)$ with be considered in our forthcoming work \cite{KLN_2024_4}, where the splitting method on initial data is applied. We remark that this idea can also be applied to obtain the LWP of \eqref{NSM-GO} for initial data in $(H^{s-1} \times H^s \times H^s)(\mathbb{R}^d)$ with $s > \frac{d}{2}$. This local result seems to be almost optimal, since recently the authors in \cite{Chen-Nie-Ye_2024} proved the ill-posedness of \eqref{MHD} with $\sigma = \infty$ and $\kappa = 0$ (without the magnetic resistivity and Hall effect) for initial data in $(H^{\frac{d}{2}-1} \times H^\frac{d}{2})(\mathbb{R}^d)$, $d \geq 2$. In the cases of \eqref{NSM} and \eqref{NSM-SO}, the GWP in $(H^\epsilon \times H^1 \times H^1)(\mathbb{R}^2)$ is given in \cite{KLN_2024_2} for any $\epsilon \in (0,1)$, and the LWP in $(H^{\frac{1}{2}+} \times H^\frac{3}{2} \times H^\frac{3}{2})(\mathbb{R}^3)$ can be proved in a suitable way. Thus, the well- or ill-posedness of \eqref{NSM}, \eqref{NSM-SO} and \eqref{NSM-GO} in $(H^{\frac{d}{2}-1} \times H^\frac{d}{2} \times H^\frac{d}{2})(\mathbb{R}^d)$, $d \in \{2,3\}$ is still open. The main difficulty relies on obtaining the $\dot{H}^\frac{d}{2}(\mathbb{R}^d)$ estimate of $(E,B)$ in which we do not have the Sobolev embedding $\dot{H}^\frac{d}{2} \xhookrightarrow{} L^\infty$ and some suitable paraproduct rule in the Sobolev context.
			
			
			\item In Part $(iii)$, the LWP of \eqref{NSM-SO} can be obtained as in \cite[Theorem 1.4]{KLN_2024_2} in some time interval $(0,T_0)$ for some constant $T_0 > 0$ depnending on the parameteres and initial data. Moreover, here we also note that it can be seen from the proof that the time existence $T_0$ in Part $(i)$ does not depend on $\kappa$ in the case $\kappa \in (0,1)$.
		\end{enumerate}
	\end{remark}

	%
	\subsubsection{Magnetic helicity conservation}
	%
	
	In this section,  we study the (conditional) conservation of magnetic helicity to \eqref{NSM-GO} and \eqref{HMHD} as the electric conductivity $\sigma$ goes to infinity.  More precisely, the statement in the case of \eqref{NSM-GO} is given as follows.
	
	\begin{theorem}[Conditional magnetic helicity conservation for \eqref{NSM-GO}] \label{theo-MH}
		The following statements hold \textnormal{(}for simplicity, we will write the dependency on $\sigma$ only\textnormal{)}.
		\begin{enumerate}
			\item[\textnormal{(}i\textnormal{)}] Let $\alpha = 1$, $c \geq 1$, $\kappa,\nu, \sigma > 0$ and  $\Gamma^\sigma_0 := (v^\sigma_0,E^\sigma_0,B^\sigma_0) \in (H^{s-1} \times H^s \times H^s)(\mathbb{R}^3)$ with $s \in \mathbb{R}, s > \frac{3}{2}$. If $\Gamma^\sigma_0$ satisfies \eqref{small_data} and $(E^\sigma_0,B^\sigma_0) \in \dot{H}^{-1}(\mathbb{R}^3)$ then \eqref{NSM-GO} has a unique global solution $\Gamma^\sigma = (v^\sigma,E^\sigma,B^\sigma)$ satisfying \eqref{NSM-GO-limit} and for $t > 0$
			\begin{align} \nonumber
				&\|\Gamma^\sigma(t)\|^2_{X^s} +  \int^t_0 \|\nabla v^\sigma\|^2_{H^{s-1}} + \|v^\sigma\|^2_{L^\infty} + \|E^\sigma\|^2_{\dot{H}^{s''}} \,d\tau 
				\\ \label{Gamma_sigma_1} \tag{MHC-1}
				&\quad + \int^t_0 \|B^\sigma\|^2_{\dot{H}^{s'}} + \|j^\sigma\|^2_{H^s} \,d\tau \leq C(d,\kappa,\nu,\sigma,s) \|\Gamma^\sigma_0\|^2_{X^s},
				\\ \label{Gamma_sigma_2} \tag{MHC-2}
				&\|\Gamma^\sigma(t)\|^2_{L^2} + \int^t_0 \nu \|\nabla v^\sigma\|^2_{L^2} + \frac{1}{\sigma}\|j^\sigma\|^2_{L^2} \,d\tau \leq \|\Gamma^\sigma_0\|^2_{L^2},
				\\ \label{Gamma_sigma_3} \tag{MHC-3}
				&\|(E^\sigma,B^\sigma)(t)\|^2_{\dot{H}^{-1}} + \frac{1}{\sigma} \int^t_0  \|j^\sigma\|^2_{\dot{H}^{-1}} \,d\tau \leq C(t,\kappa,\nu,\sigma,s,\Gamma^\sigma_0).
			\end{align}
			
			\item[\textnormal{(}ii\textnormal{)}] Suppose that Part $(i)$ still holds without assuming Condition \eqref{small_data} on the initial data $\Gamma^\sigma_0$. Assume that $\Gamma^\sigma_0$ is uniformly bounded in $L^2(\mathbb{R}^3)$ in terms of $\sigma$. Furthermore, if $B^\sigma_0 \to B_0$ in $\dot{H}^{-1}(\mathbb{R}^3)$ as $\sigma \to \infty$ for some $B_0 \in (H^s \cap \dot{H}^{-1})(\mathbb{R}^3)$ with $\textnormal{div}\, B_0 = 0$ then
			\begin{equation} \label{MHC} \tag{MHC-a}
				\lim_{\sigma \to \infty} \int_{\mathbb{R}^3} A^\sigma(t) \cdot B^\sigma(t) \,dx 
				=  \int_{\mathbb{R}^3} A_0 \cdot B_0\,dx \quad \text{for a.e. }   t \in (0,\infty),
			\end{equation}
			where  $\nabla \times  f = g$ and $\textnormal{div}\,f = 0$ for $f \in \{A^\sigma,A_0\}$ corresponding to $g \in \{B^\sigma,B_0\}$. Finally, if the initial magnetic helicity is positive then there exists an absolute positive constant $C$ such that
			\begin{equation} \label{MHC_1} \tag{MHC-b}
				\liminf_{t\to \infty} \liminf_{\sigma \to \infty} \|B^{\sigma}(t)\|^2_{\dot{H}^{-\frac{1}{2}}(\mathbb{R}^3)} \geq  C\int_{\mathbb{R}^3} A_0 \cdot B_0 \,dx > 0.
			\end{equation}
		\end{enumerate}
	\end{theorem}
	
	We now focus on the conservation of magnetic helicity to \eqref{HMHD} with the critical fractional Lapacians for both the velocity and magnetic fields, i.e., (to simplify the presentation, we will write the dependency on $\sigma$ only) for  $\alpha_c := \frac{5}{4}$ and $\beta_c := \frac{7}{4}$.

	\begin{theorem}[Magnetic helicity conservation for the critical H-MHD] \label{theo-C-HMHD}
		Let $\kappa,\nu,\sigma > 0$. Assume that $v^\sigma_0 \in H^s(\mathbb{R}^3)$ and $B^\sigma_0 \in (H^s \cap \dot{H}^{-1})(\mathbb{R}^3)$ with $s > \frac{7}{4}$. Then \eqref{HMHD} with $(\alpha,\beta) = (\alpha_c,\beta_c)$ has a unique solution $(v^\sigma,B^\sigma)$ with $(v^\sigma,B^\sigma)_{|_{t=0}} = (v^\sigma_0,B^\sigma_0)$ such that for any $t \in (0,\infty)$
		\begin{align*}
			(v^\sigma,B^\sigma) &\in L^\infty(0,\infty;H^s) \cap L^2(0,\infty;\dot{H}^{\alpha_c} \times \dot{H}^{\beta_c}) \cap L^2(0,t;H^{s+\alpha_c} \times H^{s+\beta_c}), 
			\\
			B^\sigma &\in L^\infty(0,t;\dot{H}^{-1}),
		\end{align*}
		and for some positive constant $C = C(t,\kappa,\sigma,\nu,s,v^\sigma_0,B^\sigma_0)$
		\begin{align*}
			&\|(v^\sigma,B^\sigma)(t)\|^2_{L^2}  + \int^t_0 \nu \|v^\sigma\|^2_{\dot{H}^{\alpha_c}}  + \frac{1}{\sigma} \|B^\sigma\|^2_{\dot{H}^{\beta_c}} \,d\tau \leq \|(v^\sigma_0,B^\sigma_0)\|^2_{L^2},
			\\
			&\|(v^\sigma,B^\sigma)(t)\|^2_{H^s}  + \|B^\sigma(t)\|^2_{\dot{H}^{-1}} + \int^t_0 \nu \|v^\sigma\|^2_{H^{s+\alpha_c}}  + \frac{1}{\sigma} \|B^\sigma\|^2_{H^{s+\beta_c}} \,d\tau \leq C.
		\end{align*}
		In addition,  if $(v^\sigma_0,B^\sigma_0)$ is uniformly bounded in $L^2(\mathbb{R}^3)$ in terms of $\sigma$ and $B^\sigma_0 \to B_0$ in $\dot{H}^{-1}(\mathbb{R}^3)$ as $\sigma \to \infty$ for some $B_0 \in (H^s \cap \dot{H}^{-1})(\mathbb{R}^3)$ with $\textnormal{div}\, B_0 = 0$ then \eqref{MHC} follows. Finally, if the initial magnetic helicity is positive then \eqref{MHC_1} holds.
	\end{theorem}
	
	\begin{remark} \label{re-MHC} We add some comments to Theorems \ref{theo-MH} and \ref{theo-C-HMHD} as follows: 
		\begin{enumerate}
			\item As mentioned previously, that the additional negative regularity condition of $(E_0,B_0)$ is needed since we are in the whole space $\mathbb{R}^3$, see an example of being not well-defined of the magnetic helicity in \cite{Faraco-Lindberg-Szekelyhidi_2021}. Formally, the main additional condition for the conservation of the magnetic helicity to \eqref{NSM-GO} is the GWP for large initial data, while to \eqref{HMHD}, that is the critical fractional Laplacian exponents. Although \eqref{HMHD} with $\alpha = \beta = 1$ does not admit a scaling invariant property, in three dimensions $\alpha_c$ is the critical exponent to the usual Navier-Stokes system coming from the scaling $v_\lambda(t,x) \mapsto \lambda v(\lambda^2 t, \lambda x)$ and $\beta_c$ is the one to the Hall equations under the scaling $B_\lambda(t,x) \mapsto B(\lambda^2 t,\lambda x)$ for $\lambda > 0$. Finally, the results here are also  formally related to the Taylor's conjecture on the magnetic helicity conservation of \eqref{MHD} with $\alpha = \beta = 1$, see more details in the next section below.
		\end{enumerate}
	\end{remark}

	%
	\subsubsection{Asymptotic stability}
	%
	
	In this part, our aim is to investigate the stability of \eqref{NSM-GO} in the case of $\alpha = 0$ and $\nu > 0$. More precisely, let us now focus on the stability issue of \eqref{NSM-GO} around its stationary states. In this case, if we look for a zero-velocity steady solution, i.e., $(v^* = 0,E^*,B^*,\pi^*_1,\pi^*_2)$ then it should satisfy
	\begin{equation} \label{S-NSM-GO} \tag{S-NSM-GO}
		\left\{
		\begin{aligned}
			\nabla \pi^*_1 &= j^* \times B^*, 
			\quad
			\nabla \times B^* = j^*,
			\quad
			\nabla \times E^* = 0,
			\\
			j^* &= -\kappa (j^* \times B^*) + \sigma(-\nabla \pi^*_2 + cE^*),
			\\
			\textnormal{div}\, B^* &= \textnormal{div}\, E^* = \textnormal{div}\, j^* = 0.
		\end{aligned}
		\right.
	\end{equation}
	Indeed, by using the following well-known identity
	\begin{equation*}
		j^* \times B^* = (\nabla \times B^*) \times B^* = B^* \cdot \nabla B^* - \frac{1}{2} \nabla |B^*|^2,
	\end{equation*}
	it follows that $B^*$ also satisfies the following stationary Euler-type equations, which is also known as the magnetohydrostatic system 
	\begin{equation} \label{MHS} \tag{MHS}
		B^* \cdot \nabla B^* + \nabla p^* = 0 \qquad \text{and}\qquad \textnormal{div}\, B^* = 0 \qquad \text{where}\qquad p^* := -\frac{1}{2}|B^*|^2 - \pi^*_1.
	\end{equation}
	In three dimensions, solutions to \eqref{MHS} either in bounded domains or on the torus are recently constructed in \cite{Constantin-Pasqualotto_2023} as infinite time limits of Voigt approximations\footnote{That means $(\partial_t v,\partial_t B)$ is replaced by $(\partial_t (-\Delta)^{\alpha_0} v, \partial_t (-\Delta)^{\beta_0} B)$ for some $\alpha_0,\beta_0 > 0$.} of viscous and non-resistive \eqref{MHD} (i.e., with $\alpha = 1$ and $\sigma = \infty$). It is also believed that \eqref{MHS} plays an important role in connection to the design of nuclear fusion devices such as tokamaks and stellarators. There are several examples of $(v^*,E^*,B^*,\pi^*_1,\pi^*_2)$ to either \eqref{S-NSM-GO} or \eqref{MHS} such as for $x \in \mathbb{R}^d$
	\begin{align*}
		v^* &= E^* = 0, \quad  B^*(x) = \text{constant vector in } \mathbb{R}^3, \quad \pi^*_1 = \text{constant} \quad \text{and}\quad  \pi^*_2 = \text{constant};
		\\
		v^* &= E^* = 0, \quad B^*(x) \in \{(-x_1,x_2,0),(x_2,x_1,0),...\}, \quad \pi^*_1 = \text{constant} \quad \text{and}\quad  \pi^*_2 = \text{constant}.
	\end{align*}
	By setting
	\begin{equation*}
		\bar{v} := v + v^* = v, \quad \bar{E} := E + E^*, \quad \bar{B} := B + B^*,  \quad  \bar{\pi}_1 := \pi_1 + \pi^*_1 \qquad \text{and} \qquad \bar{\pi}_2 := \pi_2 + \pi^*_2,
	\end{equation*}
	it can be seen from \eqref{NSM-GO} for $(\bar{v},\bar{E},\bar{B},\bar{\pi}_1,\bar{\pi}_2,\bar{j})$ and \eqref{S-NSM-GO} that the perturbation $(v,E,B,\pi_1,\pi_2)$ satisfies
	\begin{equation} \label{NSM-GO*} \tag{NSM-GO*}
		\left\{
		\begin{aligned}
			\partial_t v + v \cdot \nabla v &= -\nabla (\pi_1 + \pi^*_1)  -\nu (-\Delta)^\alpha v + \bar{j} \times (B + B^*), 
			\\
			\frac{1}{c}\partial_t E - \nabla \times (B + B^*)  &= -\bar{j},
			\\
			\frac{1}{c}\partial_t B + \nabla \times E &= 0,
			\\
			\bar{j} + \kappa \bar{j} \times (B + B^*) &= \sigma(- \nabla (\pi_2 + \pi^*_2) + c(E + E^*) + v \times (B + B^*)), 
			\\
			\textnormal{div}\, v = \textnormal{div}\, B = \textnormal{div}\, E &= \textnormal{div}\, \bar{j} = 0,
		\end{aligned}
		\right.
	\end{equation}
	with the initial data is denoted by $(v,E,B)_{|_{t=0}} = (v_0,E_0,B_0)$. The above system has new terms compared to \eqref{NSM-GO}. Thus, in order to have a better understanding, we now focus on the case that $B^*$ is a constant vector and $\alpha = 0$, which seems to be the easiest case to start. The case of either $\alpha = 1$ or $B^*$ is not a constant vector, seems to be much more complicated and requires new ideas. In addition, we also have a comparison with \eqref{HMHD} as $c \to \infty$. More precisely, if $(v^* = 0,B^*,p^*)$ is a stationary solution to \eqref{HMHD} with $\alpha = 0$ and $\beta = 1$ then 
	\begin{equation} \label{S-H-MHD} \tag{S-H-MHD}
		\nabla \pi^* = j^* \times B^*,
		\quad
		\frac{1}{\sigma}\Delta B^* = \frac{\kappa}{\sigma} \nabla \times (j^* \times B^*),\quad  j^* := \nabla \times B^*
		\quad \text{and}\quad
		\textnormal{div}\, B^* = 0.
	\end{equation}
	As mentioned previously, if $B^*$ is a solution to \eqref{S-H-MHD} then $B^*$ also satisfies \eqref{MHS}. Note that the above examples also satisfy \eqref{S-H-MHD}. Moreover, it follows from \eqref{HMHD} for $(\bar{v},\bar{B},\bar{\pi},\bar{j})$ with $\alpha = 0$, $\beta = 1$ and \eqref{S-H-MHD} that the perturbation $(v := \bar{v}-v^*,B := \bar{B}-B^*,\pi := \bar{\pi}-\pi^*,j := \bar{j}-j^*)$ with $j := \nabla \times B$ satisfies
	\begin{equation} \label{HMHD*} \tag{H-MHD*}
		\left\{
		\begin{aligned}
			\partial_t v + v \cdot \nabla v + \nabla \pi &= -\nu v + j \times (B+B^*) + j^* \times B,
			\\
			\partial_t B - \nabla \times (v \times (B+B^*)) &= \frac{1}{\sigma}\Delta B - \frac{\kappa}{\sigma} \nabla \times (j \times (B+B^*)) - \frac{\kappa}{\sigma} \nabla \times (j^* \times B),
			\\
			\textnormal{div}\, v = \textnormal{div}\, B &= 0.
		\end{aligned}
		\right.
	\end{equation}
	
	We are now going to the statement, which is given as follows.
	
	\begin{theorem}[\textnormal{Asymptotic stability around a constant magnetic field}] \label{theo_stability}
		Let $\alpha = 0$, $d \in \{2,3\}$, $c,\kappa, \nu,\sigma > 0$. Assume that $\Gamma_0 := (v_0,E_0,B_0) \in H^s(\mathbb{R}^d)$ with $s \in \mathbb{R}, s > \frac{d}{2} + 1$ and $\textnormal{div}\, v_0 = \textnormal{div}\, B_0 = 0$. In addition, assume that $B^*$ be a constant vector in $\mathbb{R}^3$. There exists a positive constant $C_1 = C_1(B^*,d,\kappa,\nu,\sigma,s)$ such that if $C_1\|\Gamma_0\|_{H^s} \leq 1$ then the following statements hold.
		\begin{enumerate}
			\item[\textnormal{(}i\textnormal{)}] \textnormal{(Global well-posedness).} \eqref{NSM-GO*} admits a global solution $\Gamma := (v,E,B)$ such that for $s'' \in [0,s]$ and $s' \in [1,s]$ 
			\begin{align*}
				\Gamma \in L^\infty(0,\infty;H^s) \cap L^2(0,\infty;H^s \times \dot{H}^{s''} \times \dot{H}^{s'}) \quad \text{and} \quad \bar{j} &\in L^2(0,\infty;H^s).
			\end{align*}
			and for $t > 0$
			\begin{align*}
				\|\Gamma(t)\|^2_{H^s} +  \int^t_0 \|v\|^2_{H^s} + 
				\|\bar{j}\|^2_{H^s} \,d\tau &\leq C(B^*,d,\kappa, \nu, \sigma,s) \|\Gamma_0\|^2_{H^s},
				\\
				\int^t_0 \|E\|^2_{\dot{H}^{s''}} + \|B\|^2_{\dot{H}^{s'}}  \,d\tau &\leq C(B^*,c,d,\kappa, \nu, \sigma,s) \|\Gamma_0\|^2_{H^s}.
			\end{align*}
			
			\item[\textnormal{(}ii\textnormal{)}] \textnormal{(Large-time behavior).} As $t \to \infty$, for $s'_0 \in [1,s)$, $s''_0 \in [0,s)$, $p \in (2,\infty]$ and $q \in [1,\infty]$
			\begin{equation*}
				\|(v,E,\bar{j})(t)\|_{H^{s''_0}}, \|B(t)\|_{L^p}, \|B(t)\|_{\dot{H}^{s'_0}} \quad \text{and} \quad \|B(t)\|_{L^q_{\textnormal{loc}}} \to 0. 
			\end{equation*}
			
			\item[\textnormal{(}iii\textnormal{)}] \textnormal{(The limit as $c \to \infty$).} 
			Let $c > 0$ and $\Gamma^c_0 := (v^c_0,E^c_0,B^c_0) \in H^s$ satisfying $\textnormal{div}\, v^c_0 =  \textnormal{div}\, B^c_0 = 0$, $C_1\|\Gamma^c_0\|_{H^s} \leq 1$ uniformly in terms of $c$, and
			\begin{equation*}
				\Gamma^c_0 \rightharpoonup (\bar{v}_0,\bar{E}_0,\bar{B}_0) \quad \text{in } \, H^s \text{ as } c \to \infty
			\end{equation*}
			for some $(\bar{v}_0,\bar{E}_0,\bar{B}_0) \in H^s$ with $\textnormal{div}\, \bar{v}_0 = \textnormal{div}\, \bar{B}_0 = 0$. Then, there exists a sequence of global solutions $\Gamma^c := (v^c,E^c,B^c)$ to \eqref{NSM-GO*} with $\alpha = 0$ and $\Gamma^c_{|_{t=0}} = \Gamma^c_0$ given as in Part $(i)$. In addition,  up to an extraction of a subsequence, $(v^c,B^c)$ converges to $(v,B)$ in the sense of distributions as $c \to \infty$, where
			$(v,B)$ satisfies \eqref{HMHD*} with $(v,B)_{|_{t=0}} = (\bar{v}_0,\bar{B}_0)$.
		\end{enumerate}
		
	\end{theorem}
	
	\begin{remark}
		We add some comments to Theorem \ref{theo_stability} as follows:
		\begin{enumerate}
			\item  Moreover, it can be seen from \eqref{S-NSM-GO} that $\Delta B^* = 0$, and furthermore by Liouville's theorem (see \cite{Evans_2010}), if $B^*$ is bounded then $B^*$ is a constant vector. Thus, the boundedness of $B^*$ is equivalent to the constant one. In addition, if $B^*$ is a constant vector in $\mathbb{R}^3$ then  $cE^* - \nabla \pi^*_2  = \nabla \pi^*_1 = 0$.
			
			\item As done previously in \cite{KLN_2024_2}, if we assume in addition that $(v_0,B_0) \in L^1(\mathbb{R}^d)$ then we can obtain an explicit rate of decay in time for solution $(v,B)$ to \eqref{HMHD*}. However, it is not clear for us at the moment, how we can obtain a similar thing in the case of  \eqref{NSM-GO*}. 
			
			\item The stability of \eqref{NSM-GO} near $B^*$ in the case either $\alpha = 1$ or $B^*$ is not a constant vector is much more complicated, where we should either control $L^2_tL^\infty_x$ norm of the velocity with $L^2_tH^s_x$ right-hand side or define a suitable weighted functional setting, which is at the moment not so clear to us.
		\end{enumerate}
	\end{remark}
	
	The rest of the present paper is organized as follows. Before going to the proofs of the stated theorems, we will briefly review strategy of proofs in Section \ref{sec:stra}. After that the proofs of Theorems \ref{theo_global_small}, \ref{theo_local}, \ref{theo-MH}, \ref{theo-C-HMHD} and \ref{theo_stability} will be provided in Sections \ref{sec:global_small}, \ref{sec:local}, \ref{sec:theo-MH}, \ref{sec:theo-C-HMHD} and \ref{sec:stability}, respectively. In addition, some technical tools and results are recalled  and proved in Appendix given in Section \ref{sec:app}.
	
	%
	\section{Strategy of proofs} \label{sec:stra}
	%
	
	In this section, we aim to provide briefly strategy of proofs in Theorems \ref{theo_global_small}, \ref{theo_local}, \ref{theo-MH}, \ref{theo-C-HMHD} and \ref{theo_stability}. Roughly speaking, the proofs of  Theorems Theorems \ref{theo_global_small}, \ref{theo_local}, and \ref{theo_stability} mainly rely on the analysis of the following iteration system, which is corresponding to the Leray projection\footnote{Here, The usual Leray projection into divergence-free vectors is formally defined by $\mathbb{P}(f) := f + \nabla(-\Delta)^{-1} \text{div} \,f$.} version of \eqref{NSM-GO} and is given as follows for each $n \in \mathbb{N}$
	\begin{equation} \label{NSM-GO_na} 
		\left\{
		\begin{aligned}
			\partial_t v^n + \mathbb{P}\mathcal{T}_{\overline{n}}(v^{n-1} \cdot \nabla v^n) &= -\nu(-\Delta)^\alpha v^n + \mathbb{P}\mathcal{T}_{\overline{n}}(j^n \times B^{n-1}), 
			\\
			\frac{1}{c}\partial_t E^n - \nabla \times B^n &= -j^n,
			\\
			\frac{1}{c}\partial_t B^n + \nabla \times E^n &= 0,
			\\
			\sigma(cE^n + \mathbb{P}\mathcal{T}_{\overline{
					n}}(v^n \times B^{n-1})) &= j^n + \kappa\mathbb{P}\mathcal{T}_{\overline{n}}(j^{n-1} \times B^n),
			\\
			\textnormal{div}\, v^n = \textnormal{div}\, E^n &= \textnormal{div}\, B^n = \textnormal{div}\, j^n = 0,
			\\
			(v^n,E^n,B^n)_{|_{t=0}} &= \mathcal{T}_{\overline{n}}(v_0,E_0,B_0), 
			\\ 		
			(v^0,E^0,B^0,j^0) &= (0,0,0,0),
		\end{aligned}
		\right.
	\end{equation}
	where for $n \in \mathbb{N}$ we define $\bar{n} := 2^n$ and the initial data is denoted by $\Gamma_0 := (v_0,E_0,B_0)$ belonging to suitable spaces. Here, for $m \in \mathbb{N}$, $\mathcal{T}_m$ is the usual Fourier truncation operator.\footnote{More precisely, $\mathcal{F}(\mathcal{T}_m f)(\xi) := \mathbbm{1}_{B_{m}}(\xi) \mathcal{F}(f)(\xi)$, where $\mathbbm{1}_{B_m}$ is the characteristic function of $B_m$ and $B_m$ is the ball of radius $m \in \mathbb{N}$ centered at the origin.} 
	The main reason to introduce this iteration system is that in the Ohm's  in \eqref{NSM-GO}, it seems to us that the electric current field $j$ can not be explicitly expressed only in terms of $(v,E,B)$, which is different compared to that of \eqref{NSM} as for example in \cite{Arsenio-Gallagher_2020,Germain-Ibrahim-Masmoudi_2014,Ibrahim-Keraani_2012,KLN_2024_2,Masmoudi_2010}, where it is not difficult to obtain the local existence and uniqueness of solutions to an approximate system. Thus, in the case of \eqref{NSM-GO}, we may not apply directly the standard Cauchy-Lipschitz theorem (see \cite{Majda_Bertozzi_2002}) to obtain the local existence and uniqueness of approximate solutions. Therefore, we need to linearize the original system in a suitable way, especially for the nonlinear terms, which are related to $j$, for example the Hall one. More precisely, in the rest of this section, we will visit each theorem with reviewing its strategy of proof. 
	
	$\bullet$ We first focus on the proof of Theorem \ref{theo_global_small}-$(i)$ for the initial data $\Gamma_0 \in (H^{s-1} \times H^s \times H^s)(\mathbb{R}^d)$ with $s > \frac{d}{2}$ for $d \in \{2,3\}$ and $\alpha = 1$, where at each level $n \in \mathbb{N}$, we can provide the  existence and uniqueness of approximate solutions $\Gamma^n := (v^{n},E^{n},B^{n}) \in C^1([0,T^{n}_*);H^{s-1}_{n,0} \times H^s_{n,0} \times H^s_{n,0})$ (see the definition of this functional space below) for some $T^{n}_* > 0$ by using the standard Cauchy-Lipschitz theorem. As usual, we then focus on obtaining uniform estimates in terms of $n$, which will allow us to conclude that $T^{n}_* = \infty$ and to control solutions at the next level $n+1$ as well. More precisely, we first focus on the level $n = 1$ with providing some good estimates and proving $T^1_* = \infty$, which will be applied in the next level $n = 2$ to obtain some good estimates and $T^2_* = \infty$. In this level $n = 2$, when we estimate the system in the Sobolev norm, there is a difficulty on the controlling of the velocity in the $L^2_tL^\infty_x$ norm. To overcome this point, we focus on the Navier-Stokes equations by considering its as a heat equation with the convection term and Lorentz force on the right-hand side. Surprisingly, the two-dimensional case is more complicated than that of the case in three dimensions, where we are not able to apply some paraproduct estimate, since it is the critical case in two dimensions. Therefore, in the case $d = 2$, by using the ideas in \cite{Germain-Ibrahim-Masmoudi_2014,Ibrahim-Keraani_2012}, we decompse the velocity in suitable form and apply the usual parabolic estimates for heat equation. These will allows us to control the velocity either in the $L^2_tL^\infty_x$ norm directly by using Duhamel formula with the bounds in the $L^1_t\dot{H}^{\frac{d}{2}-1}_x$ norm of some terms on the right-hand side of the corresponding heat equation or in the $L^2_t\dot{B}^\frac{d}{2}_{2,1,x}$ norm by using maximal regularity type estimates for the heat equation in Besov spaces. Next, we use some paraproduct estimates in Lemma \ref{lem_paraproduct}, which are slightly different to those of in \cite{Germain-Ibrahim-Masmoudi_2014,Ibrahim-Keraani_2012}. In addition, we also need some suitable bound of $(E,B)$ in the $L^2_t\dot{H}^1_x$ norm in the case of $d = 2$, which can be done by using a non usual test function for the Maxwell system. In the case $d = 3$, we also decompose the Navier-Stokes system in a simpler way, where it is possible to apply the Duhamel formula and the usual paraproduct inequality directly. Under using the smallness assumption of the initial data, we are able to close the main estimate at this level $n = 2$. We then use an induction argument in which we assume some good properties (which have been established in the level $n = 2$) hold in some level $n = k > 2$ with $T^k_* = \infty$, $k \in \mathbb{N}$, and we need to prove that these are also true in the next level $n = k+1$ and $T^{k+1}_* = \infty$. In this level $n = k + 1$, we can repeat the proof in the level $n = 2$ by using the induction assumption, since we have a similar form of the two approximate systems.  After obtaining good information at every level $n \in \mathbb{N}$, we also need to consider the limit as $n \to \infty$, where we expect to obtain the original system \eqref{NSM-GO} from \eqref{NSM-GO_na}. In order to pass to the limit as $n \to \infty$, for example it is needed to show that this approximate solution $\Gamma^n$ and $j^n$ are Cauchy sequences in some particular sense, where we need to consider the system for the difference with defining its suitable energy form and use the smallness of the initial data. This Cauchy property will allow us to provide that \eqref{NSM-GO_na} strongly converges to \eqref{NSM-GO} in a suitabe sense. Finally, the uniqueness of solutions is established in the usual way as well. In addition, the proof of the second part, i.e., Theorem \ref{theo_global_small}-$(ii)$, follows the idea in \cite{Arsenio-Gallagher_2020} with using uniform bounds in terms of $c$, which are obtained as in Theorem \ref{theo_global_small}-$(i)$. This provides us a connection  between \eqref{NSM-GO} and \eqref{HMHD} with $\alpha = \beta = 1$ as $c \to \infty$. While Theorem \ref{theo_global_small}-$(iii)$ gives us a connection between \eqref{NSM-GO} and \eqref{NSM-SO} with $\alpha = 1$ in which the difference is computed in terms of the Hall constant $\kappa$. Here, the idea of the proof is similar to the proof of inviscid limit from Navier-Stokes to Euler equations.
	
	$\bullet$ Similarly, we also apply partially the above strategy to the proof of Theorem \ref{theo_local}, where we investigate the inviscid case, i.e., $\nu = 0$ for the initial data $\Gamma_0 \in H^s(\mathbb{R}^d)$ with $s > \frac{d}{2} + 1$ and $d \in \{2,3\}$. Since we do not have the diffusion term in this case, then the main part in the previous theorem can not be obtained. In fact, we do not need to follow the idea in the proof of the local well-posedness to the Euler equations as given in \cite{Majda_Bertozzi_2002}, since at each level the approximate system is linear. By choosing the time existence to be small enough, we are able to control the approximate solutions $\Gamma^n$ at every level $n \in \mathbb{N}$. However, we can not follow the idea in the proof of Theorem \ref{theo_global_small} due to some technical difficulties, to prove some Cauchy properties, which allows us to pass to the limit as $n \to \infty$. Thus, we first have to pass to the limit in the sense of distibutions and then prove the uniqueness of solutions in the usual way, since the limiting solution is smooth enough. Furthermore, the proof of the inviscid limit is similar to that of from the Navier-Stokes to the Euler equations. Similarly, the limit as either $c \to \infty$ or $\kappa \to 0$ can be done as in the proofs of Theorem \ref{theo_global_small}-$(ii)-(iii)$.
	
	$\bullet$ Next, the proof of Theorem \ref{theo_stability} almost follows the main idea as that of in Theorem \ref{theo_global_small} in which we study the asymptotic stability of \eqref{NSM-GO} with $\alpha = 0$ around a constant (or equivalently bounded) magnetic field $B^*$ for the initial data $\Gamma_0 \in H^s(\mathbb{R}^d)$ with $s > \frac{d}{2} + 1$ and $d \in \{2,3\}$ (see \cite{KLN_2024_2} in the case of \eqref{NSM}). In this case, there are additonal term which are related to $B^*$. Since $B^*$ is constant vector in $\mathbb{R}^3$, we will have some good cancellation properties. However, it seems difficult to control all these new terms by using the Laplacian of the velocity with the strategy of decomposing the Navier-Stokes equations as in Theorem \ref{theo_global_small}. Thus, we first force ourselves to the case of having a velocity damping term. In this case, it allows us to close the main estimate by using the smallness assumption of the initial data. Moreover, to investigate the large-time behavior, we can do in a similar way as in \cite{KLN_2024_2} for $(v,E,\bar{j})$ in the $L^2$ norm and then by using interpolation inequalities and uniform bounds to get similar things in higher Sobolev norms. But to obtain the asymptotic behavior for the magnetic field $B$, new ideas should be considered here (compared to \cite{KLN_2024_2} in the case of \eqref{NSM}). Somehow, we need to obtain the controlling of $B$ in the $L^2_t\dot{H}^1_x$ norm. In order to do that, we test the Maxwell esuations by $-(\nabla \times B,\nabla \times E)$ and use the bound of $E$ in the $L^2_t\dot{H}^1_x$ norm, which can be obtained in an easier way. Then, interpolation and Sobolev inequalities allow us to obtain the decay in time of $B$ in higher homogeneous Sobolev norms and in the $L^p$ norm for $p \in (2,\infty]$ as well. As a consequence, we also obtain the large-time behavior of $B$ in the $L^2_{\textnormal{loc}}$ norm. However, it is not clear to us at the moment, how we can investigate the decay of the magnetic field in the $L^2$ norm. Finally, the proof of the limit as $c \to \infty$ is similar to that of Theorem \ref{theo_global_small}-$(ii)$.
	
	$\bullet$ Finally, it turns out to the proofs of Theorems \ref{theo-MH} and \ref{theo-C-HMHD}, where the motivation mainly comes from Theorem \ref{theo_global_small}-$(iii)$ in which \eqref{NSM-GO} converges weakly to \eqref{HMHD} with the magnetic resistivity $\frac{1}{\sigma}$ (or \eqref{NSM} converges weakly to \eqref{MHD} in \cite{Arsenio-Gallagher_2020,KLN_2024_2}) as the speed of light tends to infinity. It is well-known that (in bounded domain or on the three-dimensional torus $\mathbb{T}^3$) if a strong solutions to the ideal \eqref{MHD} (i.e., with $\nu = 0$ and $\sigma = \infty$) exists then it conserves the total energy, the cross helicity and also the magnetic helicity. Therefore, it is also natural to investigate similar things when taking the limit as $\frac{1}{\sigma} \to 0$ or equivalently $\sigma \to \infty$ to either \eqref{NSM} (see \cite{KLN_2024_2}) or \eqref{NSM-GO}. Here, there is connection with the well-known Taylor's conjecture, which is on the magnetic helicity conservation of the limit $(v,b)$ of the corresponding sequence of Leray-Hopf weak solutions $(v^{\nu,\sigma},B^{\nu,\sigma})$ to \eqref{MHD} with $\alpha = \beta = 1$ by investigating the limits as $\nu \to 0$ and $\sigma \to \infty$, which is recently solved in \cite{Faraco-Lindberg_2020,Faraco-Lindberg-MacTaggart-Valli_2022} in bounded domains. See also \cite{Aluie-Eyink_2010,Beekie-Buckmaster-Vicol_2020,Caflisch-Klapper-Steele_1997,Faraco-Lindberg_2018,Kang-Lee_2007} for conditional magnetic helicity conservation for the ideal \eqref{MHD} either in $\mathbb{R}^3$ or on $\mathbb{T}^3$. For the history and further discussions of the conjecture, we refer the reader to recent developments in \cite{Faraco-Lindberg-Szekelyhidi_2022,Faraco-Lindberg-Szekelyhidi_2021,Faraco-Lindberg-Szekelyhidi_2024}. This conjecture formally says that weak ideal limits of Leray-Hopf solutions to \eqref{MHD} (with $\nu, \sigma > 0$) conserve the magnetic helicity when taking the limit as the viscosity  $\nu$ and the magnetic resistivity $\frac{1}{\sigma}$ both go to zero. We note that there is also another version in which the viscosity is fixed and only taking the limit as the magnetic resistivity tends to zero, it is known as the weak
	non-resistive limit (see \cite{Faraco-Lindberg_2020}). Therefore, we will focus on the latter version of the conjecture to \eqref{NSM-GO} and \eqref{HMHD} under suitable assumptions. On the other hand, it can be seen that the energy estimate, which is the only one, at the momment, that it does not depend on $\sigma$, except the $L^2_{t,x}$ norm of $j^{\sigma}$, but  a multiplication of $\frac{1}{\sigma}$ of this norm does. In addition, in the case of either in bounded domains with either homogeneous Dirichlet or no-slip
	and perfect conductivity (see \cite{Faraco-Lindberg_2020,Faraco-Lindberg-MacTaggart-Valli_2022}) boundary conditions, or on the three-dimensional torus with the usual zero mean property, where we usually have a Poincar\'{e} inequality, which ensures the finiteness of the magnetic helicity.  However, it is different in the case of the whole sapce $\mathbb{R}^3$, we need to assume an additional conditions on the initial data such as $(E^{\sigma},B^{\sigma}) \in \dot{H}^{-1}$ (in fact, without this additional negative regularity condition, an example on the ill-definedness of the magnetic helicity in three dimensions was established in \cite{Faraco-Lindberg-Szekelyhidi_2021}) and a corresponding strong convergence of the magnetic initial data from $B^{\sigma}_0$ to $B_0$ in $\dot{H}^{-1}$ as $\sigma \to \infty$ for some suitable $B_0$. It allows us to bound the magnetic helicity in a suitable way. We then use the form of the equations of $B^{\sigma}$ to control the magnetic helicity in a suitable way in terms of good quantities, where we can use the $\sigma$-independent property of the energy estimate to pass to the limit as $\sigma \to \infty$. We note here that for \eqref{NSM-GO}, the GWP in Theorem \ref{theo_global_small}-$(i)$ needs the smallness assumption of the initial data, see \eqref{small_data}, which depends on $\sigma$ in a bab way (it can be seen during the proof). A similar thing happens in Theorem \ref{theo_local}-$(i)$, where the local time existence depends on $\sigma$. Therefore, in both cases, it it not clear to us how we can study the limit as $\sigma \to \infty$. Thus, in Theorem \ref{theo-MH}-$(ii)$, we need to assume that global solutions as in Theorem \ref{theo-MH}-$(i)$ exist with similar properties, but without assuming Condition \eqref{small_data}. After that we can follow the idea in \cite{KLN_2024_2} as in the case of \eqref{NSM} to obtain the conservation of the magnetic helicity. The same situation happens to \eqref{HMHD}, where we can only work with the critical fractional Laplacian exponents $(\alpha_c,\beta_c)$ for both the velocity and magnetic fields in Theorem \ref{theo-C-HMHD}. This choice allows us to obtain the GWP for large initial data and pass to the limit as $\sigma \to \infty$ similarly as in the proof of Theorem \ref{theo-MH}-$(ii)$.
	
	%
	\section{Proof of Theorem \ref{theo_global_small}} \label{sec:global_small}
	%
	
	In this section, we will provide a proof of Theorem	 \ref{theo_global_small}.
	 
	\begin{proof}[Proof of Theorem \ref{theo_global_small}-$(i)$] The proof will be divided into the following several steps. 
		
		\textbf{Step 1: Approximate system, local and global existence.} For each $n \in \mathbb{N}$, we will consider \eqref{NSM-GO_na} with $\alpha = 1$ as follows
		\begin{equation} \label{NSM-GO_n} 
			\frac{d}{dt} \Gamma^n = F^n(\Gamma^n) \qquad \text{and} \qquad \Gamma^n_{|_{t=0}} = \mathcal{T}_{\overline{n}}(\Gamma_0),
		\end{equation}
		where $\Gamma^n := (v^n,E^n,B^n)$, $F^n := (F^n_1,F^n_2,F^n_3)$ with 
		\begin{align*}
			F^n_1 &:=  -\mathbb{P}\mathcal{T}_{\overline{n}}(v^{n-1} \cdot \nabla v^n) + \nu \Delta v^n + \mathbb{P}\mathcal{T}_{\overline{n}}(j^n \times B^{n-1}),
			\\
			F^n_2 &:=	c(\nabla \times B^n - j^n)
			\quad \text{and} \quad F^n_3 := -c\nabla \times E^n.	
		\end{align*} 
		In the sequel, we aim to use an induction argument to prove the following statement: For each $n \in \mathbb{N}$, \eqref{NSM-GO_n} has a global solution $\Gamma^n \in C^1([0,\infty);X^s_n)$, where $X^s_n$ will be defined below, such that for $t > 0$, $s'' \in [0,s]$, $s' \in [1,s]$, $c \geq 1$ and $C = C(d,\kappa,\nu,\sigma,s)$
		\begin{equation} \label{IH}
			\esssup_{\tau \in (0,t)} \|\Gamma^{n}(\tau)\|^2_{X^s} +  \int^t_0 \|\nabla v^{n}\|^2_{H^{s-1}} + \|v^n\|^2_{L^\infty} + \|E^n\|^2_{\dot{H}^{s''}} + \|B^n\|^2_{\dot{H}^{s'}} + \|j^{n}\|^2_{H^s} \,d\tau \leq C \|\Gamma_0\|^2_{X^s}.
		\end{equation}
		In order to do that, we first define the following functional spaces for $s_0 \in \mathbb{R}$ with $s_0 \geq 0$ and for $n \in \mathbb{N}$
		\begin{align*}
			H^{s_0}_n := \left\{f \in H^{s_0}(\mathbb{R}^d) : \text{supp}(\mathcal{F}(f)) \subseteq B_{\overline{n}}\right\} \qquad \text{and} \qquad H^{s_0}_{n,0} := \left\{f \in H^{s_0}_n : \text{div}\, f = 0\right\},
		\end{align*}
		which are equipped with the usual $H^{s_0}$ norm,\footnote{Here, we use the usual notation for  norms in nonhomogeneous Sobolev spaces, i.e., for $s \in \mathbb{R}$, $\|f\|_{H^s} := \|(1+|\xi|^2)^\frac{s}{2}\mathcal{F}(f)\|_{L^2}$. For $s \geq 0$,  $\|f\|^2_{H^s} \approx \|f\|^2_{L^2} + \|f\|^2_{\dot{H}^s}$ with $\|f\|_{\dot{H}^s} := \|\Lambda^s f\|_{L^2}$, where the fractional derivative operator for $s \in \mathbb{R}$, $\Lambda^s := (-\Delta)^\frac{s}{2}$ is defined via Fourier transform by $\mathcal{F}(\Lambda^s f)(\xi) := |\xi|^{s}\mathcal{F}(f)(\xi)$.}  and the following mapping for $s > \frac{d}{2}$
		\begin{equation*}
			F^{n} : X^s_n := H^{s-1}_{n,0} \times H^s_{n,0} \times H^s_{n,0} \to X^s_n
			\qquad \text{with}\qquad 
			\Gamma^{n} \mapsto F^{n}(\Gamma^{n}),
		\end{equation*}
		where the space $X^s_n$ is equipped with the following natural norm
		\begin{equation*}
			\|(f_1,f_2,f_3)\|^2_{X^s_n} := \|f_1\|^2_{H^{s-1}} + \|(f_2,f_3)\|^2_{H^s}.
		\end{equation*}
		Similarly, we also denote the following spaces and their corresponding norms
		\begin{align*}
			X^s &:= H^{s-1} \times H^s \times H^s \qquad \text{with} \qquad \|(f_1,f_2,f_3)\|^2_{X^s} := \|f_1\|^2_{H^{s-1}} + \|(f_2,f_3)\|^2_{H^s},
			\\
			\dot{X}^s &:= \dot{H}^{s-1} \times \dot{H}^s \times \dot{H}^s \qquad \text{with} \qquad \|(f_1,f_2,f_3)\|^2_{\dot{X}^s} := \|f_1\|^2_{\dot{H}^{s-1}} + \|(f_2,f_3)\|^2_{\dot{H}^s}.
		\end{align*}
		In the sequel, at each level $n \in \mathbb{N}$, we aim to provide first the local existence and then the global one to the approximate system \eqref{NSM-GO_n}. We first focus on the level $n = 1$ as follows.
		
		\textbf{Step 1a: The case $n = 1$.} In this case,
		it is not difficult to check that $F^1$ is well-defined and is locally Lipschitz on $X^s_1$. Thus, the usual Cauchy-Lipschitz theorem (see \cite{Majda_Bertozzi_2002}) gives us the existence and uniqueness of solution $\Gamma^1 \in C^1([0,T^1_*);X^s_1)$ for some $T^1_* > 0$ with the following property if $T^1_* < \infty$ then
		\begin{equation*}
			\lim_{t \to T^1_*} \left(\|v^1(t)\|^2_{H^{s-1}} + \|(E^1,B^1)(t)\|^2_{H^s}\right) = \infty.
		\end{equation*}
		Moreover, it can be seen that the $X^s$ estimate of \eqref{NSM-GO_n} with $n = 1$ is given by for $t \in (0,T^1_*)$
		\begin{equation} \label{Gamma_1}
			\|\Gamma^1(t)\|^2_{X^s} + \int^t_0 \nu \|\nabla v^1\|^2_{H^{s-1}} + \frac{1}{\sigma} \|j^1\|^2_{H^s} \,d\tau \leq \|\Gamma_0\|^2_{X^s},
		\end{equation}
		which yields $T^1_* = \infty$. We are going to bound the norms $\|(E,B)\|_{L^2(0,t;\dot{H}^1)}$ and $\|v^1\|_{L^2(0,t;L^\infty)}$ for any $t > 0$, since these bounds will be used in the next substep. It can be seen that by applying $\Lambda^{s-1}$ both sides of the Maxwell equations in \eqref{NSM-GO_n} with $n = 1$, testing the result one by  $-(\Lambda^{s'-1}\nabla \times B^{1},\Lambda^{s'-1}\nabla \times E^{1})$ for $s' \in [1,s]$ and integrating in time for $\tau > 0$
		\begin{align*}
			\int^{\tau}_0 \|\Lambda^{s'-1} \nabla \times B^1\|^2_{L^2} \,dt 
			&= \frac{1}{c} \int^{\tau}_0 \frac{d}{dt} \int_{\mathbb{R}^d} \Lambda^{s'-1} E^1 \cdot \Lambda^{s'-1} \nabla \times B^1 \,dxdt 
			\\
			&\quad +  \|\Lambda^{s'-1} \nabla \times E^1\|^2_{L^2(0,t;L^2)} 
			+ \int^{\tau}_0 \int_{\mathbb{R}^d} \Lambda^{s'-1} j^1 \cdot \Lambda^{s'-1} \nabla \times B^1 \,dxdt,
		\end{align*} 
		which by using \eqref{Gamma_1} and \eqref{E1_L2Hs} below, the facts $c \geq 1$ and $\textnormal{div}\, B^1 = \textnormal{div}\, E^1 = 0$, yields for $t > 0$
		\begin{equation} \label{B1_L2Hs}
			\|B^1\|^2_{L^2(0,t;\dot{H}^{s'})} 
			\leq C(\sigma) \|\Gamma_0\|^2_{X^s} + 2 \|E^1\|^2_{L^2(0,t;\dot{H}^{s'})} \leq C(\sigma) \|\Gamma_0\|^2_{X^s}.
		\end{equation}
		It remains now to bound the norm $\|E^{1}\|_{L^2(0,t;\dot{H}^s)}$ in terms of initial data as follows for $t > 0$. It follows from the Maxwell system in \eqref{NSM-GO_n} with $n = 1$ again that for $t > 0$ and $s'' \in [0,s]$
		\begin{equation} \label{E1_L2Hs}
			\|(E^1,B^1)(t)\|^2_{\dot{H}^{s''}} + 2c^2\sigma \|E^1\|^2_{L^2(0,t;\dot{H}^{s''})} = \|(E^1, B^1)(0)\|^2_{\dot{H}^{s''}} \leq \|\Gamma_0\|^2_{X^s}.
		\end{equation}
		Thus, it follows from the heat equation of $v^1$ in \eqref{NSM-GO_n} with $n = 1$ and \cite[Proposition 2.20 and Theorem 2.34]{Bahouri-Chemin-Danchin_2011} for $d \in \{2,3\}$, i.e.,
		\begin{equation*}
			\dot{H}^{\frac{d}{2}-1} \hookrightarrow \dot{B}^{-1}_{\infty,2} \qquad \text{and} \qquad \left\|\exp\{\nu t(-\Delta)\} v^{2,1}(0)\right\|_{L^2(0,\infty;L^\infty)} \approx \|v^{2,1}(0)\|_{\dot{B}^{-1}_{\infty,2}},
		\end{equation*}
		that for $t > 0$
		\begin{equation} \label{v1_L2Linfty}
			\|v^{1}\|_{L^2(0,t;L^\infty)} = \|\exp\{\nu \tau (-\Delta)\} v^{1}(0)\|_{L^2(0,t;L^\infty)} 
			\leq C(d,\nu) \|v_0\|_{L^2}.
		\end{equation}
		Finally, we summarize \eqref{Gamma_1}-\eqref{v1_L2Linfty} as follows for $t > 0$, $s' \in [1,s]$, $s'' \in [0,s]$ and $C = C(d,\nu,\sigma)$
		\begin{align} \label{Gamma_1_2}
			\|\Gamma^1(t)\|^2_{X^s} + \int^t_0 \|\nabla v^1\|^2_{H^{s-1}} + \|v^1\|^2_{L^\infty} + \|E^1\|^2_{\dot{H}^{s''}} + \|B^1\|^2_{\dot{H}^{s'}} 
			+ \|j^1\|^2_{H^s} \,d\tau \leq C\|\Gamma_0\|^2_{X^s},
		\end{align}
		which is \eqref{IH} with $n = 1$.
		
		\textbf{Step 1b: The case $n = 2$.} 
		In the case $n = 2$, 
		similar to Step 1a, we also have the existence and uniqueness of solution $\Gamma^2 \in C^1([0,T^2_*);X^s_2)$ for some $T^2_* > 0$. Furthermore, in this case the $X^s$ estimate of \eqref{NSM-GO_n} with $n = 2$ is given as follows
		\begin{equation*}
			\frac{1}{2}\frac{d}{dt} \|\Gamma^2\|^2_{X^s} + \nu \|\nabla v^2\|^2_{H^{s-1}} + \frac{1}{\sigma}\|j^2\|^2_{H^s} =: \sum^5_{i=1} I^2_i,
		\end{equation*}
		where for some $\epsilon \in (0,1)$ and for $s > \frac{d}{2}$ 
		\begin{align*}
			I^2_1 &= \int_{\mathbb{R}^d} \Lambda^{s-1}(v^1 \cdot \nabla v^2) \cdot \Lambda^{s-1} v^2\,dx;
			\\
			I^2_2 &= \int_{\mathbb{R}^d} \Lambda^{s-1}(j^2 \times B^1) \cdot \Lambda^{s-1} v^2\,dx
			\\
			&\leq 
			\begin{cases}
				\epsilon\nu \|\nabla v^2\|^2_{H^{s-1}} + C(\epsilon,\nu)\|j^2\|^2_{H^s} \|B^1\|^2_{H^s} & \text{if} \quad s > 2,
				\\
				\frac{\epsilon}{\sigma}\|j^2\|^2_{H^s} + C(\epsilon,\sigma,s) \|B^1\|^2_{H^s} \|v^2\|^2_{L^\infty} & \text{if} \quad s \in (\frac{d}{2},2];
			\end{cases}
			\\
			I^2_3 &= \int_{\mathbb{R}^d} \Lambda^s j^2 \cdot \Lambda^s(v^2 \times B^1)\,dx
			\leq C(d,s)\|j^2\|_{\dot{H}^s} \left(\|v^2\|_{\dot{H}^s} \|B^1\|_{L^\infty} + \|v^2\|_{L^\infty}\|B^1\|_{\dot{H}^s}\right)
			\\
			&\leq C(d,s)\|B^1\|_{H^s}\left(\|\nabla v^2\|^2_{\dot{H}^{s-1}} + \|j^2\|^2_{H^s}\right) + \frac{\epsilon}{\sigma}\|j^2\|^2_{H^s} + C(d,\epsilon,\sigma,s)\|v^2\|^2_{L^\infty}\|B^1\|^2_{H^s};
			\\
			I^2_4 &= -\frac{\kappa}{\sigma} \int_{\mathbb{R}^d} \Lambda^s j^2 \cdot  \Lambda^s(j^1 \times B^2) \,dx
			\leq \frac{\epsilon}{\sigma}\|j^2\|^2_{H^s} + C(d,\epsilon,\kappa,\sigma,s)\|j^1\|^2_{H^s} \|B^2\|^2_{H^s};
			\\
			I^2_5 &= -\frac{\kappa}{\sigma} \int_{\mathbb{R}^d} j^2 \cdot (j^1 \times B^2)\,dx
			\leq \frac{\epsilon}{\sigma}\|j^2\|^2_{L^2} + C(d,\epsilon,\kappa,\sigma,s)\|j^1\|^2_{H^s}\|B^2\|^2_{L^2}.
		\end{align*}
		It remains to bound $I^{2}_1$. If $d = 2$ and $s \in (1,2)$ then
		\begin{align*}
			I^{1}_1 
			&\leq 3\epsilon\nu \|\nabla v^{2}\|^2_{\dot{H}^{s-1}} + C(d,\epsilon,\nu,s) \left(\|\nabla v^{1}\|^2_{H^{s-1}} + \|v^{1}\|^2_{L^\infty}\right)\|v^{1}\|^2_{\dot{H}^{s-1}}.
		\end{align*}
		If $d = 2$ and $s \geq 2$ then 
		\begin{align*}
			I^{1}_1 
			&\leq 3\epsilon\nu \|\nabla v^{2}\|^2_{H^{s-1}} + C(d,\epsilon,\nu,s) \left(\|\nabla v^{1}\|^2_{H^{s-1}} + \|v^{1}\|^2_{L^\infty}\right)\|v^{2}\|^2_{\dot{H}^{s-1}}.
		\end{align*}
		If $d = 3$ and $s > \frac{3}{2}$ then
		\begin{align*}
			I^{1}_1 
			&\leq 3\epsilon\nu \|\nabla v^{2}\|^2_{H^{s-1}} + C(d,\epsilon,\nu,s) \left(\|\nabla v^{1}\|^2_{H^{s-1}} + \|v^{1}\|^2_{L^\infty}\right)\|v^{2}\|^2_{\dot{H}^{s-1}}.
		\end{align*}
		Here (and in the sequel without saying), we used some well-known inequalities (see \cite{Bahouri-Chemin-Danchin_2011})
		\begin{align*}
			&&\|f\|_{L^{p_0}(\mathbb{R}^d)} &\leq C(d,p_0,s_0) \|f\|_{\dot{H}^{s_0}(\mathbb{R}^d)}  & &\text{for}\quad s_0 \in \left[0,\frac{d}{2}\right),\, p_0 = \frac{2d}{d-2s_0},&&
			\\
			&&\|f\|_{\dot{H}^{s_1}(\mathbb{R}^d)} &\leq \|f\|^{\alpha_0}_{L^2(\mathbb{R}^d)} \|f\|^{1-\alpha_0}_{\dot{H}^{s_2}(\mathbb{R}^d)} &
			&\text{for}\quad s_1,s_2 \in (0,\infty), s_1 < s_2, \alpha_0 = 1 - \frac{s_1}{s_2},&&
		\end{align*}
		and the following homogeneous Kato-Ponce type inequality (see \cite{Grafakos-Oh_2014})
		for $1 < p_i,q_i \leq \infty$, $i \in \{1,2\}$, $s_0 \in \mathbb{R}, s_0 > 0$ and  $\frac{1}{p_i} + \frac{1}{q_i} = \frac{1}{2}$
		\begin{equation*} 
			\|\Lambda^{s_0}(fg)\|_{L^2(\mathbb{R}^d)} \leq C(d,s_0,p_i,q_i)\left(\|\Lambda^{s_0} f\|_{L^{p_1}(\mathbb{R}^d)}\|g\|_{L^{q_1}(\mathbb{R}^d)} + \|f\|_{L^{p_2}(\mathbb{R}^d)}\|\Lambda^{s_0} g\|_{L^{q_2}(\mathbb{R}^d)} \right).
		\end{equation*}
		By choosing $\epsilon = \frac{1}{8}$ and using \eqref{Gamma_1} under assuming  $C\|\Gamma_0\|^2_{X^s} \leq 1$ for some positive constant $C = C(d,\kappa,\nu,\sigma,s)$, we find that for $t \in (0,T^2_*)$ and for $C = C(d,\kappa,\nu,\sigma,s)$
		\begin{equation} \label{Main_2}
			\esssup_{\tau \in (0,t)} \|\Gamma^2(\tau)\|^2_{X^s} + \int^t_0 \nu \|\nabla v^2\|^2_{H^{s-1}} + \frac{1}{\sigma}\|j^2\|^2_{H^s} \,d\tau \leq C\|\Gamma_0\|^2_{X^s}\left(1 + \int^t_0 \|v^2\|^2_{L^\infty} \,d\tau\right).
		\end{equation}
		It remains to bound the norm $\|v^2\|_{L^2(0,T^2;L^\infty)}$ for $T^2 \in (0,T^2_*)$ in a suitable way. In order to do that, we will separate the cases $d = 2$ and $d = 3$ as follows. 
		
		\textbf{Step 1b-1: The case $d = 2$.} We follow the ideas in \cite{Germain-Ibrahim-Masmoudi_2014,Ibrahim-Keraani_2012} by considering the following velocity decomposition  $v^{2} := v^{2,1} + v^{2,2}$ with $\text{supp}(\mathcal{F}(v^{2,i})) \subseteq B_{\overline{2}}$ for $i \in \{1,2\}$,
		\begin{equation*}
			\left\{
			\begin{aligned}
				\partial_t v^{2,1} - \nu \Delta v^{2,1} &= - \mathbb{P}\mathcal{T}_{\overline{2}}(v^{1} \cdot \nabla v^{2}) + \mathbb{P}\mathcal{T}_{\overline{2}}(\dot{S}_2(\dot{R}(j^{2}, B^{1}))) =: f^{2,1},  
				\\
				\text{div}\, v^{2,1} &= 0, 
				\quad 
				v^{2,1}_{|_{t=0}} = v^{2}_{|_{t=0}},
			\end{aligned}
			\right.
		\end{equation*}
		and 
		\begin{equation*}
			\left\{
			\begin{aligned}
				\partial_t v^{2,2} - \nu \Delta v^{2,2} &= \mathbb{P}\mathcal{T}_{\overline{2}}(\dot{T}_{j^{2}} B^{1} + \dot{T}_{B^{1}} j^{2} + (\textnormal{Id}-\dot{S}_2)(\dot{R}(j^{2},B^{1}))) =: f^{2,2},  
				\\
				\text{div} \, v^{2,2} &= 0,
				\quad
				v^{2,2}_{|_{t=0}} = 0,
			\end{aligned}
			\right.
		\end{equation*}
		where $\textnormal{Id}$ denotes the identity operator and we also use the following decomposition
		\begin{equation*}
			j^{2} \times B^{1} := \dot{T}_{j^{2}} B^{1} + \dot{T}_{B^{1}} j^{2} + \dot{R}(j^{2},B^{1}).
		\end{equation*}
		More precisely, by using the definition of the cross product
		\begin{equation*}
			j^{2} \times B^{1} := (j^{2}_2 B^{1}_3 - j^{2}_3 B^{1}_2, j^{2}_3B^{1}_1 - j^{2}_1B^{1}_3,j^{2}_1B^{1}_2 - j^{2}_2B^{1}_1),
		\end{equation*}
		we denoted
		\begin{align} \label{T_j2B1}
			\dot{T}_{j^{2}} B^{1} &:= (\dot{T}_{j^{2}_2} B^{1}_3 - \dot{T}_{j^{2}_3} B^{1}_2,
			\dot{T}_{j^{2}_3} B^{1}_1 - \dot{T}_{j^{2}_1} B^{1}_3, \dot{T}_{j^{2}_1} B^{1}_2 - \dot{T}_{j^{2}_2} B^{1}_1),
			\\ \label{T_B1j2}
			\dot{T}_{B^{1}} j^{2} &:= (\dot{T}_{B^{1}_3} j^{2}_2 - \dot{T}_{B^{1}_2} j^{2}_3,\dot{T}_{B^{1}_1} j^{2}_3 - \dot{T}_{B^{1}_3} j^{2}_1,\dot{T}_{B^{1}_2} j^{2}_1 - \dot{T}_{B^{1}} j^{2}_2),
			\\ \label{R_j2B1}
			\dot{R}(j^{2},B^{1}) &:= (\dot{R}(j^{2}_2,B^{1}_3) - \dot{R}(j^{2}_3,B^{1}_2),\dot{R}(j^{2}_3,B^{1}_1) - \dot{R}(j^{2}_1,B^{1}_3), \dot{R}(j^{2}_1,B^{1}_2) -\dot{R}(j^{2}_2,B^{1}_1)),
		\end{align}
		where the definitions of the homogeneous Bony decomposition and $\dot{S}_2$ can be found in \cite{Bahouri_2019,Bahouri-Chemin-Danchin_2011}. In the sequel, we aim to estimate the norm $\|v^{2}\|_{L^2(0,T^{2};L^\infty)}$ in terms of the ones $\|v^{2,1}\|_{L^2(0,T^{2};L^\infty)}$ and $\|v^{2,2}\|_{L^2(0,T^{2};L^\infty)}$ for $T^2 \in (0,T^2_*)$. By using Duhamel formula, we can represent $v^{2,1}$ as follows for $t \in (0,T^{2}_*)$
		\begin{align} \label{R_v2_1}
			v^{2,1}(t) &= \exp\{\nu t (-\Delta)\} v^{2,1}(0) + \int^t_0 \exp\{\nu(t-\tau)(-\Delta)\} f^{2,1}(\tau) \,d\tau,
		\end{align}
		where $\exp\{\nu t(-\Delta)\}$ is the usual heat kernel and is defined in Appendix. Therefore, \eqref{R_v2_1} implies that for $T^2 \in (0,T^2_*)$
		\begin{equation} \label{v2_1_L2Linfty}
			\|v^{2,1}(t)\|_{L^2(0,T^{2};L^\infty)} \leq I^{2,1}_1 + I^{2,1}_2.
		\end{equation} 
		Similar to \eqref{v1_L2Linfty}, we have the following bound for the first term
		\begin{equation} \label{I_21_1}
			I^{2,1}_1 := \left\|\exp\{\nu t (-\Delta)\} v^{2,1}(0)\right\|_{L^2(0,T^{2};L^\infty)} \leq C(d,\nu) \|v^{2,1}(0)\|_{\dot{B}^{-1}_{\infty,2}} 
			\leq C(d,\nu)\|v_0\|_{L^2}.
		\end{equation}
		While the second term is bounded by (see \cite{Chemin-Gallagher_2006,Germain-Ibrahim-Masmoudi_2014})
		\begin{align} \label{I_21_2}
			I^{2,1}_2 &:= \left\|\int^t_0 \exp\{\nu(t-\tau)(-\Delta)\} f^{2,1}(\tau) \,d\tau \right\|_{L^2(0,T^{2};L^\infty)} \leq C(d,\nu) \|f^{2,1}\|_{L^1(0,T^{2};\dot{H}^{\frac{d}{2}-1})} \\ \nonumber
			&\leq I^{2,1}_{21} + I^{2,1}_{22}.
		\end{align}
		In two dimensions, the right-hand side of \eqref{I_21_2} can be dominated by 
		\begin{align} 
			\label{I_21_21}
			I^{2,1}_{21} &:= C(d,\nu)\|v^{1} \cdot \nabla v^{2}\|_{L^1(0,T^{2};L^2)} \leq C(d,\nu) \|v^{1}\|_{L^2(0,T^{2};L^\infty)} \|\nabla v^{2}\|_{L^2(0,T^{2};L^2)};
			\\ 	\label{I_21_22}
			I^{2,1}_{22} &:= C(d,\nu)\|\dot{S}_2(\dot{R}(j^{2},B^{1}))\|_{L^1(0,T^{2};L^2)} \leq C(d,\nu)\|j^{2}\|_{L^2(0,T^{2};L^2)} \|B^{1}\|_{L^2(0,T^{2};\dot{H}^{1,0})}, 
		\end{align}
		where in the estimate of $I^{2,1}_{22}$, we also used Lemma \ref{lem_paraproduct}, which is applied to each component in \eqref{R_j2B1}. By using the continuous embedding $\dot{H}^1 = \dot{H}^{1,1}  \hookrightarrow \dot{H}^{1,0}$ (see Appendix), \eqref{B1_L2Hs}, \eqref{v1_L2Linfty} and \eqref{v2_1_L2Linfty}-\eqref{I_21_22}
		\begin{equation} \label{v_21_L2Linfty_2}
			\|v^{2,1}\|^2_{L^2(0,T^{2};L^\infty)} \leq C(d,\nu,\sigma)\|\Gamma_0\|^2_{H^s}\left(1 +  \|\nabla v^{2}\|^2_{L^2(0,T^{2};L^2)} + \|j^{2}\|^2_{L^2(0,T^{2};L^2)}\right).
		\end{equation}
		We continue with the estimate of $v^{2,2}$ as follows. By using the continuous embedding $\dot{B}^\frac{d}{2}_{2,1} \hookrightarrow L^\infty$ for $d \in \{2,3\}$ (see \cite[Proposition 2.39]{Bahouri-Chemin-Danchin_2011}) and Lemma \ref{pro-F_heat}, which is applied to the heat equation of $v^{2,2}$ with $v^{2,2}_{|_{t=0}} = 0$, we find that 
		\begin{equation} \label{v_22_L2Linfty}
			\|v^{2,2}\|_{L^2(0,T^{2};L^\infty)} \leq C \|v^{2,2}\|_{L^2(0,T^{2};\dot{B}^1_{2,1})} \leq C(\nu)\|f^{2,2}\|_{L^2(0,T^{2};\dot{B}^{-1}_{2,1})}.
		\end{equation}
		Furthermore, we employ Lemma \ref{lem_paraproduct} again, which is applied to each component in \eqref{T_B1j2}-\eqref{R_j2B1}, to obtain the following estimates
		\begin{align} \label{para_21_1}
			\|\dot{T}_{j^{2}} B^{1}\|_{L^2(0,T^{2};\dot{B}^{-1}_{2,1})} &\leq C\|B^{1}\|_{L^\infty(0,T^{2};L^2)} \|j^{2}\|_{L^2(0,T^{2};L^2)};
			\\ \label{para_21_2}
			\|\dot{T}_{B^{1}} j^{2}\|_{L^2(0,T^{2};\dot{B}^{-1}_{2,1})} 
			&\leq C\|B^{1}\|_{L^\infty(0,T^{2};L^2)} \|j^{2}\|_{L^2(0,T^{2};L^2)};
			\\ \label{para_21_3}
			\|(\textnormal{Id}-\dot{S}_2)(\dot{R}(j^{2},B^{1}))\|_{L^2(0,T^{2};\dot{B}^{-1}_{2,1})}
			&\leq C \|B^{1}\|_{L^\infty(0,T^{2};L^2_{\textnormal{log}})} \|j^{2}\|_{L^2(0,T^{2};L^2_{\textnormal{log}})}.
		\end{align}
		Collecting \eqref{v_22_L2Linfty}-\eqref{para_21_3} and using \eqref{Gamma_1_2}, which yield
		\begin{equation} \label{v_2m2_L2Linfty_1}
			\|v^{2,2}\|^2_{L^2(0,T^{2};L^\infty)} \leq  C(d,\nu,\sigma)\|\Gamma_0\|^2_{X^s} \|j^{2}\|^2_{L^2(0,T^{2};H^s)},
		\end{equation}
		that combines with \eqref{v_21_L2Linfty_2} leading to for $T^2 \in (0,T^2_*)$
		\begin{equation} \label{v2_L2Linfty}
			\|v^{2}\|^2_{L^2(0,T^{2};L^\infty)} \leq C(d,\nu,\sigma)\|\Gamma_0\|^2_{X^s}\left(1 + \|\nabla v^{2}\|^2_{L^2(0,T^{2};L^2)} +  \|j^{2}\|^2_{L^2(0,T^{2};H^s)}\right).
		\end{equation}
		
		\textbf{Step 1b-2: The case $d = 3$.} Similar to the case $d = 2$, we decompose $v^{2} := v^{2,1} + v^{2,2}$ with $\text{supp}(\mathcal{F}(v^{2,i})) \subseteq B_{\overline{2}}$ for $i \in \{1,2\}$,
		\begin{align*}
			\partial_t v^{2,1} - \nu \Delta v^{2,1} &= \mathbb{P}\mathcal{T}_{\overline{2}}(-v^{1} \cdot \nabla v^{2}),  
			\quad 
			\text{div}\, v^{2,1} = 0, 
			\quad
			v^{2,1}_{|_{t=0}} = v^{2}_{|_{t=0}},
			\\
			\partial_t v^{2,2} - \nu \Delta v^{2,2} &= \mathbb{P}\mathcal{T}_{\overline{2}}(j^{2} \times B^{1}),  
			\qquad
			\text{div} \, v^{2,2} = 0,
			\quad
			v^{2,2}_{|_{t=0}} = 0.
		\end{align*}
		Similar to the previous case, i.e., as in \eqref{R_v2_1}-\eqref{I_21_2} and \eqref{v_22_L2Linfty}
		\begin{align}
			\nonumber
			\|v^{2,1}\|^2_{L^2(0,T^{2};L^\infty)} 
			&\leq C(d,\nu)\|v_0\|^2_{\dot{H}^{\frac{1}{2}}} + C(\nu)\|v^{1}\|^2_{L^2(0,T^{2};L^\infty)} \|v^{2}\|^2_{L^2(0,T^{2};\dot{H}^\frac{3}{2})} 
			\\ \label{v_21_L2Linfty_3}
			&\quad + C(d,\nu)\|v^{1}\|^2_{L^2(0,T^{2};\dot{H}^\frac{3}{2})} \|v^{2}\|^2_{L^2(0,T^{2};L^\infty)};
			\\ \label{v_22_L2Linfty_2} 
			\|v^{2,2}\|^2_{L^2(0,T^{2};L^\infty)} 
			&\leq C(d,\nu) \|j^{2}\|^2_{L^2(0,T^{2};\dot{H}^\frac{1}{2})} \|B^{1}\|^2_{L^\infty(0,T^{2};\dot{H}^\frac{1}{2})},
		\end{align}
		where we also used \cite[Corollary 2.55]{Bahouri-Chemin-Danchin_2011} 
		in the last inequality. 
		
		\textbf{Step 1b-3: Conclusion.} Thus, under suitable smallness assumption on the initial data, we find from \eqref{Gamma_1} and  \eqref{v2_L2Linfty}-\eqref{v_22_L2Linfty_2}
		\begin{equation} \label{v2_L2Linfty_1}
			\|v^{2}(t)\|^2_{L^2(0,T^{2};L^\infty)} \leq C(d,\nu,\sigma)\|\Gamma_0\|^2_{X^s}\left(1 + \|\nabla v^{2}\|^2_{L^2(0,T^{2};H^{s-1})} + \|j^{2}\|^2_{L^2(0,T^{2};H^s)}\right),
		\end{equation}
		which by using \eqref{Main_2} implies that for $t \in (0,T^2_*)$
		\begin{equation} \label{Main_2_1}
			\esssup_{\tau \in (0,t)} \|\Gamma^{2}(\tau)\|^2_{X^s} +  \int^t_0 \nu \|\nabla v^{2}\|^2_{H^{s-1}} + \frac{1}{\sigma}\|j^{2}\|^2_{H^s} \,d\tau \leq C(d,\kappa,\nu,\sigma,s) \|\Gamma_0\|^2_{X^s}.
		\end{equation}
		It implies that $T^{2}_* = \infty$ as well. In addition, it can be seen from \eqref{v2_L2Linfty}, \eqref{v2_L2Linfty_1} and \eqref{Main_2_1} that
		\begin{equation} \label{v2_L2Linfty_2}
			\|v^{2}\|_{L^2(0,\infty;L^\infty)} \leq C(d,\kappa,\nu,\sigma,s) \|\Gamma_0\|_{X^s}.
		\end{equation}
		Furthermore, by using the Maxwell system  in \eqref{NSM-GO_n} with $n = 2$, similar to Step 1a, it follows that for $s' \in [1,s]$ and $\tau \in (0,\infty)$
		\begin{align*}
			\int^{\tau}_0 \|\Lambda^{s'-1} \nabla \times B^2\|^2_{L^2} \,dt 
			&= \frac{1}{c} \int^{\tau}_0 \frac{d}{dt} \int_{\mathbb{R}^d} \Lambda^{s'-1} E^2 \cdot \Lambda^{s'-1} \nabla \times B^2 \,dxdt 
			\\
			&\quad +  \|\Lambda^{s'-1} \nabla \times E^2\|^2_{L^2(0,t;L^2)} 
			+ \int^{\tau}_0 \int_{\mathbb{R}^d} \Lambda^{s'-1} j^2 \cdot \Lambda^{s'-1}\nabla \times B^2 \,dxdt.
		\end{align*} 
		In addition, as in \eqref{E1_L2Hs} with using
		\begin{equation*}
			j^{2} = -\kappa \mathbb{P}\mathcal{T}_{\overline{2}}(j^{1} \times B^{2}) + \sigma(cE^{2} + \mathbb{P}\mathcal{T}_{\overline{2}}(v^{2} \times B^{1})),
		\end{equation*}
		it follows from \eqref{Gamma_1}, \eqref{Main_2_1}-\eqref{v2_L2Linfty_2} that for $t \in (0,\infty)$ and $s'' \in [0,s]$
		\begin{equation} \label{E2_L2Hs}
			\|(E^2,B^2)(t)\|^2_{\dot{H}^{s''}} + c^2\sigma \|E^2\|^2_{L^2(0,t;\dot{H}^{s''})} \leq C(d,\kappa,\nu,\sigma,s) \|\Gamma_0\|^2_{X^s},
		\end{equation}
		which as in \eqref{B1_L2Hs} with using the facts that $\textnormal{div}\, B^2 = \textnormal{div}\, E^2 = 0$ and $c \geq 1$, further implies that for $s' \in [1,s]$
		\begin{equation} \label{B2_L2Hs}
			\|B^2\|_{L^2(0,\infty;\dot{H}^{s'})} \leq C(d,\kappa,\nu,\sigma,s) \|\Gamma_0\|_{X^s}. 
		\end{equation}
		Finally, it follows from \eqref{Main_2} and \eqref{v2_L2Linfty_2}-\eqref{B2_L2Hs} that for $t > 0$, $s' \in [1,s]$, $s'' \in [0,s]$ and for $C = C(d,\kappa,\nu,\sigma,s) $
		\begin{equation} \label{Main_2_2}
			\esssup_{\tau \in (0,t)} \|\Gamma^{2}(\tau)\|^2_{X^s} +  \int^t_0 \|\nabla v^{2}\|^2_{H^{s-1}} + \|v^2\|^2_{L^\infty} + \|E^2\|^2_{\dot{H}^{s''}} + \|B^2\|^2_{\dot{H}^{s'}} + \|j^{2}\|^2_{H^s} \,d\tau \leq C\|\Gamma_0\|^2_{X^s},
		\end{equation}
		which is \eqref{IH} with $n = 2$.
		
		\textbf{Step 1c: The case $n > 2$ with an induction argument.} In the sequel, we aim to use an induction argument. More precisely, we obtained \eqref{Main_2_2}, which is \eqref{IH} with $n = 2$ and $T^{2}_* = \infty$. Now, by assuming \eqref{IH} holds for $n = k > 2$ with $T^{k}_* = \infty$, i.e., for $t > 0$ and for some positive $k$-independent constant $C = C(d,\kappa,\nu,\sigma,s)$, $s' \in [1,s]$ and $s'' \in [0,s]$
		\begin{equation} \label{E_km_induction}
			\esssup_{\tau \in (0,t)} \|\Gamma^{k}(\tau)\|^2_{X^s} +  \int^t_0 \|\nabla v^{k}\|^2_{H^{s-1}} + \|v^k\|^2_{L^\infty} + \|E^k\|^2_{\dot{H}^{s''}} + \|B^k\|^2_{\dot{H}^{s'}} + \|j^{k}\|^2_{H^s} \,d\tau \leq C\|\Gamma_0\|^2_{X^s},
		\end{equation}
		where $\Gamma^{k}$ and $j^{k}$ satisfy \eqref{NSM-GO_n} with $n = k$. We then will prove that  \eqref{IH} is also satisfied for $n = k+1$ and it leads to $T^{k+1}_* = \infty$ as well, under further smallness assumptions of the initial data and using \eqref{E_km_induction}. 
		As in the previous parts, we have the local existence and uniqueness of solutions $\Gamma^{k+1} \in C^1([0,T^{k+1}_*);X^s_{k+1})$ for some $T^{k+1}_* > 0$. If the initial data is small suitably, we can use \eqref{E_km_induction} and repeat Step 1c by considering the pair $(k+1,k)$, which plays a role as the one $(2,1)$, i.e., $(\Gamma^{2},\Gamma^{1},j^{2},j^{1})$ and \eqref{NSM-GO_n} with $n = 2$ can be replaced by $(\Gamma^{k+1},\Gamma^{k},j^{k+1},j^{k})$ and 
		\eqref{NSM-GO_n} with $n = k+1$
		, respectively, to obtain \eqref{IH} with $n = k+1$, i.e., for $t \in (0,T^{k+1}_*)$
		\begin{align} \nonumber
			\esssup_{\tau \in (0,t)} \|\Gamma^{k+1}(\tau)\|^2_{X^s} +  &\int^t_0 \|\nabla v^{k+1}\|^2_{H^{s-1}} + \|v^{k+1}\|^2_{L^\infty} + \|E^{k+1}\|^2_{\dot{H}^{s''}} \,d\tau 
			\\ \label{E_k+1_induction}
			&\quad + \int^t_0  \|B^{k+1}\|^2_{\dot{H}^{s'}} + \|j^{k+1}\|^2_{H^s} \,d\tau 
			\leq C\|\Gamma_0\|^2_{X^s}.
		\end{align}
		We omit further details. Thus, $T^{k+1}_* = \infty$ and there exists a positive $n$-independent constant $C = C(d,\kappa,\nu,\sigma,s)$  such that for $n \in \mathbb{N}$, $t > 0$, $s' \in [1,s]$ and $s'' \in [0,s]$
		\begin{equation} \label{En_global}
			\esssup_{\tau \in (0,t)} \|\Gamma^{n}(\tau)\|^2_{X^s} +  \int^t_0 \|\nabla v^{n}\|^2_{H^{s-1}} + \|v^n\|^2_{L^\infty} + \|E^n\|^2_{\dot{H}^{s''}} + \|B^n\|^2_{\dot{H}^{s'}} + \|j^{n}\|^2_{H^s} \,d\tau \leq C \|\Gamma_0\|^2_{X^s}.
		\end{equation}
		
		\textbf{Step 2: Cauchy sequence and passsing to the limit as $n \to \infty$.} Firstly, we aim to prove that $\Gamma^{n}$ and $(\nabla v^{n},v^n,E^n,E^n,B^n,j^{n})$ are Cauchy sequences in $L^\infty(0,\infty;H^{r-1} \times H^r \times H^r)$ and $L^2(0,\infty;H^{r-1} \times L^\infty \times \dot{H}^1 \times L^2 \times \dot{H}^1 \times H^r)$, respectively for $r \in (\frac{d}{2},s)$. Secondly, the Cauchy property will allow us to pass to the limit as $n \to \infty$ in which we expect that \eqref{NSM-GO_n} converges to \eqref{NSM-GO} strongly in some particular sense.

		\textbf{Step 2a: Cauchy sequence.}  For $n \in \mathbb{N}$, we consider two solutions $\Gamma^n = (v^n,E^n,B^n)$ and $\Gamma^{n+1} = (v^{n+1},E^{n+1},B^{n+1})$  with the initial data $\Gamma^n_{|_{t=0}} = \mathcal{T}_{\overline{n}}(\Gamma_0)$ and $\Gamma^{n+1}_{|_{t=0}} = \mathcal{T}_{\overline{n+1}}( \Gamma_0)$, respectively. 
		We denote $f^{n+1,n} := f^{n+1} - f^n$ and $\Gamma^{n+1,n} := (v^{n+1,n},E^{n+1,n},B^{n+1,n})$ for $f \in \{v,E,B,j\}$. 
		
		\textbf{Step 2a-1: The $X^r$ estimate.} Let us fix any $r \in (\frac{d}{2},s)$. The $X^r$ estimate of the difference 
		is given as follows
		\begin{equation} \label{Gamma_n+1n}
			\frac{1}{2}\frac{d}{dt}\|\Gamma^{n+1,n}\|^2_{X^r} + \nu \|\nabla v^{n+1,n}\|^2_{H^{r-1}} + \frac{1}{\sigma}\|j^{n+1,n}\|^2_{H^{r}} =: \sum^{12}_{i=1} \sum_{k \in \{0,r\}}  I^{n+1,n}_{i,k}.
		\end{equation}
		Here (and in the sequel) we will use the following property 
		(see 
		\cite{Fefferman-McCormick-Robinson-Rodrigo_2014}) 
		for $s_1,s_2 \in \mathbb{R}$, $s_2 \geq 0$ 
		\begin{equation} \label{FT}
			\|\mathcal{T}_m f - f\|_{Z^{s_1}} \leq m^{-s_2} \|f\|_{Z^{s_1+s_2}} \qquad \text{for} \quad Z \in \{\dot{H},H\}.
		\end{equation}
		By using \eqref{FT} and direct computations, 
		it follows from \eqref{Gamma_n+1n} that for $t > 0$ and for $C_1 = C_1(d,\kappa,\nu,\sigma,s) > 1$
		\begin{align} \nonumber
			\mathcal{E}^{n+1,n}_r(t) &\leq C_1\max\{\overline{n}^{-2(s-r)},\overline{n}^{-s}\} \|\Gamma_0\|^2_{X^r} + \frac{1}{16} \mathcal{E}^{n,n-1}_r(t) + C_1 \int^t_0 \|v^{n+1,n}\|^2_{L^\infty} \,d\tau 
			\\ \label{Gamma_n+1n_1}
			&\quad +  \int^t_0 \delta_0 \|(E^{n+1,n},B^{n+1,n})\|^2_{\dot{H}^1} + \|E^{n+1,n}\|^2_{L^2} \,d\tau,
		\end{align}
		where we chose the right-hand side of \eqref{En_global} to be small enough and defined the following energy form for $n \in \mathbb{N}$ 
		\begin{align*}
			\mathcal{E}^{n+1,n}_r(t) &:= \esssup_{\tau \in (0,t)}\|\Gamma^{n+1,n}(\tau)\|^2_{X^r} + \int^t_0 \nu \|\nabla v^{n+1,n}\|^2_{H^{r-1}} + \|v^{n+1,n}\|^2_{L^\infty} \,d\tau
			\\
			&\quad + \int^t_0  \delta_0  \|(E^{n+1,n},B^{n+1,n})\|^2_{\dot{H}^1} +  \|E^{n+1,n}\|^2_{L^2} + \frac{1}{\sigma}\|j^{n+1,n}\|^2_{H^r} \,d\tau
		\end{align*}
		for some small constant $\delta_0 \in (0,1)$ to be chosen later. Next, we will bound the norms  $\|v^{n+1,n}\|_{L^2(0,t;L^\infty)}$, $\|(E^{n+1,n},B^{n+1,n})\|_{L^2(0,t;\dot{H}^1)}$ and $\|E^{n+1,n}\|_{L^2(0,t;L^2)}$ in a suitable way in terms of $\mathcal{E}^{n+1,n}_r(t)$ and $\mathcal{E}^{n,n-1}_r(t)$ for $t > 0$, since these bounds are needed in closing the estimate \eqref{Gamma_n+1n_1}.
		
		\textbf{Step 2a-2: Estimate of $\|v^{n+1,n}\|_{L^2_tL^\infty_x}$.} 
		It can be seen that 
		\begin{equation} \label{v_n+1n}
			\partial_t v^{n+1,n} - \nu \Delta v^{n+1,n} = \sum^{6}_{i=1} f^{n+1,n,i}, \qquad v^{n+1,n}_{|_{t=0}} = (\mathcal{T}_{\overline{n+1}}-\mathcal{T}_{\overline{n}})(v_0),
		\end{equation}
		where 
		\begin{align*}
			f^{n+1,n,1} &:= \mathbb{P}(\mathcal{T}_{\overline{n+1}}-\mathcal{T}_{\overline{n}})(-v^{n-1} \cdot \nabla v^n),
			&&f^{n+1,n,2} := \mathbb{P}\mathcal{T}_{\overline{n+1}}(-v^{n,n-1} \cdot \nabla v^{n+1}),
			\\
			f^{n+1,n,3} &:=  \mathbb{P}\mathcal{T}_{\overline{n+1}}(-v^{n-1} \cdot \nabla v^{n+1,n}),
			&&f^{n+1,n,4} :=  \mathbb{P}(\mathcal{T}_{\overline{n+1}}-\mathcal{T}_{\overline{n}})(j^{n}\times B^{n-1}),
			\\
			f^{n+1,n,5} &:=  \mathbb{P}\mathcal{T}_{\overline{n+1}}(j^{n+1,n} \times B^n),
			&&f^{n+1,n,6} := \mathbb{P}\mathcal{T}_{\overline{n+1}}(j^n \times B^{n,n-1}).
		\end{align*}
		
		\textbf{Step2a-2a: The case $d=2$.} In this case, in order to control the norm $\|v^{n+1,n}\|_{L^2_tL^\infty_x}$, we follow the idea as in Step 1b-1. More precisely, we will use the following decomposition 
		\begin{equation*}
			\left\{
			\begin{aligned}
				\partial_t v^{n+1,n,1} - \nu \Delta v^{n+1,n,1} &=  f^{n+1,n,1} + f^{n+1,n,2} + f^{n+1,n,3} + \mathbb{P}(\mathcal{T}_{\overline{n+1}}-\mathcal{T}_{\overline{n}})(\dot{S}_2(\dot{R}(j^n,B^{n-1})))
				\\
				&\quad + \mathbb{P}\mathcal{T}_{\overline{n+1}}(\dot{S}_2(\dot{R}(j^{n+1,n},B^n))) + \mathbb{P}\mathcal{T}_{\overline{n+1}}(\dot{S}_2(\dot{R}(j^n,B^{n,n-1}))),  
				\\
				\text{div}\, v^{n+1,n,1} &= 0, 
				\quad 
				v^{n+1,n,1}_{|_{t=0}} = v^{n+1,n}_{|_{t=0}},
			\end{aligned}
			\right.
		\end{equation*}
		and 
		\begin{equation*}
			\left\{
			\begin{aligned}
				\partial_t v^{n+1,n,2} - \nu \Delta v^{n+1,n,2} &= \mathbb{P}(\mathcal{T}_{\overline{n+1}}-\mathcal{T}_{\overline{n}})(\dot{T}_{j^n} B^{n-1} + \dot{T}_{B^{n-1}} j^n + (\textnormal{Id}-\dot{S}_2)(\dot{R}(j^n,B^{n-1})))
				\\
				&\quad + \mathbb{P}\mathcal{T}_{\overline{n+1}}(\dot{T}_{j^{n+1,n}} B^n + \dot{T}_{B^n} j^{n+1,n} + (\textnormal{Id}-\dot{S}_2)(R(j^{n+1,n},B^n)))
				\\
				&\quad + \mathbb{P}\mathcal{T}_{\overline{n+1}}(\dot{T}_{j^n} B^{n,n-1} + \dot{T}_{B^{n,n-1}} j^n + (\textnormal{Id} -\dot{S}_2)(\dot{R}(j^n,B^{n,n-1}))),  
				\\
				\text{div} \, v^{n+1,n,2} &= 0,
				\quad
				v^{n+1,n,2}_{|_{t=0}} = 0.
			\end{aligned}
			\right.
		\end{equation*}
		Therefore, similar to \eqref{I_21_21}-\eqref{I_21_22} by using Lemma \ref{lem_paraproduct} and \cite[Corollary 2.55]{Bahouri-Chemin-Danchin_2011}, 
		we have the following estimates for $t > 0$, $C = C(d,s)$ and $s' \in (0,1)$ (other terms can be similarly defined and easily bounded) 
		\begin{align*}
			I^{n+1,n,1}_4 &:= \|(\mathcal{T}_{\overline{n+1}}-\mathcal{T}_{\overline{n}})(\dot{S}_2(\dot{R}(j^n,B^{n-1})))\|_{L^1(0,t;L^2)};
			\\
			I^{n+1,n,1}_6 &:= \|\dot{S}_2(\dot{R}(j^n,B^{n,n-1}))\|_{L^1(0,t;L^2)} 
			\leq C \|j^n\|_{L^2(0,t;L^2)}\|B^{n,n-1}\|_{L^2(0,t;\dot{H}^{1,0})};
			\\
			I^{n+1,n,2}_3 &:= \|(\mathcal{T}_{\overline{n+1}}-\mathcal{T}_{\overline{n}})(\textnormal{Id}-\dot{S}_2)(\dot{R}(j^n,B^{n-1}))\|_{L^2(0,t;\dot{B}^{-1}_{2,1})}.
		\end{align*}
		It remains to bound the terms $I^{n+1,n,1}_4$, $\|B^{n,n-1}\|_{L^2(0,t;\dot{H}^{1,0})}$ in the estimate of $I^{n+1,n,1}_6$ and $I^{n+1,n,2}_3$. We first focus on the norm $\|B^{n+1,n}\|_{L^2(0,t;\dot{H}^{1,0})}$ for $t > 0$. By using the definition of $\dot{H}^{1,0}$ (see Appendix), it is enough to bound the norm $\|B^{n+1,n}\|_{L^2(0,t;\dot{H}^1)}$. More precisely, similar to the previous parts, it follows from the Maxwell system 
		for the difference
		\begin{equation*}
			\left\{
			\begin{aligned}
				\frac{1}{c} \partial_t E^{n+1,n} - \nabla \times B^{n+1,n} &= - j^{n+1,n},
				\\
				\frac{1}{c} \partial_t B^{n+1,n} + \nabla \times E^{n+1,n} &= 0,
				\\
				j^{n+1,n} + \kappa \mathbb{P}\mathcal{T}_{\overline{n+1}}(j^{n} \times B^{n+1}) &=  \kappa \mathbb{P}\mathcal{T}_{\overline{n}}(j^{n-1} \times B^n) 
				+ \sigma(cE^{n+1} + \mathbb{P}\mathcal{T}_{\overline{n+1}}(v^{n+1} \times B^{n})) 
				\\
				&\quad - \sigma(cE^n + \mathbb{P}\mathcal{T}_{\overline{n}}(v^n \times B^{n-1})),
			\end{aligned}
			\right.
		\end{equation*}
		by testing $-(\nabla \times B^{n+1,n},\nabla \times E^{n+1,n})$, and using $c \geq 1$ and $\textnormal{div}\, B^{n+1,n} = \textnormal{div}\, E^{n+1,n} = 0$ that for $\tau \in (0,\infty)$ 
		\begin{align} \nonumber
			\|B^{n+1,n}\|^2_{L^2(0,\tau;\dot{H}^1)} 
			&\leq \|E^{n+1,n}(\tau)\|_{L^2} \|B^{n+1,n}(\tau)\|_{\dot{H}^1} + \|E^{n+1,n}(0)\|_{L^2}\|B^{n+1,n}(0)\|_{\dot{H}^1} 
			\\ \label{B_n+1n_L2H1}
			&\quad + \|E^{n+1,n}\|^2_{L^2(0,\tau;\dot{H}^1)} + \|j^{n+1,n}\|_{L^2(0,\tau;L^2)} \|B^{n+1,n}\|_{L^2(0,\tau;\dot{H}^1)}.
		\end{align}
		We now estimate the norm $\|E^{n+1,n}\|_{L^2(0,\tau;\dot{H}^1)}$ for $\tau > 0$. Similarly, by taking the cross product both sides of the above Maxwell system and testing the result system by $(\nabla \times E^{n+1,n},\nabla \times B^{n+1,n})$, it follows that
		\begin{equation*}
			\frac{1}{2}\frac{d}{dt}\|(\nabla \times E^{n+1,n},\nabla \times B^{n+1,n})\|^2_{L^2}  + c^2 \sigma \|\nabla \times E^{n+1,n}\|^2_{L^2} =: \sum^6_{i=1} E^{n+1,n}_i.
		\end{equation*}
		Therefore, by using $\textnormal{div}\, E^{n+1,n} = 0$, for $t > 0$,  $c \geq 1$ and $C = C(d,\kappa,\sigma,s)$
		\begin{align} \nonumber
			\|E^{n+1,n}\|^2_{L^2(0,t;\dot{H}^1)} &\leq C \int^t_0 \|j^{n,n-1}\|^2_{H^r} \|B^{n+1}\|^2_{H^s} + \|j^{n-1}\|^2_{H^s}\|B^{n+1,n}\|^2_{H^r} \,d\tau
			\\ \nonumber
			&\quad + C \int^t_0 \left(\|\nabla v^{n+1,n}\|^2_{H^{r-1}} + \|v^{n+1,n}\|^2_{L^\infty}\right) \|B^n\|^2_{H^s}  \,d\tau
			\\ \nonumber
			&\quad + C \int^t_0  \left(\|\nabla v^n\|^2_{L^2} + \|v^n\|^2_{L^\infty}\right) \|B^{n,n-1}\|^2_{H^r}  \,d\tau 
			\\ \label{E_n+1n_L2H1}
			&\quad + C \overline{n}^{-2(s-1)}\|\Gamma_0\|^2_{X^s}.
		\end{align}
		Similarly, we also find that
		\begin{align} \nonumber
			\|E^{n+1,n}\|^2_{L^2(0,t;L^2)} &\leq C \int^t_0 \|j^{n,n-1}\|^2_{H^r} \|B^{n+1}\|^2_{H^s} + \|j^{n-1}\|^2_{H^s}\|B^{n+1,n}\|^2_{H^r} \,d\tau
			\\ \nonumber
			&\quad + C \int^t_0  \|v^{n+1,n}\|^2_{L^\infty} \|B^n\|^2_{H^s} +  \|v^n\|^2_{L^\infty} \|B^{n,n-1}\|^2_{H^r}  \,d\tau
			\\ \label{E_n+1n_L2L2}
			&\quad + C \overline{n}^{-2s}\|\Gamma_0\|^2_{X^s}.
		\end{align}
		Furthermore, the two remaining terms $I^{n+1,n,1}_4$ and $I^{n+1,n,2}_3$ can be bounded as follows. By applying Lemma \ref{lem_paraproduct}, and using \eqref{FT} and \eqref{FT_B}, we find that for $t > 0$
		\begin{align} 
			I^{n+1,n,1}_4 
			\label{I_n+1n1_4}
			&\leq C(s) \overline{n}^{-(s-1)} \|j^n\|_{L^2(0,t;L^2)} \|B^{n-1}\|_{L^2(0,t;\dot{H}^s)};
			\\ 
			I^{n+1,n,2}_3 
			\label{I_n+1n2_3}
			&\leq C(s) \overline{n}^{-s} \|j^n\|_{L^2(0,t;L^2)}  \|B^{n-1}\|_{L^\infty(0,t;\dot{H}^s)}.
		\end{align}
		Therefore, we now can estimate the norm  $\|v^{n+1,n}\|_{L^2(0,t;L^\infty)}$ for $t > 0$ as follows. Similar to Step 1b-1, it follows from the above estimates of $I^{n+1,n,1}_i$ and $I^{n+1,n,2}_k$ for $i \in \{1,...,6\}$ and $k \in \{1,...,9\}$,  \eqref{I_n+1n1_4}-\eqref{I_n+1n2_3}, and by choosing the right-hand side of \eqref{En_global} to be much smaller, we obtain that for $t > 0$ and $C = C(d,\kappa,\nu,\sigma,s)$
		\begin{equation} \label{v_n+1,n_L2Linfty}
			\|v^{n+1,n}\|^2_{L^2(0,t;L^\infty)} \leq  C \overline{n}^{-s_{01}}\|\Gamma_0\|^2_{X^s} + \frac{1}{32C_1} \mathcal{E}^{n,n-1}_r(t) + \frac{1}{4C_1} \mathcal{E}^{n+1,n}_r(t),
		\end{equation}
		where $s_{01} := \min\{2(s-1),4s'\}$ for some $s' \in (0,1)$ and $C_1$ is the constant given in \eqref{Gamma_n+1n_1}. Furthermore, it follows from \eqref{B_n+1n_L2H1}-\eqref{E_n+1n_L2L2} that (we may choose smaller initial data) for $t > 0$
		\begin{multline} \label{B_n+1n_L2H1_2}
			\delta_0 \|(E^{n+1,n},B^{n+1,n})\|^2_{L^2(0,t;\dot{H}^1)} + \|E^{n+1,n}\|^2_{L^2(0,t;L^2)} 
			\\
			\leq  C\overline{n}^{-2(s-1)} \|\Gamma_0\|^2_{X^s} +  3\delta_0 \mathcal{E}^{n+1,n}_r(t) 
			+ \frac{1}{32} \mathcal{E}^{n,n-1}_r(t) + \frac{1}{8} \mathcal{E}^{n+1,n}_r(t).
		\end{multline}
		By combining the two previous inequalities \eqref{v_n+1,n_L2Linfty}-\eqref{B_n+1n_L2H1_2} and choosing $\delta_0 = \frac{1}{24}$, we find that for $t > 0$ and $C = C(d,\kappa,\nu,\sigma,s)$
		\begin{multline} \label{v_n+1n_L2Linfty_2}
			C_1\|v^{n+1,n}\|^2_{L^2(0,t;L^\infty)} + \delta_0 \|(E^{n+1,n},B^{n+1,n})\|^2_{L^2(0,t;\dot{H}^1)} + \|E^{n+1,n}\|^2_{L^2(0,t;L^2)}
			\\
			\leq  C \overline{n}^{-s_{01}}\|\Gamma_0\|^2_{X^s} + \frac{1}{16} \mathcal{E}^{n,n-1}_r(t) + \frac{1}{2} \mathcal{E}^{n+1,n}_r(t).
		\end{multline}
		
		\textbf{Step2a-2b: The case $d=3$.} Similar to Step 1b-2, we decompose \eqref{v_n+1n} as follows
		\begin{align*} 
			\partial_t v^{n+1,n,1} - \nu \Delta v^{n+1,n,1} &= \sum^{3}_{i=1} f^{n+1,n,i}, \quad \textnormal{div}\, v^{n+1,n,1} = 0, \quad v^{n+1,n,1}_{|_{t=0}} = (\mathcal{T}_{\overline{n+1}}-\mathcal{T}_{\overline{n}})(v_0),
			\\
			\partial_t v^{n+1,n,2} - \nu \Delta v^{n+1,n,2} &= \sum^{6}_{i=4} f^{n+1,n,i}, \quad \textnormal{div}\, v^{n+1,n,2} = 0, \quad v^{n+1,n,2}_{|_{t=0}} = 0,
		\end{align*}
		which yields as in \eqref{v_21_L2Linfty_3}-\eqref{v_22_L2Linfty_2} for $t \in (0,\infty)$ and $C = C(\nu)$
		\begin{align*}
			\|v^{n+1,n,1}\|_{L^2(0,t;L^\infty)} &\leq  C\|f^{n+1,n,1} + f^{n+1,n,2} + f^{n+1,n,3}\|_{L^1(0,t;\dot{H}^\frac{1}{2})};
			\\
			\|v^{n+1,n,2}\|_{L^2(0,t;L^\infty)} &\leq  C\|f^{n+1,n,4} + f^{n+1,n,5} + f^{n+1,n,6}\|_{L^2(0,t;\dot{B}^{-\frac{1}{2}}_{2,1}).}
		\end{align*} 
		By using 
		\cite[Corollary 2.55]{Bahouri-Chemin-Danchin_2011} and the following estimate 
		(its proof is simple) 
		\begin{equation} \label{FT_B}
			\|\mathcal{T}_m f - f\|_{\dot{B}^{s_1}_{2,1}} \leq C m^{-s_2} \|f\|_{\dot{B}^{s_1+s_2}_{2,1}}.
		\end{equation}
 		and \eqref{En_global} for small data, we find that for some $s'' \in (\frac{1}{2},\frac{3}{2})$ and for $C = C(d,\kappa,\nu,\sigma,s)$
 		\begin{equation} \label{v_n+1n_L2Linfty}
 			C_1\|v^{n+1,n}\|^2_{L^2(0,t;L^\infty)} \leq C\overline{n}^{s_{02}} \|\Gamma_0\|^2_{X^s} + \frac{1}{32} \mathcal{E}^{n,n-1}_r(t) + \frac{1}{4} \mathcal{E}^{n+1,n}_r(t),
 		\end{equation}
 		where $s_{02} := \min\{2(s-\frac{3}{2}),2(2s''-1)\}$ and $C_1$ is the constant given in \eqref{Gamma_n+1n_1}.
		
		\textbf{Step 2a-2c: Conclusion.} Therefore, it follows from \eqref{Gamma_n+1n_1}, \eqref{B_n+1n_L2H1_2}, \eqref{v_n+1n_L2Linfty_2} and \eqref{v_n+1n_L2Linfty} that for $t > 0$ and for $C = C(d,\kappa,\nu,\sigma,s)$
 		\begin{equation} \label{E_n+1n_nn-1}
 			\mathcal{E}^{n+1,n}_r(t) \leq C2^{-ns_0} \|\Gamma_0\|^2_{X^s} + \frac{1}{4} \mathcal{E}^{n,n-1}_r(t),
 		\end{equation}
		where $s_0 := \min\{2(s-r),s_{01},s_{02}\}$. Therefore, by using the usual argument (for example, see \cite{Brezis_2011}), we can conclude from \eqref{E_n+1n_nn-1} that $\Gamma^n$ and $(\nabla v^n, v^n, \nabla E^n,E^n,\nabla B^n,j^n)$ are Cauchy sequences in $L^\infty(0,\infty;H^{r-1} \times H^r \times H^r)$ and $L^2(0,\infty;H^{r-1} \times L^\infty \times L^2 \times L^2 \times L^2 \times H^r)$ for $r \in (\frac{d}{2},s)$, respectively.
		
		\textbf{Step 2b: Passing to the limit as $n \to \infty$.} Now, we are ready to pass to the limit as $n \to \infty$ in \eqref{NSM-GO_n}. From the previous substep, there exists $(v,E,B,j)$ such that as $n \to \infty$
		\begin{align*}
			&&(v^n,E^n,B^n) &\to (v,E,B) &&\text{in} \quad L^\infty(0,\infty;H^{r-1} \times H^r \times H^r),&&
			\\
			&&(\nabla v^n, \nabla E^n, \nabla B^n) &\to (\nabla v, \nabla E, \nabla B)  &&\text{in} \quad L^2(0,\infty;H^{r-1} \times  L^2 \times L^2),&&
			\\
			&&(v^n,E^n,j^n) &\to (v,E,j)  &&\text{in} \quad L^2(0,\infty; L^\infty \times  H^r \times H^r),&&
			\\
			&&(\Delta v^n, \nabla \times E^n,\nabla \times B^n) &\to (\Delta v, \nabla \times E,\nabla \times B) &&\text{in} \quad L^2(0,\infty;H^{r-2} \times H^{r-1} \times H^{r-1}),&&
		\end{align*}
		where we also used \eqref{En_global} and interpolation inequalities. Moreover, for the nonlinear terms as $n \to \infty$, it follows that 
		\begin{align*}
			&&\mathcal{T}_{\overline{n}}(v^{n-1} \cdot \nabla v^n,v^n \times B^{n-1}) &\to (v \cdot \nabla v,v \times B) &&\text{in} \quad L^2(0,\infty;H^{r-2} \times H^{r}),&&
			\\
			&&\mathcal{T}_{\overline{n}}(j^n \times B^{n-1}, j^{n-1} \times B^n) &\to (j \times B, j \times B) &&\text{in} \quad L^2(0,\infty;H^{r}),&&
		\end{align*}
		since for any $t \in (0,\infty)$, the following estimates hold (similar to other nonlinear terms)
		\begin{align*}
			\textnormal{NL}_1 &:= \int^t_0 \|\mathcal{T}_{\overline{n}}(v^{n-1} \cdot \nabla v^n) - v \cdot \nabla v\|^2_{H^{r-2}}  \,d\tau 
			\\
			&\leq C\overline{n}^{-(s-r)} \int^t_0 \|v^{n-1}\|^2_{H^{s-1}} \|v^n\|^2_{L^\infty} + \|v^{n-1}\|^2_{L^\infty} \|v^n\|^2_{H^{s-1}} \,d\tau 
			\\
			&\quad + \int^t_0 \|v^{n-1}-v\|^2_{H^{r-1}} \|v^n\|^2_{L^\infty} + \|v^{n-1}-v\|^2_{L^\infty} \|v^n\|^2_{H^{r-1}} \,d\tau 
			\\
			&\quad + \int^t_0 \|v\|^2_{H^{r-1}}\|v^n-v\|^2_{L^\infty} + \|v\|^2_{L^\infty} \|v^n-v\|^2_{H^{r-1}} \,d\tau.
		\end{align*}
		Thus, we obtain the form of $j$ as in \eqref{NSM-GO-limit} below. Next, we focus on the time derivative terms, it can be seen from \eqref{NSM-GO_n} and the uniform bounds of $(v^n,E^n,B^n,j^n)$ that $(\partial_t v^n, \partial_t E^n, \partial_t B^n)$ is uniformly bounded in $L^2(0,\infty;H^{s-2} \times H^{s-1} \times H^{s-1})$, which yields
		\begin{equation*}
			(\partial_t v^{n_k}, \partial_t E^{n_k}, \partial_t B^{n_k}) \rightharpoonup (\partial_t v, \partial_t E, \partial_t B) \qquad \text{in} \quad  L^2(0,\infty;H^{s-2} \times H^{s-1} \times H^{s-1}).
		\end{equation*}
		Furthermore, \eqref{NSM-GO_n} and the above strong convergences imply that as $n \to \infty$
		\begin{equation*}
			(\partial_t v^n, \partial_t E^n, \partial_t B^n) \to (\partial_t v, \partial_t E, \partial_t B) \qquad \text{in} \quad  L^2(0,\infty;H^{r-2} \times H^{r-1} \times H^{r-1}).
		\end{equation*}
		Moreover, for the initial data as $n \to \infty$
		\begin{equation*}
			(v^n,E^n,B^n)_{|_{t=0}} = \mathcal{T}_{\overline{n}}(v_0,E_0,B_0) \to (v_0,E_0,B_0) \quad \text{in} \quad H^{s-1} \times H^s \times H^s.
		\end{equation*}
		Therefore, after passing to the limit as $n \to \infty$ in \eqref{NSM-GO_n}, we obtain the following system for $r \in (\frac{d}{2},s)$
		\begin{equation} \label{NSM-GO-limit}
			\left\{
			\begin{aligned}
				\partial_t v + \mathbb{P}(v \cdot \nabla v) &= \nu \Delta v + \mathbb{P}(j \times B)  &\text{in} \quad L^2(0,\infty;H^{r-2}),
				\\
				\frac{1}{c} \partial_t E - \nabla \times B &= -j &\text{in} \quad L^2(0,\infty;H^{r-1}),
				\\
				\frac{1}{c} \partial_t B + \nabla \times E &= 0 &\text{in} \quad L^2(0,\infty;H^{r-1}),
				\\
				j + \kappa\mathbb{P}(j \times B) &= \sigma(c E + \mathbb{P}(v \times B)) &\text{in} \quad L^2(0,\infty;H^{r}),
				\\
				\textnormal{div}\, v = \textnormal{div}\, E &= \textnormal{div}\, B = \textnormal{div}\, j = 0 &\text{in} \quad L^2(0,\infty;H^{r-1}),
				\\
				(v,E,B)_{|_{t=0}} &= (v_0,E_0,B_0) &\text{in} \quad H^{s-1} \times H^s \times H^s.
			\end{aligned}
			\right.
		\end{equation}
		In addition, for $\Gamma := (v,E,B)$, $t > 0$, $C = C(d,\kappa,\nu,\sigma,s)$, $s'' \in [0,s]$ and $s' \in [1,s]$
		\begin{equation} \label{E_global}
			\|\Gamma(t)\|^2_{X^s} +  \int^t_0 \|\nabla v\|^2_{H^{s-1}} + \|v\|^2_{L^\infty} + \|E\|^2_{\dot{H}^{s''}} + \|B\|^2_{\dot{H}^{s'}} + \|j\|^2_{H^s} \,d\tau \leq C \|\Gamma_0\|^2_{X^s}.
		\end{equation}
		It follows from \eqref{NSM-GO-limit} that the pressures $\pi_1$ and $\pi_2$ can be recovered in the usual way (see \cite{Temam_2001}) in which we will obtain \eqref{NSM-GO} in some particular sense. We omit further details.
		
		\textbf{Step 3: Uniqueness.} The proof of this step is standard, we omit the details.

	\end{proof}
	
	We continue with the proof of Part $(ii)$ as follows.
		
	\begin{proof}[Proof of Theorem \ref{theo_global_small}-$(ii)$] 
		The idea of proof comes from that of \cite[Corollary 1.3]{Arsenio-Gallagher_2020}, where the author focused on \eqref{NSM} with $\alpha = 1$ in two dimensions. By using the global well-posedness given as in Part $(i)$, we aim to apply the same idea to \eqref{NSM-GO} with $\alpha = 1$ and $d \in \{2,3\}$. 
		By applying Part $(i)$ to \eqref{NSM-GO} under using the smallness assumption of the initial data, we know that there exists a unique global solution $(v^c,E^c,B^c)$ with uniform bound in terms of $c$ as in \eqref{E_global}, i.e., there exists a $c$-independent constant $C = C(d,\kappa,\nu,\sigma,s)$ such that for $s'' \in [0,s]$, $s' \in [1,s]$ and  $t > 0$ 
		\begin{equation} \label{NSM-GO-En-c}
			\|\Gamma^c(t)\|^2_{X^s} + \int^t_0 \|\nabla v^c\|^2_{H^{s-1}} + \|v^c\|^2_{L^\infty} + \|E^c\|^2_{\dot{H}^{s''}} + \|B^c\|^2_{\dot{H}^{s'}} + \|j^c\|^2_{H^s} \,d\tau \leq C \|\Gamma^c_0\|^2_{X^s},
		\end{equation}
		where $\Gamma^c := (v^c,E^c,B^c)$ is uniformly bounded in $X^s = H^{s-1} \times H^s \times H^s$ for $s > \frac{d}{2}$.
		Therefore, \eqref{NSM-GO-En-c} implies that there exists $(v,E,B,j)$ such that up to an extraction of a subsequence (with using the same notation) as $c \to \infty$
		\begin{align*}
			&&(v^c,E^c,B^c) &\overset{\ast}{\rightharpoonup}  (v,E,B) &&\text{in} \quad L^\infty_tL^2_x,&&
			\\
			&&(\nabla v^c,j^c) &\rightharpoonup (v,j) &&\text{in} \quad L^2_tL^2_x.&&
		\end{align*}
		In addition, we find from \eqref{NSM-GO} that
		\begin{align*}
			(\partial_t v^c,\partial_tB^c) \quad \text{is uniformly bounded in} \quad L^2_{\textnormal{loc},t} H^{-1}_{\textnormal{loc},x}
		\end{align*}
		and by using the Aubin-Lions lemma (see \cite{Boyer-Fabrie_2013,Temam_2001}) as $c \to \infty$
		\begin{equation*}
			(v^c,B^c) \to (v,B) \quad \text{ in}\quad L^2_{\textnormal{loc},t} L^2_{\textnormal{loc},x}.
		\end{equation*}
		As usual, for $\phi,\varphi \in C^\infty_0([0,\infty) \times \mathbb{R}^d;\mathbb{R}^3)$ with $\textnormal{div}\,\phi = 0$, the weak formula of \eqref{NSM-GO} is given by 
		\begin{align*}
			&\int^\infty_0 \int_{\mathbb{R}^d} v^c \cdot \partial_t\phi + (v^c \otimes v^c) : \nabla \phi - \nu \nabla v^c : \nabla \phi + (j^c \times B^c) \cdot \phi \,dxdt = -\int_{\mathbb{R}^d} v^c(0) \cdot \phi(0) \,dx,
			\\
			&\int^\infty_0 \int_{\mathbb{R}^d} \frac{1}{c}E^c \cdot \partial_t\varphi + B^c \cdot (\nabla \times \varphi) - j^c \cdot \varphi \,dxdt = -\int_{\mathbb{R}^d} \frac{1}{c}E^c(0) \cdot \varphi(0) \,dx,
			\\
			&\int^\infty_0 \int_{\mathbb{R}^d} B^c \cdot \partial_t\varphi + [(v^c \times B^c) - \frac{1}{\sigma} j^c - \frac{\kappa}{\sigma} j^c \times B^c] \cdot (\nabla \times \varphi) \,dxdt  = - \int_{\mathbb{R}^d} B^c(0) \cdot \varphi(0) \,dx.
		\end{align*}
		Therefore, we can pass to the limit by using the assumption $(v^c_0,E^c_0,B^c_0) \rightharpoonup (\bar{v}_0,\bar{E}_0,\bar{B}_0)$ in $H^{s-1} \times H^s \times H^s$ and the above strong convergences as $c \to \infty$, to obtain that \eqref{NSM-GO} converges in the sense of distributions to \eqref{HMHD} with $\alpha = \beta = 1$.
	\end{proof}
	
	We conclude the whole proof with the proof of Part $(iii)$ as follows.
	
	\begin{proof}[Proof of Theorem \ref{theo_global_small}-$(iii)$] The proof of this part consists of two steps as follows.
		
		\textbf{Step 1: The GWP of \eqref{NSM-SO} and of \eqref{NSM-GO} for $\Gamma^\kappa$.} The existence of a unique global solution $\Gamma = (v,E,B)$ to \eqref{NSM-SO} with $\Gamma_{|_{t=0}} = \Gamma_0$ can be done as in Part $(i)$ with setting $\kappa = 0$. In addition, it can be seen during the proof of Part $(i)$ that the dependency on $\kappa$ of the constants (especially in Condition \eqref{small_data}) can be given explicitly in the form $\kappa^2 C(d,\nu,\sigma,s)$. Therefore, if $\kappa \in (0,1)$ then we can assume that $C'_1\|\Gamma_0\|_{X^s} \leq 1$, where $C_1 = \kappa^2 C'_1$ and $C_1$ is the constant in Condition \eqref{small_data}, which allows us to establish a similar result as in Part $(i)$ for $\Gamma^\kappa$.
		
		\textbf{Step 2: Estimate of the difference.} It can be seen that $\Gamma^\kappa$ and $\Gamma$ satisfy \eqref{NSM-GO-limit}-\eqref{E_global} with $\kappa \in (0,1)$ and $\kappa = 0$, respectively. 
		Therefore, it follows from \eqref{NSM-GO-limit}-\eqref{E_global} and \cite[Chapter 3, Lemma 1.2]{Temam_2001} that
		\begin{equation*}
			\frac{1}{2} \frac{d}{dt}\|\Gamma^\kappa-\Gamma\|^2_{L^2} + \nu \|\nabla (v^\kappa - v)\|^2_{L^2} + \frac{1}{\sigma} \|j^\kappa-j\|^2_{L^2} =: \sum^3_{i=1} I^\kappa_i.
		\end{equation*}
		By direct computations and denoting $Y^\kappa(t) := \|(v^\kappa-v,E^\kappa-E,B^\kappa-B)(t)\|^2_{L^2}$ with $t > 0$, we find that for $C = C(\nu,\sigma)$
		\begin{equation*}
			\frac{d}{dt}Y^\kappa + \nu \|\nabla(v^\kappa-v)\|^2_{L^2} + \frac{1}{\sigma}\|j^\kappa-j\|^2_{L^2} 
			\leq C\|(|v|^2,j)\|_{L^\infty}Y^\kappa + C\kappa^2 \|j^\kappa\|^2_{H^s} \|B^\kappa\|^2_{L^2},
		\end{equation*}
		which yields 
		\begin{equation*}
			Y^\kappa(t)
			\leq C\kappa^2 \|B^\kappa\|^2_{L^\infty_tL^2_x}\int^t_0 \|j^\kappa\|^2_{H^s} \,d\tau \exp\left\{C\int^t_0 \|(|v|^2,j)\|_{L^\infty}  \,d\tau \right\}.
		\end{equation*}
		Thus, the conclusion follows by using \eqref{E_global} and interpolation inequality for $s' \in [0,s-1)$ and $s'' \in [0,s)$, $t \in (0,T)$ for any $T \in (0,\infty)$ (similar to $(B^\kappa-B)$)
		\begin{align*}
			\|(v^\kappa-v)(t)\|_{H^{s'}} \leq \|(v^\kappa-v)(t)\|^\frac{s-s'-1}{s-1}_{L^2}\|(v^\kappa-v)(t)\|^\frac{s'}{s}_{H^s} &\leq \kappa^\frac{s-s'-1}{s-1}C(T,d,\nu,\sigma,s,\Gamma_0),
			\\
			\|(E^\kappa-E)(t)\|_{H^{s''}} \leq \|(E^\kappa-E)(t)\|^\frac{s-s''}{s}_{L^2}\|(E^\kappa-E)(t)\|^\frac{s''}{s}_{H^s} &\leq \kappa^\frac{s-s'}{s}C(T,d,\nu,\sigma,s,\Gamma_0).
		\end{align*}
	\end{proof}
	
	%
	\section{Proof of Theorem \ref{theo_local}} \label{sec:local}
	%
	
	In this section, we focus on obtaining the local well-posedness of \eqref{NSM-GO} in the inviscid case, i.e., $\nu = 0$ for possibly large initial data. 
	
	\begin{proof}[Proof of Theorem \ref{theo_local}] 
		The idea of proof is similar to that of Theorem \ref{theo_global_small}. However, we are not able to follow the main part of that proof and we should consider the initial data for $v$ and $(E,B)$ in the same space, since we do not have the diffusion term here. 
		
		\textbf{Step 1: Local estimate.} We will use the same approximate system \eqref{NSM-GO_n} without the velocity diffusion term, i.e., $\nu = 0$. We aim to prove that there exists a sufficiently small $T = T(d,\kappa,\sigma,s,\Gamma_0) \in (0,1]$ and solution $\Gamma^n \in C^1([0,T);Y^s_n)$, where $Y^s_n$ will be defined below, to the approximate system \eqref{NSM-GO_n} with $\nu = 0$ for each $n \in \mathbb{N}$  such that for $t \in (0,T)$
		\begin{equation} \label{L-Gamma-n_1}
			\|\Gamma^{n}(t)\|^2_{H^s} + \int^t_0 
			\|j^{n}\|^2_{H^s} \,d\tau \leq C(d,\kappa,\sigma,s)\|\Gamma_0\|^2_{H^s}. 
		\end{equation}
		
		\textbf{Step 1a: The case $n = 1$.} In this case, similar to \eqref{NSM-GO_n} with $n = 1$, 
		there exists a unique solution 
		$\Gamma^1 = (v^1,E^1,B^1) \in C^1([0,T^1_*);Y^s_1)$, where $Y^s_n := H^s_{n,0} \times H^s_{n,0} \times H^s_{n,0}$ for $n \in \mathbb{N}$. In addition, as in the proof of Theorem \ref{theo_global_small}, it follows 
		that $T^1_* = \infty$ since for $t \in (0,T^1_*)$
		\begin{equation} \label{L-Gamma-1}
			\|\Gamma^1(t)\|^2_{H^s} +  \int^t_0 
			 \|j^1\|^2_{H^s} \,d\tau \leq C(d,\sigma,s)\|\Gamma_0\|^2_{H^s},
		\end{equation}
		which is \eqref{L-Gamma-n_1} with $n = 1$.
		
		\textbf{Step 1b: The case $n = 2$.} Similar to \eqref{NSM-GO_n} with $n = 1$ and $\nu = 0$, 
		there exists a unique solution $\Gamma^2 := (v^2,E^2,B^2) \in C^1([0,T^2_*);Y^s_2)$ such that 
		\begin{equation*}
			\frac{1}{2}\frac{d}{dt} \|\Gamma^2\|^2_{H^s} + \frac{1}{\sigma}\|j^2\|^2_{H^s} =: \sum^5_{i=1} I^2_i,
		\end{equation*}
		where since $s > \frac{d}{2} + 1$, for some $\epsilon \in (0,1)$
		\begin{align*}
			I^2_1 &= \int_{\mathbb{R}^d} [\Lambda^s (v^1 \cdot \nabla v^2) - v^1 \cdot \nabla \Lambda^s v^2]  \cdot \Lambda^s v^2\,dx
			\leq C(d,s) \|v^1\|_{H^s} \|v^2\|^2_{H^s},
			\\
			I^2_2 &= \int_{\mathbb{R}^d} \Lambda^s(j^2 \times B^1) \cdot \Lambda^s v^2 \,dx \leq \frac{\epsilon}{\sigma} \|j^2\|^2_{H^s} + C(d,\epsilon,\sigma,s) \|B^1\|^2_{H^s}\|v^2\|^2_{H^s};
			\\
			I^2_3 &= \int_{\mathbb{R}^d} \Lambda^s j^2 \cdot \Lambda^s(v^2 \times B^1)\,dx \leq \frac{\epsilon}{\sigma} \|j^2\|^2_{H^s} + C(d,\epsilon,\sigma,s) \|B^1\|^2_{H^s}\|v^2\|^2_{H^s};
			\\
			I^2_4 &= -\frac{\kappa}{\sigma}\int_{\mathbb{R}^d} \Lambda^s j^2 \cdot \Lambda^s (j^1 \times B^2)\,dx
			\leq \frac{\epsilon}{\sigma} \|j^2\|^2_{H^s} + C(d,\epsilon,\kappa,\sigma,s) \|j^1\|^2_{H^s}\|B^2\|^2_{H^s};
			\\
			I^2_5 &= -\frac{\kappa}{\sigma}\int_{\mathbb{R}^d} j^2 \cdot (j^1 \times B^2) \,dx \leq \frac{\epsilon}{\sigma} \|j^2\|^2_{H^s} + C(d,\epsilon,\kappa,\sigma,s) \|j^1\|^2_{H^s}\|B^2\|^2_{H^s},
		\end{align*}
		here we used the following well-known commutator estimate (for example, see \cite{Li_2019}) for $s_0 \in \mathbb{R}, s_0 > 0$, $p \in (1,\infty)$, $p_i,q_i \in (1,\infty]$, $i \in \{1,2\}$ with $\frac{1}{p_i} + \frac{1}{q_i} = \frac{1}{p}$ for $C = C(d,s_0,p,p_i,q_i)$
		\begin{equation*}
			\|\Lambda^{s_0}(fg) - f\Lambda^{s_0} g\|_{L^p(\mathbb{R}^d)} \leq C \left(\|\Lambda^{s_0} f\|_{L^{p_1}(\mathbb{R}^d)} \|g\|_{L^{q_1}(\mathbb{R}^d)} + \|\nabla f\|_{L^{p_2}(\mathbb{R}^d)}\|\Lambda^{s_0-1} g\|_{L^{q_2}(\mathbb{R}^d)}\right).
		\end{equation*}
		Therefore, by choosing $\epsilon = \frac{1}{8}$, we find that 
		\begin{equation} \label{L-Gamma-2}
			\frac{d}{dt} \|\Gamma^2\|^2_{H^s} + \frac{1}{\sigma} \|j^2\|^2_{H^s} \leq C(d,\kappa,\sigma,s)\left(\|v^1\|_{H^s} + \|(B^1,j^1)\|^2_{H^s}\right)\|(v^2,B^2)\|^2_{H^s},
		\end{equation}
		which by using \eqref{L-Gamma-1} leads to $T^2_* = \infty$. Choosing $T_* = T_*(d,\kappa,\sigma,s,\Gamma_0) \leq 1$, which satisfies
		\begin{equation*}
			C(d,\kappa,\sigma,s)\left(T_* \|\Gamma_0\|_{H^s} + T_*\|\Gamma_0\|^2_{H^s} + \int^{T_*}_0 \|j^1\|^2_{H^s}\,d\tau\right) \leq 1.
		\end{equation*}
		As in the previous step by using \eqref{L-Gamma-1}, it comes from \eqref{L-Gamma-2} that for $t \in (0,T_*)$
		\begin{equation} \label{L-Gamma-2_1}
			\|\Gamma^2(t)\|^2_{H^s} + \int^t_0 
			\|j^2\|^2_{H^s} \,d\tau \leq C(d,\kappa,\sigma,s)\|\Gamma_0\|^2_{H^s}. 
		\end{equation}
		which is \eqref{L-Gamma-n_1} with $n = 2$.
				
		\textbf{Step 1c: The case $n > 2$ with an induction argument.} Assume that $n = k > 2$. 
		We now assume an induction hypothesis that	\eqref{L-Gamma-n_1} holds for $n = k$, i.e., there exists some $T_{**} > 0$ and solution $\Gamma^k := (v^k,E^k,B^k) \in C^1([0,T_{**});Y^s_k)$ such that for $t \in (0,T_{**})$ 
		\begin{equation} \label{L-Gamma-k}
			\|\Gamma^k(t)\|^2_{H^s} + \int^t_0 
			\|j^k\|^2_{H^s} \,d\tau \leq C(d,\kappa,\sigma,s)\|\Gamma_0\|^2_{H^s}. 
		\end{equation}
		The above assumption is true in the case $n = 2$ in Step 1b, see \eqref{L-Gamma-2_1}. 
		Similar to Step 1b, the induction hypothesis assumption gives us the local existence and uniqueness of $\Gamma^{k+1} := (v^{k+1},E^{k+1},B^{k+1}) \in C^1([0,T^{k+1}_*);Y^s_{k+1})$ for some $T^{k+1}_* \in (0,T_{**}]$, and then as in \eqref{L-Gamma-2} (with replacing $(\Gamma^2,j^2)$ and $(v^1,B^1,j^1)$ by $(\Gamma^{k+1},j^{k+1})$ and $(v^k,B^k,j^k)$ respectively) we obtain $T^{k+1}_* = T_{**}$. Furthermore, by choosing some $T_{***} \leq T_{**}$ sufficiently small and using \eqref{L-Gamma-k}, we find that for $t \in (0,T_{***})$
		\begin{equation} \label{L-Gamma-k+1}
			\|\Gamma^{k+1}(t)\|^2_{H^s} + \int^t_0
			\|j^{k+1}\|^2_{H^s} \,d\tau \leq C(d,\kappa,\sigma,s)\|\Gamma_0\|^2_{H^s}. 
		\end{equation}
		Therefore, there exists $T = T(d,\kappa,\sigma,s,\Gamma_0) > 0$ sufficiently small such that for $t \in (0,T)$ and $n \in \mathbb{N}$
		\begin{equation} \label{L-Gamma-n}
			\|\Gamma^{n}(t)\|^2_{H^s} + \int^t_0 
			\|j^{n}\|^2_{H^s} \,d\tau \leq C(d,\kappa,\sigma,s)\|\Gamma_0\|^2_{H^s}. 
		\end{equation}
		
		\textbf{Step 2: Passsing to the limit as $n \to \infty$ and uniqueness.} It maybe not be suitable to follow Step 2 in the proof of Theorem \ref{theo_global_small}-$(i)$. Therefore, by using the uniform bound \eqref{L-Gamma-n}, we can pass to the limit in the sense of distributions from \eqref{NSM-GO_n} with $\nu = 0$, where the limit solution satisfies
		\begin{equation} \label{L-Gamma}
			\|\Gamma(t)\|^2_{H^s} + \int^t_0 
			\|j\|^2_{H^s} \,d\tau \leq C(d,\kappa,\sigma,s)\|\Gamma_0\|^2_{H^s},
		\end{equation}
		for $t \in (0,T_0)$ for some $T_0 = T_0(d,\kappa,\sigma,s,\Gamma_0) \in (0,1]$. Since the limiting solution is smooth enough then the proof of the uniqueness can be done as that of Theorem \ref{theo_global_small}. We omit further details. 
		
		\textbf{Step 3: Proof of Part $(ii)$.} By repeating the two previous steps in the case $\nu > 0$, but without using this advantage, we obtain the existence of a unique solution $\Gamma^\nu := (v^\nu,E^\nu,B^\nu)$ to \eqref{NSM-GO} in some time interval $(0,T_{02})$ with $T_{02} = T_{02}(d,\kappa,\sigma,s,\Gamma_0) \in (0,1]$ and $\Gamma^\nu_{|_{t=0}} = \Gamma_0$. Assume that $\Gamma := (v,E,B)$ is a unique solution to \eqref{NSM-GO} with $\nu = 0$ and $\Gamma_{|_{t=0}} = \Gamma_0$ given as in Part $(i)$ in $(0,T_{01})$.
		Therefore, it follows that (since the solutions $\Gamma^\nu$ and $\Gamma$ here are smooth enough, then the difference can be taken as a test function, see \cite[Chapter 3, Lemma 1.2]{Temam_2001})
		\begin{equation*}
			\frac{1}{2} \frac{d}{dt}\|\Gamma^\nu-\Gamma\|^2_{L^2} + \nu \|\nabla (v^\nu - v)\|^2_{L^2} + \frac{1}{\sigma} \|j^\nu-j\|^2_{L^2} =: \sum^4_{i=1} I^\nu_i. 
		\end{equation*}
		By direct computations and denoting $Y^\nu(t) := \|(v^\nu-v,E^\nu-E,B^\nu-B)(t)\|^2_{L^2}$ with $t \in (0,T_0)$, we find that
		\begin{equation*}
			\frac{d}{dt}Y^\nu + \nu \|\nabla(v^\nu-v)\|^2_{L^2} + \frac{1}{\sigma}\|j^\nu-j\|^2_{L^2} 
			\leq C(\kappa,\sigma)\left(\|(|v|^2,|j|^2,j)\|_{L^\infty} + 1\right)Y^\nu + \nu^2 \|\Delta v\|^2_{L^2},
		\end{equation*}
		which yields 
		\begin{equation*}
			Y^\nu(t)
			\leq \nu^2 \int^t_0 \|\Delta v\|^2_{L^2} \,d\tau \exp\left\{C(\kappa,\sigma)\int^t_0 \|(|v|^2,|j|^2,j)\|_{L^\infty}  + 1 \,d\tau \right\}.
		\end{equation*}
		Thus, the conclusion follows by using \eqref{L-Gamma} and interpolation inequality for $s' \in [0,s)$ (similar to $(E^\nu-E,B^\nu-B)$) and $t \in (0,T_0)$
		\begin{equation*}
			\|(v^\nu-v)(t)\|_{H^{s'}} \leq \|(v^\nu-v)(t)\|^\frac{s-s'}{s}_{L^2}\|(v^\nu-v)(t)\|^\frac{s'}{s}_{H^s} \leq \nu^\frac{s-s'}{s}C(d,\kappa,\sigma,s,\Gamma_0).
		\end{equation*}
		
		\textbf{Step 3: Proof of Part $(iii)$ and $(iv)$.} The proof can be done as those of Theorem \ref{theo_global_small}-$(ii)$ and $(iii)$, respectively. We omit the details and end the proof.
	\end{proof} 
	
	%
	\section{Proof of Theorem \ref{theo-MH}} \label{sec:theo-MH}
	%
	
	In this section, we provide a proof of Theorem \ref{theo-MH}.
	
	\begin{proof}[Proof of Theorem \ref{theo-MH}-$(i)$]
		The proof shares a similar idea as that of in the proof of Theorem \ref{theo_global_small} and consists of the following steps.
		
		\textbf{Step 1: Approximate system, local and global existence.} We will use the same approximate system as in the proof of Theorem \ref{theo_global_small} and we need to change the definition of the function $F^n$ as follows for $\Gamma^{\sigma,n} := (v^{\sigma,n},E^{\sigma,n},B^{\sigma,n})$
		\begin{equation*}
			F^{n} :\, W^s_n := H^{s-1}_{n,0} \times (H^s_{n,0} \cap \dot{H}^{-1}) \times (H^s_{n,0} \cap \dot{H}^{-1}) \to W^s_n
			\quad \text{with}\quad 
			\Gamma^{\sigma,n} \mapsto F^{n}(\Gamma^{\sigma,n}).
		\end{equation*}
		The norm in $W^s_n$ is defined by
		\begin{equation*}
			\|\Gamma^{\sigma,n}\|^2_{W^s_n} := \|v^{\sigma,n}\|^2_{H^{s-1}} + \|(E^{\sigma,n},B^{\sigma,n})\|^2_{H^s} + \|(E^{\sigma,n},B^{\sigma,n})\|^2_{\dot{H}^{-1}}.
		\end{equation*}
		We now mainly focus on the $\dot{H}^{-1}$ estimate of $(E^{\sigma,n},B^{\sigma,n})$ at each level $n \in \mathbb{N}$ as follows.
		
		\textbf{Step 1a: The case $n = 1$.} We first see that $F^1$ is well-defined and locally Lipschitz on $W^s_1$. Thus, there exists a unique solution $\Gamma^{\sigma,1} \in C^1([0,T^1_*);W^s_1)$ for some $T^1_* > 0$ such that if $T^1_* < \infty$ then 
		\begin{equation*}
			\lim_{t \to T^1_*} \|\Gamma^{\sigma,1}(t)\|^2_{W^s_1} = \infty.
		\end{equation*}
		In addition, since $j^{\sigma,1} = \sigma c E^{\sigma,1}$ then \eqref{NSM-GO_n} with $n = 1$ gives us
		\begin{equation} \label{H-1_1}
			\|(E^{\sigma,1},B^{\sigma,1})(t)\|^2_{\dot{H}^{-1}} + \frac{1}{\sigma} \int^t_0 \|j^{\sigma,1}\|^2_{\dot{H}^{-1}} \,d\tau = \|(E^{\sigma,1},B^{\sigma,1})(0)\|^2_{\dot{H}^{-1}},
		\end{equation}
		which together with \eqref{Gamma_1} implies that $T^1_* = \infty$.
		
		\textbf{Step 1b: The case $n = 2$.} By using the previous step, there exists a unique solution $\Gamma^{\sigma,2} \in C^1([0,T^2_*);W^s_2)$ for some $T^2_* > 0$. Moreover, \eqref{NSM-GO_n} with $n = 2$ leads to for $t \in (0,T^2_*)$ and $C = C(\kappa,\sigma)$
		\begin{equation*}
			\|(E^{\sigma,2},B^{\sigma,2})(t)\|^2_{\dot{H}^{-1}} + \frac{1}{\sigma} \int^t_0 \|j^{\sigma,2}\|^2_{\dot{H}^{-1}} \,d\tau  
			\leq C \int^t_0 \|v^{\sigma,2} \times B^{\sigma,1} \|^2_{\dot{H}^{-1}} + \|j^{\sigma,1} \times B^{\sigma,2} \|^2_{\dot{H}^{-1}} \,d\tau.
		\end{equation*}
		By using the embedding $L^{p_0}(\mathbb{R}^3) \hookrightarrow \dot{H}^{s_0}(\mathbb{R}^3)$ (see \cite{Bahouri-Chemin-Danchin_2011}) for $p_0 \in (1,2]$ and $s_0 = \frac{3}{2} - \frac{3}{p_0}$ 
		with $(s_0,p_0) = (-1,\frac{6}{5})$, 
		\begin{align*}
			\|v^{\sigma,2} \times B^{\sigma,1} \|^2_{\dot{H}^{-1}} &\leq 
			C \|v^{\sigma,2} \times B^{\sigma,1} \|^2_{L^\frac{6}{5}} 
			\leq C\|v^{\sigma,2}\|_{L^2} \|v^{\sigma,2}\|_{\dot{H}^1} \|B^{\sigma,1}\|^2_{L^2},
			\\
			\|j^{\sigma,1} \times B^{\sigma,2} \|^2_{\dot{H}^{-1}} &\leq 
			C \|j^{\sigma,1} \times B^{\sigma,2} \|^2_{L^\frac{6}{5}} 
			\leq C\|j^{\sigma,1}\|_{L^2} \|j^{\sigma,1}\|_{\dot{H}^1} \|B^{\sigma,2}\|^2_{L^2},
		\end{align*}
		which by using \eqref{Gamma_1} and \eqref{Main_2_1} yields $T^2_* = \infty$ since for $t \in (0,T^2_*)$
		\begin{equation} \label{H-1_2}
			\|(E^{\sigma,2},B^{\sigma,2})(t)\|^2_{\dot{H}^{-1}} + \frac{1}{\sigma} \int^t_0 \|j^{\sigma,2}\|^2_{\dot{H}^{-1}} \,d\tau  
			\leq C(t,d,\kappa,\sigma,s,\Gamma^\sigma_0).
		\end{equation}
		
		\textbf{Step 1c: The case $n = k > 2$.} By using a similar induction argument as done in the proof of Theorem \ref{theo_global_small}-Step 1c with assuming that \eqref{E_km_induction} and \eqref{H-1_2} holds for $(\Gamma^{\sigma,k},j^{\sigma,k})$, we can first obtain 
		\eqref{E_k+1_induction} and \eqref{H-1_2} for $(\Gamma^{\sigma,k+1},j^{\sigma,k+1})$ instead, and then $T^{k+1}_* = \infty$. Therefore, it follows that for $t > 0$ and $n \in \mathbb{N}$
		\begin{align} \nonumber
			&\|\Gamma^{\sigma,n}(t)\|^2_{X^s} + \int^t_0 \|\nabla v^{\sigma,n}\|^2_{H^{s-1}} + \|v^{\sigma,n}\|^2_{L^\infty} + \|E^{\sigma,n}\|^2_{\dot{H}^{s''}} \,d\tau
			\\ \label{E_sigma_global}
			&\quad + \int^t_0 \|B^{\sigma,n}\|^2_{\dot{H}^{s'}} + \|j^{\sigma,n}\|^2_{H^s} \,d\tau \leq C(d,\kappa,\nu,\sigma,s) \|\Gamma^\sigma_0\|^2_{X^s},
			\\
			\label{H-1_n}
			&\|(E^{\sigma,n},B^{\sigma,n})(t)\|^2_{\dot{H}^{-1}} + \int^t_0 \|j^{\sigma,n}\|^2_{\dot{H}^{-1}} \,d\tau  
			\leq C(t,d,\kappa,\sigma,s,\Gamma^\sigma_0).
		\end{align}
		
		\textbf{Step 2: Passing to the limit and uniqueness.} This step can be done as in the proof of Theorem \ref{theo_global_small}-Step 2, where we will obtain \eqref{NSM-GO-limit}, the energy estimate and \eqref{E_sigma_global}-\eqref{H-1_n} for $(\Gamma^\sigma,j^\sigma)$, i.e., for $t > 0$, $s'' \in [0,s]$ and $s' \in [1,s]$, \eqref{Gamma_sigma_1}-\eqref{Gamma_sigma_3} hold.
	\end{proof}
	
	We now continue with the proof of the second part as follows.
	
	\begin{proof}[Proof of Theorem \ref{theo-MH}-$(ii)$] We now suppose that Part $(i)$ holds without assuming Condition \eqref{small_data} on the initial data, i.e., there exists a unique global solution $\Gamma^\sigma$ for possibly large data, which satisfies \eqref{NSM-GO-limit} for $(\Gamma^\sigma,j^\sigma)$ and \eqref{Gamma_sigma_1}-\eqref{Gamma_sigma_3}. Furthermore, by defining $A^{\sigma}$ such that $\nabla \times A^{\sigma} = B^{\sigma}$ and $\textnormal{div}\, A^{\sigma} = 0$ ($A^{\sigma}$ are not unique), using and \cite[Proposition 1.36]{Bahouri-Chemin-Danchin_2011}, it follows from \eqref{Gamma_sigma_2}-\eqref{Gamma_sigma_3} that for $t > 0$  
		\begin{align} \label{AB_csigma}
			\mathcal{H}^{\sigma}(t) &:= \int_{\mathbb{R}^3} A^{\sigma}(t) \cdot B^{\sigma}(t)  \,dx 
			\leq C \|B^{\sigma}(t)\|^2_{\dot{H}^{-\frac{1}{2}}}
			\leq  C(t,\kappa,\nu,\sigma,s,\Gamma^{\sigma}_0).
		\end{align}
		In addition,  \eqref{Gamma_sigma_1}-\eqref{AB_csigma}, Tonelli and Fubini (twice and we need the estimate of $\|B^\sigma\|_{\dot{H}^{-1}}$ here\footnote{In fact, $\mathcal{H}^\sigma(t)$ is well-defined if $\|B^\sigma(t)\|_{\dot{H}^{-\frac{1}{2}}}$ is finite for $t \in (0,\infty)$, which is possible if we consider the $\dot{H}^{-\frac{1}{2}}$ estimate instead of the $\dot{H}^{-1}$ one in \eqref{H-1_n} for $(E^\sigma_0,B^\sigma_0) \in \dot{H}^{-\frac{1}{2}}$.}) Theorems and the limiting system for $\Gamma^{\sigma}$ yield (see \cite{Brezis_2011}, to compute the weak derivative in the first line below using \eqref{NSM-GO-limit} for $(\Gamma^\sigma,j^\sigma)$),
		\begin{align}  \label{AB_csigma_1}
			\frac{d}{dt} \mathcal{H}^{\sigma}(t) 
			&= \int_{\mathbb{R}^3} \partial_t A^{\sigma} \cdot B^{\sigma} + A^{\sigma} \cdot \partial_t B^{\sigma} \,dx 
			= -2 \int_{\mathbb{R}^3} \frac{1}{\sigma} j^{\sigma} \cdot B^{\sigma}\,dx.
		\end{align}
		Integrating in time, we find from \eqref{Gamma_sigma_2} and \eqref{AB_csigma}-\eqref{AB_csigma_1} that for $\tau > 0$ (see \cite{Brezis_2011})
		\begin{align*}
			\left|\mathcal{H}^{\sigma}(\tau) - \mathcal{H}^{\sigma}(0)\right| 
			&= \left|\int^{\tau}_0 \frac{d}{dt} \mathcal{H}^{\sigma}(t) \,dt\right| 
			\leq \sigma^{-\frac{1}{2}} (\tau + 1)\|\Gamma^{\sigma}_0\|^2_{L^2},
		\end{align*}
		which after taking $\sigma \to \infty$ implies that for a.e. $t > 0$
		\begin{equation*}
			\lim_{\sigma \to \infty} \int_{\mathbb{R}^3} A^{\sigma}(t) \cdot B^{\sigma}(t)  \,dx 
			= \int_{\mathbb{R}^3} A_0 \cdot B_0 \,dx =: \mathcal{H}_0,
		\end{equation*}
		by using the uniform bound in $L^2$ of $\Gamma^\sigma_0$ and $B^{\sigma}_0 \to B_0$ in $\dot{H}^{-1}$ as $\sigma \to \infty$ with 
		\begin{align*}
			\left|\int_{\mathbb{R}^3} A^{\sigma}_0 \cdot B^{\sigma}_0  - A_0 \cdot B_0 \,dx \right| &= \left|\int_{\mathbb{R}^3} (A^{\sigma}_0 + A_0) \cdot (B^{\sigma}_0 - B_0) \,dx\right|
			\quad 
			\to 0 \quad \text{as} \quad \sigma \to \infty.
		\end{align*}
		Finally, it follows from the above limit and \eqref{AB_csigma} that if the initial magnetic helicity is positive then there exists an absolute positive constant $C$ such that
		\begin{equation*}
			\liminf_{t\to \infty} \liminf_{\sigma \to \infty} \|B^{\sigma}(t)\|^2_{\dot{H}^{-\frac{1}{2}}} \geq  C\liminf_{t\to \infty} \lim_{\sigma \to \infty} \int_{\mathbb{R}^3} A^{\sigma}(t) \cdot B^{\sigma}(t)  \,dx 
			= C\mathcal{H}_0 > 0.
		\end{equation*}
		Thus, the proof is complete.
	\end{proof}
	
	%
	\section{Proof of Theorem \ref{theo-C-HMHD}} \label{sec:theo-C-HMHD}
	%
	
	In this section, we provide a proof of Theorem \ref{theo-C-HMHD}.
	
	\begin{proof}[Proof of Theorem \ref{theo-C-HMHD}] The proof consists of several steps as follows.
		
		\textbf{Step 1: Approximate system and local existence.} For $n \in \mathbb{N}$, we will use the following approximate system of \eqref{HMHD} with $(\alpha,\beta) = (\alpha_c,\beta_c)$
		\begin{equation} \label{app-C-HMHD}
			\frac{d}{dt} \Gamma^{\sigma,n} = F^{\sigma,n}(\Gamma^{\sigma,n}), \quad \textnormal{div}\, v^{\sigma,n}  = \textnormal{div}\, B^{\sigma,n}  = 0 \quad \text{and} \quad \Gamma^{\sigma,n}_{|_{t=0}} = \mathcal{T}_{\overline{n}}(\Gamma^\sigma_0),
		\end{equation}
		where $\Gamma^{\sigma,n}:= (v^{\sigma,n},B^{\sigma,n}), \Gamma^\sigma_0 := (v^\sigma_0,B^\sigma_0)$ and $F^{\sigma,n} := (F^{\sigma,n}_2,F^{\sigma,n}_2)$ with $j^{\sigma,n} := \nabla \times B^{\sigma,n}$,
		\begin{align*}
			F^{\sigma,n}_1 &= -\mathbb{P}\mathcal{T}_{\overline{n}} (v^{\sigma,n} \cdot \nabla v^{\sigma,n}) - \nu (-\Delta)^{\alpha_c} v^{\sigma,n} + \mathbb{P}\mathcal{T}_{\overline{n}}(j^{\sigma,n} \times B^{\sigma,n}),
			\\
			F^{\sigma,n}_2 &= \mathcal{T}_{\overline{n}}(\nabla \times (v^{\sigma,n} \times B^{\sigma,n})) - \frac{1}{\sigma} (-\Delta)^{\beta_c} B^{\sigma,n} - \frac{\kappa}{\sigma} \mathcal{T}_{\overline{n}} (\nabla \times (j^{\sigma,n} \times B^{\sigma,n})).
		\end{align*}
		We then define the following functional space and mapping for $s > \frac{7}{4}$
		\begin{equation*}
			F^{\sigma,n} : Z^s_n := H^s_{n,0} \times (H^s_{n,0} \cap \dot{H}^{-1}) \to Z^s_n
			\qquad \text{with} \qquad \Gamma^{\sigma,n} \mapsto F^{\sigma,n}(\Gamma^{\sigma,n}).
		\end{equation*}
		In $Z^s_n$, we use the following norm 
		\begin{equation*}
			\|(f_1,f_2)\|^2_{Z^s_n} := \|(f_1,f_2)\|^2_{H^s} + \|f_2\|^2_{\dot{H}^{-1}}.
		\end{equation*}
		Therefore, it can be check that $F^{\sigma,n}$ is well-defined and locally Lipschitz on $Z^s_n$. Thus, there exists a unique solution $\Gamma^{\sigma,n} \in C^1([0,T^n_*);Z^s_n)$ for some $T^n_* > 0$ with the following property if $T^n_* < \infty$ then
		\begin{equation*}
			\lim_{t \to T^n_*} \left(\|(v^{\sigma,n},B^{\sigma,n})(t)\|^2_{H^s} + \|B^{\sigma,n}(t)\|^2_{\dot{H}^{-1}}\right) = \infty.
		\end{equation*}
		
		\textbf{Step 2: The $H^s$ and $\dot{H}^{-1}$ estimates.} Assume that $T^n_* < \infty$. It is standard that the energy estimate of \eqref{app-C-HMHD} is given by
		\begin{equation} \label{E_sigma-n}
			\frac{1}{2} \frac{d}{dt}\|\Gamma^{\sigma,n}\|^2_{L^2} + \nu \|v^{\sigma,n}\|^2_{\dot{H}^{\alpha_c}} + \frac{1}{\sigma} \|B^{\sigma,n}\|^2_{\dot{H}^{\beta_c}} = 0.
		\end{equation}
		We now focus on the $\dot{H}^{-1}$ estimate as follows
		\begin{equation} \label{H-1_HMHD_1}
			\frac{1}{2} \frac{d}{dt} \|B^{\sigma,n}\|^2_{\dot{H}^{-1}} + \frac{1}{\sigma} \|B^{\sigma,n}\|^2_{\dot{H}^{\beta_c - 1}} =: I^{\sigma,n}_{1,-1} + I^{\sigma,n}_{2,-1}.
		\end{equation}
			\\
		Similar to Step 1b in the proof of Thereom \ref{theo-MH}, \eqref{E_sigma-n}-\eqref{H-1_HMHD_1} leads to for $t \in (0,T^n_*)$
		\begin{equation} \label{H-1-HMHD_2}
			\|B^{\sigma,n}(t)\|^2_{\dot{H}^{-1}} +  \int^t_0 \|B^{\sigma,n}\|^2_{\dot{H}^{\beta_c - 1}} \,d\tau  \leq C(T^n_*,\kappa,\nu,\sigma) \|\Gamma^\sigma_0\|^2_{L^2}.
		\end{equation}
		We now focus on the $H^s$ estimate as follows. It can be seen from \eqref{app-C-HMHD} that 
		\begin{equation} \label{HMHD-Hs_1}
			\frac{1}{2} \frac{d}{dt} \|(v^{\sigma,n},B^{\sigma,n})\|^2_{\dot{H}^s} + \nu \|v^{\sigma,n}\|^2_{\dot{H}^{s+\alpha_c}} + \frac{1}{\sigma} \|B^{\sigma,n}\|^2_{\dot{H}^{s+\beta_c}}
			=: \sum^{4}_{i=1} I^{\sigma,n}_i.
		\end{equation}
		Thus, from direct calculations and the following Agmon-type inequality in three dimensions 
		\begin{equation*}
			\|f\|_{L^\infty} \leq C(s) \|f\|^{1-\frac{3}{2s_0}}_{L^2} \|f\|^\frac{3}{2s_0}_{\dot{H}^{s_0}} \qquad \text{for} \quad  s_0 > \frac{3}{2},
		\end{equation*}
		it follows from \eqref{HMHD-Hs_1} 
		that for $t \in (0,T^n_*)$
		\begin{equation} \label{HMHD-Hs_2}
			\|(v^{\sigma,n},B^{\sigma,n})\|^2_{H^s} + \int^t_0 \nu \|v^{\sigma,n}\|^2_{H^{s+\alpha_c}} + \frac{1}{\sigma} \|B^{\sigma,n}\|^2_{H^{s+\beta_c}} \,d\tau \leq C(T^n_*,\kappa,\nu,\sigma,s,\Gamma^\sigma_0).
		\end{equation}
		Thus, we find from \eqref{H-1-HMHD_2} and \eqref{HMHD-Hs_2} that $T^n_* = \infty$. In addition, by repeating the above computations, we also obtain \eqref{H-1-HMHD_2} and \eqref{HMHD-Hs_2} with replacing $T^n_*$ by any $T \in (0,\infty)$.
		
		\textbf{Step 3: Passing to the limit and uniqueness.} This part is standard, where by using \eqref{HMHD-Hs_2} (for any $T \in (0,\infty)$ instead of $T^n_*$), we can prove that $(v^{\sigma,n},B^{\sigma,n})$ is a Cauchy sequence in $L^\infty(0,T;H^r)$ and in $L^2(0,T;H^{r+\alpha_c} \times H^{r+\beta_c})$ for $r \in (\frac{7}{4},s)$, for example see \cite{Fefferman-McCormick-Robinson-Rodrigo_2014,KLN_2024}, in which the limiting solution satisfies the following limiting system for $t > 0$
		\begin{equation*}
			\left\{
			\begin{aligned}
				\partial_t v^\sigma + \mathbb{P}(v^{\sigma} \cdot \nabla v^{\sigma}) &= - \nu (-\Delta)^{\alpha_c} v^{\sigma} + \mathbb{P}(j^{\sigma} \times B^{\sigma}), &\text{in} \quad L^2(0,t;H^{r-\alpha_c}),
				\\
				\partial_t B^\sigma - (\nabla \times (v^{\sigma} \times B^{\sigma})) &= - \frac{1}{\sigma} (-\Delta)^{\beta_c} B^{\sigma} - \frac{\kappa}{\sigma} (\nabla \times (j^{\sigma} \times B^{\sigma})), &\text{in} \quad L^2(0,t;H^{r-\beta_c}),
				\\
				\textnormal{div}\, v^\sigma &= \textnormal{div}\, B^\sigma = 0, &\text{in} \quad L^2(0,t;H^{r-1}),
				\\
				(v^\sigma,B^\sigma)_{|_{t=0}} &= (v^\sigma_0,B^\sigma_0), &\text{in} \quad H^s,
			\end{aligned}
			\right.
		\end{equation*}
		and 
		\begin{align} \label{Gamma_sigma_HMHD_1}
			&\|\Gamma^{\sigma}(t)\|^2_{L^2} + \int^t_0 \nu \|v^{\sigma}\|^2_{\dot{H}^{\alpha_c}} + \frac{1}{\sigma} \|B^{\sigma}\|^2_{\dot{H}^{\beta_c}} \,d\tau \leq \|\Gamma^\sigma_0\|^2_{L^2},
			\\ \label{Gamma_sigma_HMHD_2}
			&\|B^{\sigma}(t)\|^2_{\dot{H}^{-1}} +  \int^t_0 \|B^{\sigma}\|^2_{\dot{H}^{\beta_c - 1}} \,d\tau  \leq C(T,\kappa,\nu,\sigma) \|\Gamma^\sigma_0\|^2_{L^2},
			\\ \label{Gamma_sigma_HMHD_3}
			&\|(v^{\sigma},B^{\sigma})(t)\|^2_{\dot{H}^s} + \int^t_0  \|v^{\sigma}\|^2_{\dot{H}^{s+\alpha_c}} + \|B^{\sigma}\|^2_{\dot{H}^{s+\beta_c}} \,d\tau \leq C(T,\kappa,\nu,\sigma,s,\Gamma^\sigma_0).
		\end{align}
		Furthermore, the uniqueness can be obtained in the usual way. We omit further details.
		
		\textbf{Step 4: Magnetic helicity conservation.} Similar to the proof of Theorem \ref{theo-MH}, we find from \eqref{H-1-HMHD_2} (without $n$) that for $t > 0$
		\begin{align} \label{AB_csigma_HMHD}
			\mathcal{H}^{\sigma}(t) &:= \int_{\mathbb{R}^3} A^{\sigma}(t) \cdot B^{\sigma}(t)  \,dx 
			\leq C \|B^{\sigma}(t)\|^2_{\dot{H}^{-\frac{1}{2}}}
			\leq  C(t,\kappa,\nu,\sigma,s,\Gamma^{\sigma}_0).
		\end{align}
		In addition,  \eqref{Gamma_sigma_HMHD_1}-\eqref{AB_csigma_HMHD}, Tonelli and Fubini Theorems (see \cite{Brezis_2011}, to compute the weak derivative in the first line below) and the limiting system yield for $t > 0$
		\begin{align} \label{AB_csigma_1_HMHD} 
			\frac{d}{dt} \mathcal{H}^{\sigma}(t) 
			&= \int_{\mathbb{R}^3} \partial_t A^{\sigma} \cdot B^{\sigma} + A^{\sigma} \cdot \partial_t B^{\sigma} \,dx 
			= -\frac{2}{\sigma} \int_{\mathbb{R}^3}  A^\sigma \cdot  (-\Delta)^{\beta_c} B^\sigma \,dx.
		\end{align}
		Integrating in time, we find from \eqref{Gamma_sigma_HMHD_1} and \eqref{AB_csigma_HMHD}-\eqref{AB_csigma_1_HMHD} that for $\tau > 0$ (see \cite{Brezis_2011})
		\begin{align*}
			\left|\mathcal{H}^{\sigma}(\tau) - \mathcal{H}^{\sigma}(0)\right| &= \left|\int^{\tau}_0 \frac{d}{dt} \mathcal{H}^{\sigma}(t) \,dt\right| 
			\leq 2 (\tau+1) \sigma^{-\frac{2}{7}} \|\Gamma^\sigma_0\|^2_{L^2}.
		\end{align*}
		Thus, as in the proof Theorem \ref{theo-MH}, the proof is complete by using \eqref{AB_csigma_HMHD}.
	\end{proof}

	%
	\section{Proof of Theorem \ref{theo_stability}} \label{sec:stability}
	%
	
	In this section, we focus on giving the proof of Theorem \ref{theo_stability}.
	
	\begin{proof}[Proof of Theorem \ref{theo_stability}-$(i)$] The proof will be divided into several steps as follows.
		
		\textbf{Step 1: Approximate system, local and global existence.} Since $B^*$ is a constant vector in $\mathbb{R}^3$ then we can choose $cE^* - \nabla \pi^*_2  = \nabla \pi^*_1 = 0$. We will focus on the following approximate system to \eqref{NSM-GO*} in the case $\alpha = 0$ for each $n \in \mathbb{N}$
		\begin{equation} \label{NSM-GO*_n} 
			\frac{d}{dt} \Gamma^n = F^n(\Gamma^n) \qquad \text{and} \qquad \Gamma^n_{|_{t=0}} = \mathcal{T}_{\overline{n}}(\Gamma_0),
		\end{equation}
		where $\Gamma^n := (v^n,E^n,B^n)$ and $F^n := (F^n_1,F^n_2,F^n_3)$ with
		\begin{align*}
			\bar{j}^n &:= - \kappa\mathbb{P}\mathcal{T}_{\overline{n}}(\bar{j}^{n-1} \times (B^n + B^*)) + \sigma(cE^n + \mathbb{P}\mathcal{T}_{\overline{n}}(v^{n} \times (B^{n-1} + B^*))),
			\\
			F^n_1 &:= - \mathbb{P}\mathcal{T}_{\overline{n}}(v^{n-1} \cdot \nabla v^n) - \nu v^n + \mathbb{P}\mathcal{T}_{\overline{n}}(\bar{j}^n \times (B^{n-1} + B^*)),
			\\
			F^n_2 &:= c(\nabla \times B^n - \bar{j^n})
			\qquad \text{and} \qquad 
			F^n_3 := -c\nabla \times E^n.
		\end{align*}
		Similar to the previous parts, we aim to prove  the following claim by using an induction argument: For each $n \in \mathbb{N}$, \eqref{NSM-GO*_n} has a unique global solution $\Gamma^n \in C^1([0,\infty);Y^s_n)$, where $Y^s_n$ will be defined below, such that for $t \in (0,\infty)$, $s'' \in [0,s]$, $s' \in [1,s]$ and $C = C(B^*,d,\kappa,\nu,\sigma,s)$,  
		\begin{align} \label{D-Gamma-n_1}
			\|\Gamma^n(t)\|^2_{H^s} + \int^t_0  \|v^n\|^2_{H^s} +  \|\bar{j}^n\|^2_{H^s} \,d\tau \leq C\|\Gamma_0\|^2_{H^s},
			\\ \label{D-Gamma-n_1_b}
			\int^t_0  \|E^n\|^2_{\dot{H}^{s''}} + \|B^n\|^2_{\dot{H}^{s'}}  \,d\tau \leq C(c)C\|\Gamma_0\|^2_{H^s}. 
		\end{align}
		Thus, as in the proof of Theorem \ref{theo_global_small}, we define for $s > \frac{d}{2} + 1$ the following functional space and mapping
		\begin{equation*}
			Y^s_n := H^s_{n,0} \times H^s_{n,0} \times H^s_{n,0}, 
			\qquad F^n : Y^s_n \to Y^s_n \quad \text{with}\quad  \Gamma^n \mapsto F^n(\Gamma^n). 
		\end{equation*}
		
		\textbf{Step 1a: The case $n = 1$.} In this case,
		as in the previous cases, we have the existence and uniqueness of solutions $\Gamma^1 \in C^1([0,T^1_*);Y^s_1)$ for some $T^1_* > 0$ with  the following property that if $T^1_* < \infty$ then
		\begin{equation*}
			\lim_{t \to T^1_*} \|\Gamma^1(t)\|^2_{H^s} = \infty.
 		\end{equation*}
		In addition, the $H^s$ estimate of \eqref{NSM-GO*_n} with $n = 1$ is given by for $t \in (0,T^1_*)$
		\begin{equation} \label{D-Gamma-1}
			\|\Gamma^1(t)\|^2_{H^s} + \int^t_0 \nu \|v^1\|^2_{H^s} + \frac{1}{\sigma} \|\bar{j}^1\|^2_{H^s} \,d\tau \leq \|\Gamma_0\|^2_{H^s},
		\end{equation}
		which implies that $T^1_* = \infty$. Similar to Step 1a in the proof Theorem \ref{theo_global_small}, we find that for $s' \in [1,s]$ and $\tau > 0$
		\begin{align*}
			\|B^1\|^2_{L^2(0,t;\dot{H}^{s'})}  
			&= \frac{1}{c} \int^{\tau}_0 \frac{d}{dt} \int_{\mathbb{R}^d} \Lambda^{s'-1} E^1 \cdot \Lambda^{s'-1}(\nabla \times B^1) \,dxdt +  \|E^1\|^2_{L^2(0,t;\dot{H}^{s'})} 
			\\
			&\quad + \int^{\tau}_0 \int_{\mathbb{R}^d} \Lambda^{s'-1} \bar{j}^1 \cdot \Lambda^{s'-1}(\nabla \times B^1) \,dxdt.
		\end{align*}
		In addition, for $s'' \in [0,s]$ and $t \in (0,\infty)$
		\begin{align*}
			2c^2\sigma \|E^1\|^2_{L^2(0,t;\dot{H}^{s''})} &= - 2c \int^t_0\int_{\mathbb{R}^d} \Lambda^{s''} (v^1 \times B^*) \cdot \Lambda^{s''} E^1\,dxd\tau 
			+ \|(E^1,B^1)(0)\|^2_{\dot{H}^{s''}}.
		\end{align*}
		By using \eqref{D-Gamma-1} and the two previous equalities, it yields for $t \in (0,\infty)$ and $C = C(B^*,d,\kappa,\nu,\sigma,s)$
		\begin{align} \label{D-Gamma-1_2}
			\|\Gamma^1(t)\|^2_{H^s} + \int^t_0 \|v^1\|^2_{H^s} + 
			\|\bar{j}^1\|^2_{H^s} \,d\tau \leq C\|\Gamma_0\|^2_{H^s},
			\\ \label{D-Gamma-1_2_b}
			\int^t_0 \|E^1\|^2_{\dot{H}^{s''}} + \|B^1\|^2_{\dot{H}^{s'}}  \,d\tau \leq C(c)C\|\Gamma_0\|^2_{H^s},
		\end{align}
		which is \eqref{D-Gamma-n_1}-\eqref{D-Gamma-n_1_b} with $n = 1$.
		
		\textbf{Step 1b: The case $n = 2$.} In this case,
		we have the existence and uniqueness of solutions $\Gamma^2 \in C^1([0,T^2_*);Y^s_2)$ for some $T^2_* > 0$. Moreover, the $H^s$ estimate for \eqref{NSM-GO*_n} with $n = 2$ is given by
		\begin{equation*}
			\frac{1}{2}\frac{d}{dt}\|\Gamma^2\|^2_{H^s} + \nu \|v^2\|^2_{H^s} + \frac{1}{\sigma}\|\bar{j}^2\|^2_{H^s} =: \sum^{5}_{i=1} I^2_i,
		\end{equation*}
		where for some $\epsilon \in (0,1)$, since $s > \frac{d}{2} + 1$
		\begin{align*}
			I^2_1 &= \int_{\mathbb{R}^d} [\Lambda^s(v^1 \cdot \nabla v^2) - v^1 \cdot \nabla \Lambda^s v^2]  \cdot \Lambda^s v^2\,dx
			\leq \epsilon \nu \|v^2\|^2_{H^s} + C(d,\epsilon,s) \|v^1\|^2_{H^s} \|v^2\|^2_{H^s};
			\\
			I^2_2 &= \int_{\mathbb{R}^d} \Lambda^s(\bar{j}^2 \times (B^1 + B^*)) \cdot \Lambda^s v^2\,dx =: I^2_{2,1} + I^2_{2,2},
			\\
			I^2_{2,1} &= \int_{\mathbb{R}^d} \Lambda^s(\bar{j}^2 \times B^1) \cdot \Lambda^s v^2 \,dx
			\leq \epsilon \nu \|v^2\|^2_{H^s} + C(d,\epsilon,s)\|\bar{j}^2\|^2_{H^s} \|B^1\|^2_{H^s};
			\\
			I^2_{2,2} &= \int_{\mathbb{R}^d} (\Lambda^s \bar{j}^2 \times B^*) \cdot \Lambda^s v^2 \,dx;
			\\
			I^2_3 &= -\frac{\kappa}{\sigma}\int_{\mathbb{R}^d} \Lambda^s \bar{j}^2 \cdot \Lambda^s(\bar{j}^1 \times (B^2 + B^*))\,dx =: I^2_{3,1} + I^2_{3,2},
			\\
			I^2_{3,1} &= -\frac{\kappa}{\sigma}\int_{\mathbb{R}^d} \Lambda^s \bar{j}^2 \cdot \Lambda^s(\bar{j}^1 \times B^2)\,dx \leq \frac{\epsilon}{\sigma} \|\bar{j}^2\|^2_{H^s} + C(d,\epsilon,\sigma,s) \|\bar{j}^1\|^2_{H^s} \|B^2\|^2_{H^s};
			\\
			I^2_{3,2} &= -\frac{\kappa}{\sigma}\int_{\mathbb{R}^d} \Lambda^s \bar{j}^2 \cdot (\Lambda^s \bar{j}^1 \times B^*)\,dx \leq \frac{\epsilon}{\sigma}\|\bar{j}^2\|^2_{H^s} + C(d,\epsilon,\sigma,s)\|\bar{j}^1\|^2_{H^s}\|B^*\|^2_{L^\infty};
			\\
			I^2_4 &= -\frac{\kappa}{\sigma} \int_{\mathbb{R}^d} \bar{j}^2 \cdot (\bar{j}^1 \times (B^2 + B^*))\,dx := I^2_{4,1} + I^2_{4,2},
			\\
			I^2_{4,1} &= -\frac{\kappa}{\sigma} \int_{\mathbb{R}^d} \bar{j}^2 \cdot (\bar{j}^1 \times B^2)\,dx \leq \frac{\epsilon}{\sigma} \|\bar{j}^2\|^2_{H^s} + C(d,\epsilon,\sigma,s)\|\bar{j}^1\|^2_{H^s} \|B^2\|^2_{H^s};
			\\
			I^2_{4,2} &= -\frac{\kappa}{\sigma} \int_{\mathbb{R}^d} \bar{j}^2 \cdot (\bar{j}^1 \times B^*)\,dx \leq \frac{\epsilon}{\sigma} \|\bar{j}^2\|^2_{L^2} + C(d,\epsilon,\sigma,s)\|\bar{j}^1\|^2_{L^2}\|B^*\|^2_{L^\infty};
			\\
			I^2_5 &= \int_{\mathbb{R}^d} \Lambda^s \bar{j}^2 \cdot \Lambda^s (v^2 \times (B^1 + B^*))\,dx =: I^2_{5,1} + I^2_{5,2},
			\\
			I^2_{5,1} &= \int_{\mathbb{R}^d} \Lambda^s \bar{j}^2 \cdot \Lambda^s (v^2 \times B^1) \,dx \leq \frac{\epsilon}{\sigma} \|\bar{j}\|^2_{H^s} + C(d,\epsilon,\sigma,s)\|v^2\|^2_{H^s}\|B^1\|^2_{H^s};
			\\
			I^2_{5,2} &= \int_{\mathbb{R}^d} \Lambda^s \bar{j}^2 \cdot (\Lambda^s v^2 \times B^*) \,dx = - I^2_{2,2},
		\end{align*}
		here we used the following homogeneous Kato-Ponce type inequality (see \cite{Grafakos-Oh_2014})
		for $1 < p_i,q_i \leq \infty$, $i \in \{1,2\}$, $s_0 \in \mathbb{R}, s_0 > 0$ and  $\frac{1}{p_i} + \frac{1}{q_i} = \frac{1}{2}$
		\begin{equation*} 
			\|\Lambda^{s_0}(fg)\|_{L^2(\mathbb{R}^d)} \leq C(d,s_0,p_i,q_i)\left(\|\Lambda^{s_0} f\|_{L^{p_1}(\mathbb{R}^d)}\|g\|_{L^{q_1}(\mathbb{R}^d)} + \|f\|_{L^{p_2}(\mathbb{R}^d)}\|\Lambda^{s_0} g\|_{L^{q_2}(\mathbb{R}^d)} \right).
		\end{equation*}
		Therefore, by choosing $\epsilon = \frac{1}{10}$, it follows that for $C = C(d,\kappa,\nu,\sigma,s)$
		\begin{align*}
			\frac{d}{dt}\|\Gamma^2\|^2_{H^s} + \nu \|v^2\|^2_{H^s} + \frac{1}{\sigma} \|\bar{j}^2\|^2_{H^s} 
			&\leq C\|(v^1,B^1)\|^2_{H^s} \left(\|v^2\|^2_{H^s} + \|\bar{j}^2\|^2_{H^s}\right) 
			\\
			&\quad + C\left(\|B^*\|^2_{L^\infty} + \|B^2\|^2_{H^s}\right)\|\bar{j}^1\|^2_{H^s},
 		\end{align*}
 		which by integrating in time and using \eqref{D-Gamma-1_2} under assuming $C\|\Gamma_0\|^2_{H^s} \leq 1$ for some big enough positive constant $C = C(d,\kappa,\nu,\sigma,s)$, leads to for $t \in (0,T^2_*)$ and $C = (B^*,d,\kappa,\nu,\sigma,s)$ 
 		\begin{equation} \label{D-Gamma-2}
 			\esssup_{\tau \in (0,t)} \|\Gamma^2(\tau)\|^2_{H^s} + \int^t_0 \nu \|v^2\|^2_{H^s} + \frac{1}{\sigma}\|\bar{j}^2\|^2_{H^s} \,d\tau 
 			\leq C\|\Gamma_0\|^2_{H^s},
 		\end{equation}
 		which yields $T^2_* = \infty$. In addition, as in the previous step, by using \eqref{D-Gamma-2}, we find that for $t \in (0,\infty)$, $C = C(B^*,d,\kappa,\nu,\sigma,s)$, $s'' \in [0,s]$ and $s' \in [1,s]$
 		\begin{equation*}
 			\|E^2\|^2_{L^2(0,t;\dot{H}^{s''})} + \|B^2\|^2_{L^2(0,t;\dot{H}^{s'})} \leq  C(c)C\|\Gamma_0\|^2_{H^s},
 		\end{equation*}
 		which further implies \eqref{D-Gamma-n_1} and \eqref{D-Gamma-n_1_b} with $n = 2$.

 		\textbf{Step 1c: The case $n > 2$ with an induction argument.} In this case, 
 		we have the existence an uniqueness of solutions $\Gamma^k \in C^1([0,T^k_*);Y^s_k)$ for some $T^k_* > 0$. In Step 1b, we proved that \eqref{D-Gamma-n_1} holding for $n = 2$. Now, we assume that \eqref{D-Gamma-n_1} also holds for $n = k > 2$, i.e., \eqref{NSM-GO*_n} with $n = k$ has a global solution $\Gamma^k$ such that for $t \in (0,\infty)$, $s'' \in [0,s]$, $s' \in [1,s]$ and  $C = C(B^*,d,\kappa,\nu,\sigma,s)$
 		\begin{align} \label{D-Gamma-k}
 			\|\Gamma^k(t)\|^2_{H^s} + \int^t_0  \|v^k\|^2_{H^s} + \|\bar{j}^k\|^2_{H^s} \,d\tau \leq C\|\Gamma_0\|^2_{H^s},
 			\\ \label{D-Gamma-k_b}
 			\int^t_0  \|E^k\|^2_{\dot{H}^{s''}} + \|B^k\|^2_{\dot{H}^{s'}}  \,d\tau \leq C(c)C\|\Gamma_0\|^2_{H^s}. 
 		\end{align}
 		Thus, it remains to prove a similar result at the next level, i.e., \eqref{D-Gamma-n_1} with $n = k+1$.
 		Similarly, we obtain the existence and uniqueness of solutions $\Gamma^{k+1} \in C^1([0,T^{k+1}_*);Y^s_{k+1})$ for some $T^{k+1}_* > 0$. By repeating Step 1b with replacing \eqref{NSM-GO*_n} with $n = 2$, $(\Gamma^2,\bar{j}^2,\Gamma^1,\bar{j}^1)$ by \eqref{NSM-GO*_n} with $n = k+1$ and $(\Gamma^{k+1},\bar{j}^{k+1},\Gamma^k,\bar{j}^k)$, respectively, and with further assuming the right-hand side of \eqref{D-Gamma-k} to be small, we find that \eqref{D-Gamma-n_1} also holds for $n = k+1$. We omit further details. 
 		
 		\textbf{Step 2: Cauchy sequence, passing to the limit as $n \to \infty$ and uniqueness.} This step can be done almost in the same way as that of Step 2a in the proof of Theorem \ref{theo_global_small}. More precisely, we will obtain that $\Gamma^n$ and $(v^n,E^n,E^n,B^n,\bar{j}^n)$ are Cauchy sequences in $L^\infty(0,\infty;H^r)$ and $L^2(0,\infty;H^r \times \dot{H}^1 \times L^2 \times \dot{H}^1 \times H^r)$ for $r \in (\frac{d}{2},s)$, respectively. In fact, this case is easier than the previous one since we have a damping velocity term. We now focus on passing to the limit as $n \to \infty$ and the uniqueness of solutions, where the proofs are similar to those of Step 2b and Step 3 in the proof of Theorem \ref{theo_global_small}. More precisely, the limiting system will be given as follows 
 		\begin{equation} \label{D-NSM-GO-limit}
 			\left\{
 			\begin{aligned}
 				\partial_t v + \mathbb{P}(v \cdot \nabla v) &= -\nu v + \mathbb{P}(\bar{j} \times (B + B^*)), &\text{in} \quad L^2(0,\infty;H^{r-1}),
 				\\
 				\frac{1}{c} \partial_t E - \nabla \times B &= - \bar{j}, &\text{in} \quad L^2(0,\infty;H^{r-1}),
 				\\
 				\frac{1}{c} \partial_t B + \nabla \times E &= 0, &\text{in} \quad L^2(0,\infty;H^{r-1}),
 				\\
 				\bar{j} + \kappa \mathbb{P}(\bar{j} \times (B + B^*)) &= \sigma(cE + \mathbb{P}(v \times (B + B^*))) &\text{in} \quad L^2(0,\infty;H^{r}),
 				\\
 				\textnormal{div}\, v = \textnormal{div}\, E &= \textnormal{div}\, B = \textnormal{div}\, \bar{j} = 0, &\text{in} \quad L^2(0,\infty;H^{r-1}),
 				\\
 				(v,E,B)_{|_{t=0}} &= (v_0,E_0,B_0) &\text{in} \quad H^s.
 			\end{aligned}
			\right.
 		\end{equation}
 		In addition, for $\Gamma := (v,E,B)$, $t \in (0,\infty)$, $C = C(B^*,d,\kappa,\nu,\sigma,s)$, $s'' \in [0,s]$ and $s' \in [1,s]$
 		\begin{align} \label{D-Gamma}
 			\|\Gamma(t)\|^2_{H^s} + \int^t_0 \|v\|^2_{H^s} + 
 			\|\bar{j}\|^2_{H^s} \,d\tau \leq C\|\Gamma_0\|^2_{H^s},
 			\\
 			\label{D-Gamma_b}
 			\int^t_0 \|E\|^2_{\dot{H}^{s''}} + \|B\|^2_{\dot{H}^{s'}} \,d\tau \leq C(c)C\|\Gamma_0\|^2_{H^s}.
 		\end{align}
 		Moreover, the pressures can also be recovered as previously. We omit further details and end the proof of this part.
	\end{proof}
	
	We now continue with the proof of the second part as follows.
	
	\begin{proof}[Proof of Theorem \ref{theo_stability}-$(ii)$] We first focus on obtaining the estimate of $(v,E)$ in $L^2(0,\infty;L^2)$. 
		It can be seen that for $t \in (0,\infty)$
		\begin{align*}
			c^2\int^t_0 \|E\|^2_{L^2} \,d\tau 
			&\leq \frac{c^2}{2} \int^t_0 \|E\|^2_{L^2} \,d\tau +  C(B^*,d,\kappa,\sigma,s) \int^t_0 \left(1 + \|B\|^2_{H^s}\right)\left(\|\bar{j}\|^2_{L^2} + \|v\|^2_{L^2}\right) \,d\tau.
		\end{align*}
		Thus, it follows from \eqref{D-Gamma} that for $c \geq 1$
		\begin{equation} \label{vE_L2L2}
			\|(v,E)\|^2_{L^2(0,\infty;L^2)} \leq C(B^*,d,\kappa,\nu,\sigma,s) \|\Gamma_0\|^2_{H^s}.
		\end{equation}
		Furthermore, it can be seen from \eqref{D-NSM-GO-limit} that
		\begin{equation*}
			\frac{1}{2}\frac{d}{dt} \|(v,E)\|^2_{L^2} + \nu \|v\|^2_{L^2} + \frac{1}{\sigma} \|\bar{j}\|^2_{L^2} =: \sum^3_{i=1} I_i. 
		\end{equation*}
		Then, it follows from \eqref{D-Gamma} that for $0 \leq t' \leq t < \infty$
		\begin{equation} \label{vE_L2L_2_1}
			\|(v,E)(t)\|^2_{L^2} - \|(v,E)(t')\|^2_{L^2} \leq C(B^*,d,\kappa,\nu,\sigma,s)\|\Gamma_0\|^2_{H^s}(t-t').
		\end{equation}
		Thus, \eqref{vE_L2L2}-\eqref{vE_L2L_2_1} leads to $\|(v,E)(t)\|_{L^2} \to 0$ as $t \to \infty
		$ (see \cite[Lemma 2.1]{Lai-Wu-Zhong_2021}). In addition, by using interpolation inequality and \eqref{D-Gamma}, we find that for $s' \in [0,s)$ and $f \in \{v,E\}$
		\begin{equation*}
			\|f(t)\|_{H^{s'}} \leq C(d,s) \|f(t)\|^\frac{s-s'}{s}_{L^2}\|f(t)\|^\frac{s'}{s}_{H^s} \quad \to 0 \quad \text{as} \quad t \to \infty.
		\end{equation*}
		We now focus on the large-time behavior of $\bar{j}$. It can be seen that for $t \in (0,\infty)$ 
		\begin{align*}
			\|\bar{j}(t)\|^2_{L^2} 
			&\leq \frac{1}{4} \|\bar{j}(t)\|^2_{L^2} + C(c,B^*,d,\kappa,\sigma,s)(1 + \|B(t)\|^2_{H^s})\left(\|E(t)\|^2_{L^2} + \|v(t)\|^2_{L^2}\right),
		\end{align*}
		which yields $\|\bar{j}(t)\|_{L^2} \to 0$ as $t \to \infty$. Similarly, 
		\begin{align*}
			\|\bar{j}(t)\|^2_{\dot{H}^s} 
			&\leq \frac{1}{4}\|\bar{j}(t)\|^2_{\dot{H}^s} + C(\kappa)\|\bar{j}(t)\|^2_{H^s}\|B(t)\|^2_{H^s} 
			\\
			&\quad + C(c,B^*,d,\kappa,\sigma,s)(1 + \|B(t)\|^2_{H^s})\left(\|E(t)\|^2_{H^s} + \|v(t)\|^2_{H^s}\right).
		\end{align*}
		Thus, from the two previous estimates that by choosing the initial data to be small suitably, it follows that for $t \in (0,\infty)$
		\begin{equation*}
			\|\bar{j}(t)\|^2_{H^s} \leq C(c,B^*,d,\kappa,\nu,\sigma,s)(1+\|B(t)\|^2_{H^s})\left(\|v(t)\|^2_{H^s} + \|E(t)\|^2_{H^s}\right).
		\end{equation*}
		As done previously for $t \in (0,\infty)$ and $s' \in [0,s)$ by using the decay in time of $\|\bar{j}(t)\|_{L^2}$ and \eqref{D-Gamma}, it leads to
		\begin{equation*}
			\|\bar{j}(t)\|_{H^{s'}} \leq C(d,s) \|\bar{j}(t)\|^\frac{s-s'}{s}_{L^2} \|\bar{j}(t)\|^\frac{s'}{s}_{H^s} \quad \to 0 \quad \text{as} \quad t \to \infty.
		\end{equation*}
		We now focus on proving that $\|\nabla B(t)\|_{L^2} \to 0$ as $t \to \infty$. In oder to do that, we need to bound $\|(E,B)\|_{L^2_t\dot{H}^1_x}$. The idea here is similar to that of in the proof of Theorem \ref{theo_global_small}. It can be seen from the Maxwell system in \eqref{D-NSM-GO-limit} that by testing $-(\nabla \times B,\nabla \times E)$ for $\tau \in (0,\infty)$
		\begin{align*}
			\|\nabla \times B\|_{L^2(0,\tau;L^2)}  
			&\leq \frac{2}{c}\|\Gamma_0\|^2_{H^s} + \|\nabla \times E\|^2_{L^2(0,\tau;L^2)} + \frac{1}{2} \|\bar{j}\|^2_{L^2(0,\tau;L^2)} + \frac{1}{2}\|\nabla \times B\|^2_{L^2(0,\tau;L^2)}.
		\end{align*}
		It remains to bound $\|E\|_{L^2(0,t;\dot{H}^1)}$ as follows. By applying the cross product to the Maxwell equations in \eqref{D-NSM-GO-limit} and testing the result by $c(\nabla \times E,\nabla \times B)$, we find that 
		\begin{equation*}
			\frac{1}{2}\frac{d}{dt}\|(\nabla \times E,\nabla \times B)\|^2_{L^2}
			+ c^2\sigma \|\nabla \times E\|^2_{L^2} =: I ,
		\end{equation*}
		where
		\begin{align*}
			I 
			&\leq \frac{1}{2} c^2\sigma \|\nabla \times E\|^2_{L^2} + C(d,\kappa,\sigma) \left(\|\bar{j}\|_{L^2} + \|v\|^2_{L^2}\right)\left(\|B\|^2_{H^s} + \|B^*\|^2_{L^\infty}\right).
		\end{align*}
		Therefore, it follows from the two previous estimates that for $t \in (0,\infty)$
		\begin{equation*}
			\|(E,B)\|^2_{L^2(0,t;\dot{H}^1)} \leq C(B^*,d,\kappa,\nu,\sigma,s)\|\Gamma_0\|^2_{H^s},
		\end{equation*}
		In addition, it can also be seen from \eqref{D-NSM-GO-limit}-\eqref{D-Gamma} that for $0 \leq t' \leq t < \infty$
		\begin{equation*}
			\|(\nabla \times E,\nabla \times B)(t)\|^2_{L^2} - \|(\nabla \times E,\nabla \times B)(t')\|^2_{L^2} \leq C(B^*,d,\kappa,\nu,\sigma,s)\|\Gamma_0\|^2_{H^s} (t-t').
		\end{equation*}
		Thus, by using $\textnormal{div}\, E = \textnormal{div}\,B = 0$, we obtain from the two previous bouds $\|(\nabla E,\nabla B)(t)\|_{L^2} \to 0$ as $t \to \infty$. Similarly, by using interpolation inequality and \eqref{D-Gamma} we also obtain for $s' \in [1,s)$ and for $t \in (0,\infty)$
		\begin{equation*}
			\|B(t)\|_{\dot{H}^{s'}} = \|\nabla B(t)\|_{\dot{H}^{s'-1}} \leq C(d,s) \|\nabla B(t)\|^\frac{s-s'}{s-1}_{L^2} \|\nabla B(t)\|^\frac{s'}{s-1}_{\dot{H}^{s-1}} \quad \to 0 \quad \text{as} \quad t \to \infty.
		\end{equation*}
		Moreover, we also have the large-time behavior of $B$ in $L^\infty$ norm as follows for $t \in (0,\infty)$ and for some $s'' \in (\frac{d}{2},s)$ as $t \to \infty$
		\begin{align*}
			\|B(t)\|_{L^\infty} &\leq C(d,s)\|B(t)\|^{1- \frac{d}{2s''}}_{L^2} \|B(t)\|^\frac{d}{2s''}_{\dot{H}^{s''}}  \quad \to 0,
			\\
			\|B(t)\|_{L^p} &\leq \|B(t)\|^\frac{2}{p}_{L^2} \|B(t)\|^\frac{p-2}{p}_{L^\infty} \quad \to 0
			.
		\end{align*}
		As a consequence, we find that $\|B(t)\|_{L^q_{\textnormal{loc}}} \to 0$ as $t \to \infty$ for $q \in [1,\infty]$. Thus, the proof of this part is complete.
	\end{proof}
		
	\begin{proof}[Proof of Theorem \ref{theo_stability}-$(iii)$] Under the smallness assumption of the initial data, an application of Part $(i)$ gives us the existence of $\Gamma^c$ to \eqref{NSM-GO*} with similar properties as those of $\Gamma$. Therefore, the rest of the proof can be done as that of Theorem \ref{theo_global_small}-$(ii)$. We omit the details and end the proof.
	\end{proof}

	%
	\section*{Acknowledgements} 
	%
	
	K. Kang’s work is supported by RS-2024-00336346. J. Lee’s work is supported by NRF-2021R1A2C1092830. D. D. Nguyen’s work is supported by NRF-2019R1A2C1084685 and NRF-2021R1A2C1092830. 
	
	%
	\section*{Disclosure statement} 
	%
	
	The authors report there are no competing interests to declare.

	%
	\section{Appendix} \label{sec:app}
	\subsection{Appendix A: Proofs of \eqref{I_21_22}, \eqref{para_21_1}-\eqref{para_21_3} and \eqref{I_n+1n1_4}-\eqref{I_n+1n2_3}}
	%
	
	In this subsection, we aim to provide the following technical result in the usual form, which is slightly similar to those of in \cite[Proposition 3.5]{Germain-Ibrahim-Masmoudi_2014} and \cite[Proposition 3.2]{Ibrahim-Keraani_2012}, where the authors considered in the form of Chemin-Lerner spaces instead. Let us recall the following hybrid homogeneous Sobolev spaces $\dot{H}^{s_1,s_2}$ of tempered distributions $f$ with $s_1,s_2 \in \mathbb{R}$ such that 
	\begin{equation*}
		\|f\|^2_{\dot{H}^{s_1,s_2}} := \sum_{q \in \mathbb{Z}, q \leq 0} 2^{2qs_1} \|\dot{\Delta}_q f\|^2_{L^2} + \sum_{q \in \mathbb{Z}, q > 0} 2^{2qs_2} \|\dot{\Delta}_q f\|^2_{L^2} < \infty,
	\end{equation*}
	where the usual homogeneous dyadic blocks are defined as in \cite{Bahouri_2019,Bahouri-Chemin-Danchin_2011}. 
	In addition, it can be seen that $\dot{H}^{s,s} \equiv \dot{H}^s$ for $s \in \mathbb{R}$. We also recall the space $L^2_{\textnormal{log}}$, which is the set of tempered distributions $f$ satisfying
	\begin{equation*}
		\|f\|^2_{L^2_{\textnormal{log}}} := \sum_{q \in \mathbb{Z}, q \leq 0} \|\dot{\Delta}_q f\|^2_{L^2} + \sum_{q \in \mathbb{Z}, q > 0} q \|\dot{\Delta}_q f\|^2_{L^2} < \infty.
	\end{equation*}
	In the proof of Theorem \ref{theo_global_small}, we used the following inequalities several times. Thus, we will provide the proofs below for completeness.
	
	\begin{lemma} \label{lem_paraproduct} Let $s > 1$ and $T \in (0,\infty]$. There exist positive time-independent constants $C$ and $C(s)$ such that the following estimates hold in two dimensions
		\begin{align}
			\label{para1}
			\|\dot{S}_2(\dot{R}(f,g))\|_{L^1(0,T;L^2)} &\leq C\|f\|_{L^2(0,T;L^2)}\|g\|_{L^2(0,T;\dot{H}^{1,0})};
			\\
			\label{para2}
			\|\dot{T}_{f} g\|_{L^2(0,T;\dot{B}^{-1}_{2,1})} &\leq C\|f\|_{L^2(0,T;L^2)}\|g\|_{L^\infty(0,T;L^2)};
			\\
			\label{para3}
			\|\dot{T}_{g} f\|_{L^2(0,T;\dot{B}^{-1}_{2,1})} &\leq C\|f\|_{L^2(0,T;L^2)}\|g\|_{L^\infty(0,T;L^2)};
			\\ \label{para4}
			\|(\textnormal{Id}-\dot{S}_2)(\dot{R}(f,g))\|_{L^2(0,T;\dot{B}^{-1}_{2,1})} &\leq C\|f\|_{L^2(0,T;L^2_{\textnormal{log}})}\|g\|_{L^\infty(0,T;L^2_{\textnormal{log}})};
			\\ \label{para5}
			\|\dot{S}_2(\dot{R}(f,g))\|_{L^1(0,T;\dot{H}^{s-1})} &\leq C(s) \|f\|_{L^2(0,T;L^2)} \|g\|_{L^2(0,T;\dot{H}^s)};
			\\ \label{para6}
			\|(\textnormal{Id}-\dot{S}_2)(\dot{R}(f,g))\|_{L^2(0,T;\dot{B}^{s-1}_{2,1})} &\leq C(s)\|f\|_{L^2(0,T;L^2)} \|g\|_{L^\infty(0,T;\dot{H}^s)}.
		\end{align}
	\end{lemma}
	\begin{proof}[Proof of Lemma \ref{lem_paraproduct}] The proofs of \eqref{para1}-\eqref{para6} are given as follows.
		
		\begin{enumerate}
			\item[1.] We begin with the proof of \eqref{para1} as follows. 
			\begin{align*}
				I_1 &:= \|\dot{S}_2(\dot{R}(f,g))\|_{L^2} 
				\\
				&\leq C\sum_{j \in \mathbb{Z}, j > 0}  \|\dot{\Delta}_j f\|_{L^2} \left(\|\dot{\Delta}_{j-1} g\|_{L^2} + \|\dot{\Delta}_{j} g\|_{L^2}+ \|\dot{\Delta}_{j+1} g\|_{L^2}\right)
				\\
				&\quad + C\sum_{j \in \mathbb{Z}, j \leq 0}  2^j \|\dot{\Delta}_j f\|_{L^2} \left(\|\dot{\Delta}_{j-1} g\|_{L^2} + \|\dot{\Delta}_{j} g\|_{L^2}+ \|\dot{\Delta}_{j+1} g\|_{L^2}\right)
				\\
				&\leq C\|f\|_{L^2} \|g\|_{\dot{H}^{1,0}},
			\end{align*}
			here we also used the following Bernstein-type inequality (see \cite{Bahouri-Chemin-Danchin_2011}) for $q \in \mathbb{Z}$
			\begin{equation*}
				\|\dot{\Delta}_q f\|_{L^2} \leq C 2^q \|\dot{\Delta}_q f\|_{L^1};
			\end{equation*}
			the fact that $\text{supp}(\mathcal{F}(\dot{\Delta}_j f (\dot{\Delta}_{j-1} g + \dot{\Delta}_j g + \dot{\Delta}_{j+1} g))) \subseteq 2^j B_{24}$; the boundedness of $\dot{\Delta}_q$ from $L^1$ to $L^1$; and the following fact 
			\begin{equation*}
				\sum_{j \in \mathbb{Z}, j < 0, q \leq j+4} 2^q = C 2^j.
			\end{equation*} 
			Therefore, 
			the proof follows 
			by integrating in time of $I_1$. 
			
			\item[2.] We continue with the proof of \eqref{para2} as follows. Similar to the previous case, it follows from the definition of $\dot{T}_f g$ that $\dot{T}_f g = \sum_{j \in \mathbb{Z}} \dot{S}_{j-1} f \dot{\Delta}_j g$, which yields for each $q \in \mathbb{Z}$
			\begin{equation*}
				\|\dot{\Delta}_q (\dot{T}_f g)\|_{L^2} \leq \sum_{j \in \mathbb{Z}} \|\dot{\Delta}_q(\dot{S}_{j-1} f \dot{\Delta}_j g)\|_{L^2}
				\leq C\sum_{j \in \mathbb{Z}, |j-q| \leq 4} \|\dot{\Delta}_q(\dot{S}_{j-1} f \dot{\Delta}_j g)\|_{L^2},
			\end{equation*}
			where we used the fact that $\text{supp}(\mathcal{F}(\dot{S}_{j-1} f \dot{\Delta}_j g)) \subseteq 2^j \tilde{\mathcal{C}}$, where $\tilde{\mathcal{C}} := B_{\frac{2}{3}} + \mathcal{C}$ in the first step. Therefore,
			\begin{align*}
				I_2 &:= \|\dot{T}_f g\|_{L^2(0,T;\dot{B}^{-1}_{2,1})} 
				\\
				&\leq C\left(\int^T_0 \left(\sum_{q \in \mathbb{Z}} 2^{-q}\sum_{j \in \mathbb{Z}, |j-q| \leq 4} \sum_{k \in \mathbb{Z}, k \leq j-2}  2^{k}  \|\dot{\Delta}_k f\|_{L^2} \|\dot{\Delta}_j g\|_{L^2} \right)^2 \,d\tau\right)^\frac{1}{2},
			\end{align*}
			where we used the boundedness of $\dot{\Delta}_q$ from $L^2$ to $L^2$ and the following Bernstein-type inequality (see \cite{Bahouri-Chemin-Danchin_2011}) for $j \in \mathbb{Z}$, since $\text{supp}(\mathcal{F}(\dot{\Delta}_j)) \subseteq 2^j \mathcal{C}$
			\begin{equation*}
				\|\dot{\Delta}_j f\|_{L^\infty} \leq C 2^j \|\dot{\Delta}_j f\|_{L^2}.
			\end{equation*}
			Here, since the summation on $j$ is finite in terms of $q$. Then, we can consider for example the case that $j = q$. Other cases $j \in \{q \pm 1, ..., q \pm 4\}$ can be done in a similar way in which we will omit further details. Thus, if $j = q$ then the above computation can be continued as follows by using Cauchy–Schwarz and Young inequalities for series
			\begin{align*}
				I_2(j=q) 
				&\leq C\|f\|_{L^2(0,T;L^2)} \|g\|_{L^\infty(0,T;L^2)}.
			\end{align*}
			Therefore, the proof of this inequality is complete.
			
			\item[3.] By exchanging the role of $f$ and $g$, the proof of \eqref{para3} follows from that of \eqref{para2}. We omit the details.
			
			\item[4.] Next, we focus on the proof of \eqref{para4}. Similar to the proofs of \eqref{para1} and \eqref{para2}, we find that
			\begin{align*}
				I_3 &:= \|(\textnormal{Id}-\dot{S}_2)(\dot{R}(f,g))\|_{\dot{B}^{-1}_{2,1}} 
				\\
				&\leq C\sum_{q \in \mathbb{Z}, q \geq 2}  \sum_{j \in  \mathbb{Z}, j \geq q-4} \|\dot{\Delta}_j f\|_{L^2} \left(\|\dot{\Delta}_{j-1} g\|_{L^2} + \|\dot{\Delta}_{j} g\|_{L^2}+ \|\dot{\Delta}_{j+1} g\|_{L^2}\right)
				\\
				&\leq C\|f\|_{L^2_{\textnormal{log}}} \|g\|_{L^2_\textnormal{log}}.
			\end{align*}
			Thus, the proof follows after integrating in $L^2$ in time of $I_3$.
			
			\item[5.] Similar to the proof of \eqref{para1}, it follows that
			\begin{align*}
				I_4 &:= \|\dot{S}_2(\dot{R}(f,g))\|_{\dot{H}^{s-1}} 
				\\
				&\leq C(s) \sum_{j \in  \mathbb{Z}} \left(\sum_{q \in \mathbb{Z}, q \leq \min\{2,j+4\}} 2^{2qs} \|\dot{\Delta}_j f\|^2_{L^2} \left(\|\dot{\Delta}_{j-1}g\|^2_{L^2} + \|\dot{\Delta}_jg\|^2_{L^2} + \|\dot{\Delta}_{j-1}g\|^2_{L^2}\right)\right)^\frac{1}{2}
				\\
				&\leq C(s) \|j\|_{\dot{B}^0_{2,2}} \|g\|_{\dot{B}^{s}_{2,2}},
			\end{align*}
			here we also used Bernstein-type, Holder and Minkowski (for series, see \cite{Stein_1970}) inequalitites. Thus, integrating in time of $I_4$ implies \eqref{para5}.
			
			\item[6.] Finally, we prove \eqref{para6} as follows. Similar to the proof of \eqref{para4},
			we find that
			\begin{align*}
				I_5 &:= \|(\textnormal{Id}-\dot{S}_2)(\dot{R}(f,g))\|_{\dot{B}^{s-1}_{2,1}} 
				\\
				&\leq C\sum_{q \in \mathbb{Z}, q \geq 2} 2^{sq} \sum_{j \in  \mathbb{Z}, j \geq q-4} \|\dot{\Delta}_j f\|_{L^2} \left(\|\dot{\Delta}_{j-1} g\|_{L^2} + \|\dot{\Delta}_{j} g\|_{L^2}+ \|\dot{\Delta}_{j+1} g\|_{L^2}\right)
				\\
				&\leq C(s)\|f\|_{\dot{B}^0_{2,2}} \|g\|_{\dot{B}^s_{2,2}}.
			\end{align*}
			Thus, the proof follows after integrating in $L^2$ in time of $I_5$.
		\end{enumerate}
	\end{proof}
	
%
%

	%
	\subsection{Appendix B: Parabolic regularity}
	%
	
	Let us consider the following fractional heat equation for $\alpha \geq 0$ and $d \geq 1$
	\begin{equation} \label{F-heat}
		\partial_t w + \nu(-\Delta)^\alpha w = f \qquad \text{in}\quad  (0,T) \times \mathbb{R}^d \qquad \text{and} \qquad w_{|_{t= 0}} = w_0.
	\end{equation}
	It is well-known that the solution to \eqref{F-heat} can be represented by using Duhamel formula
	\begin{equation*}
		w(t) = \exp\{t\nu(-\Delta)^\alpha\} w_0 + \int^t_0 \exp\{(t-\tau)\nu(-\Delta)^\alpha\} f(\tau) \,d\tau \qquad \text{for} \quad t \in (0,T),
	\end{equation*}
	where we have been used the notation $\exp\{t \nu(-\Delta)^\alpha\} f := \mathcal{F}^{-1}(\exp\{-\nu t|\xi|^{2\alpha}\}\mathcal{F}(f)(\xi))$.
	In the sequel, we aim to recall the following result in which its proof mostly follows the ideas in \cite[Proposition 3.1]{Arsenio-Gallagher_2020}, where the authors focused  on the case $\alpha = 1$. We note that for a similar result in the form of Chemin-Lerner spaces, see \cite[Proposition 2]{Chae-Lee_2004}.
	
	\begin{lemma}[\textnormal{\cite[Proposition 7.1]{KLN_2024_2}}] \label{pro-F_heat} 
		Let $d \geq 1$ and $w$ be a solution to \eqref{F-heat} with $w_{|_{t= 0}} = w_0$, $\alpha \in [0,\infty)$, and $\nu \in (0,\infty)$. Assume that $\delta_0 \in \mathbb{R}$, $p \in [1,\infty]$, $1 < r \leq m < \infty$, $1 \leq q \leq m$, $T \in (0,\infty]$, $w_0 \in \dot{B}^{\delta_0 + 2\alpha}_{p,q}(\mathbb{R}^d))$ and $f \in L^r(0,T;\dot{B}^{\delta_0 + \frac{2\alpha}{r}}_{p,q}(\mathbb{R}^d)$. Then there are some positive constants $C_1 = C_1(\alpha,d,\delta_0,m,\nu,p,q,r)$ and $C_2 = C_2(\alpha,d,\delta_0,m,\nu,p,q)$ such that
		\begin{equation} \label{F_heat_0}
			\|w\|_{L^m(0,T;\dot{B}^{\delta_0 + 2\alpha + \frac{2\alpha}{m}}_{p,q}(\mathbb{R}^d))} \leq C_1 \|f\|_{L^r(0,T;\dot{B}^{\delta_0 + \frac{2\alpha}{r}}_{p,q}(\mathbb{R}^d))} + C_2 \|w_0\|_{\dot{B}^{\delta_0 + 2\alpha}_{p,q}(\mathbb{R}^d)}.
		\end{equation}
	\end{lemma}
	
	Recall that the proof of  Lemma \ref{pro-F_heat} is based on the following technical lemma in which its proof closely follows the ideas in \cite[Lemma 2.4]{Bahouri-Chemin-Danchin_2011}, \cite[Lemma 2.1]{Chemin_1999}, where the authors considered  the case $\alpha = 1$. See also \cite[Lemma 1]{Chae-Lee_2004} with a similar proof in the case $d = 3$ and $\alpha \geq 0$.
	
	\begin{lemma}[\textnormal{\cite[Lemma 7.3]{KLN_2024_2}}] \label{lem-F_heat}
		Let $d \geq 1$ and $\mathcal{C}(c_1,c_2)$ be an annulus with the smaller radius $c_1 > 0$ and the bigger radius $c_2 > 0$. There exist positive constants $C_3 = C_3(\alpha,c_1,c_2,d)$ and $C_4 = C_4(\alpha,c_1,d)$ such that for any $\alpha \in [0,\infty)$, $p \in [1,\infty]$ and any pair $(t, \lambda)$ of positive real numbers the following property holds. If $\textnormal{supp}(\mathcal{F}(u)) \subseteq \lambda \mathcal{C}$ then 
		\begin{equation} \label{F_heat_estimate}
			\|\exp\{t\nu(-\Delta)^\alpha\} u\|_{L^p(\mathbb{R}^d)} \leq C_4 \exp\{-C_3\nu t\lambda^{2\alpha}\}\|u\|_{L^p(\mathbb{R}^d)}.
		\end{equation}
	\end{lemma}

	\addcontentsline{toc}{section}{References}

	\end{document}